\def\isarxiv{1}

\ifdefined\isarxiv%
\documentclass[10pt]{article}
\usepackage[margin=1.5in]{geometry}
\usepackage{times}
\usepackage{hyperref}
\usepackage{natbib}
\else
\documentclass{article}
\usepackage{multicol,enumitem}
\usepackage{hyperref}
\usepackage[nonatbib]{template/neurips_2024}
\usepackage[noend]{algpseudocode} 
\usepackage{algorithm} 
\fi

\usepackage{hyperref}       
\usepackage{url}   

\usepackage{amsmath}
\usepackage{amssymb}
\usepackage{amsthm}

\usepackage{booktabs}       
\usepackage{amsfonts}       
\usepackage{nicefrac}       
\usepackage{microtype}      
\usepackage{xcolor}         

\usepackage{amsmath,amsthm,amssymb}
\usepackage{graphicx}
\usepackage{cleveref}
\usepackage{comment}
\usepackage{bbm}
\usepackage{textcomp}
\usepackage{wrapfig}
\usepackage{float}

\usepackage{url}            
\usepackage{booktabs}       
\usepackage{threeparttable}  
\usepackage{amsfonts}       
\usepackage{nicefrac}       
\usepackage{microtype}      
\usepackage{multirow}
\usepackage{algorithm}
\usepackage{amsmath}
\usepackage{amsthm}
\usepackage{amssymb}
\usepackage{color}
\usepackage[english]{babel}
\usepackage{graphicx}
\usepackage{grffile}
\usepackage{wrapfig,epsfig}
\usepackage{url}
\usepackage{color}
\usepackage{algpseudocode}
\usepackage[T1]{fontenc}
\usepackage{bbm}
\usepackage{comment}
\usepackage{tikz}

\usepackage{pgfplots}
\pgfplotsset{compat=1.8}
\tikzset{elegant/.style={smooth,thick,samples=500,magenta}}

\theoremstyle{plain}
\newtheorem{theorem}{Theorem}[section]
\newtheorem{lemma}{Lemma}[section]
\newtheorem{remark}{Remark}[section]
\newtheorem{corollary}{Corollary}[section]
\newtheorem{proposition}{Proposition}[section]
\theoremstyle{definition}
\newtheorem{definition}{Definition}[section]

\newtheorem{assumption}{Assumption}[section]

\newtheorem{example}{Example}[section]

\crefname{assumption}{Assumption}{Assumptions}

\newcommand{\cov}{\mathrm{Cov}}

\newcommand{\Var}{\mathrm{Var}}

\newcommand{\R}{\mathbb{R}}

\newcommand{\keywords}[1]{\par\noindent
{\bfseries Keywords: } #1}

\definecolor{b2}{RGB}{51,153,255}
\definecolor{myGreen}{RGB}{80,180,0}
\definecolor{myGold}{rgb}{0.75,0.6,0.12}
\definecolor{myBlue}{RGB}{100,100,255}

\title{HAL-MLE Log-Splines Density Estimation (Part I: Univariate)}

\ifdefined\isarxiv%
\author{
  Yilong Hou \\
  Department of Biostatistics\\
  University of California, Berkeley, CA \\
  \texttt{yilong\_hou@berkeley.edu} \\
  \and
  Zhengpu Zhao \\
  Department of Statistics\\
  University of California, Berkeley, CA \\
  \texttt{zhengpu@berkeley.edu} \\
  \and
  Yi Li \\
  Department of Biostatistics\\
  University of California, Berkeley, CA \\
  \texttt{yi\_li@berkeley.edu} \\
  \and
  Mark van der Laan \\
  Department of Biostatistics\\
  University of California, Berkeley, CA \\
  \texttt{laan@berkeley.edu} \\
}
\date{March 3, 2026}

\else

\author{%
  David S.~Hippocampus\thanks{Use footnote for providing further information
    about author (webpage, alternative address)---\emph{not} for acknowledging
    funding agencies.} \\
  Department of Computer Science\\
  Cranberry-Lemon University\\
  Pittsburgh, PA 15213 \\
  \texttt{hippo@cs.cranberry-lemon.edu} \\
}

\fi
\begin{document}

\setlength{\abovedisplayskip}{3.2pt}
\setlength{\belowdisplayskip}{3.2pt}

\ifdefined\isarxiv%
  \maketitle
\begin{abstract}
We study nonparametric maximum likelihood estimation of probability densities under a total variation (TV) type penalty, sectional variation norm (also named as Hardy-Krause variation). TV regularization has a long history in regression and density estimation, including results on $L^2$ and KL divergence convergence rates. Here, we revisit this task using the Highly Adaptive Lasso (HAL) framework. We formulate a HAL-based maximum likelihood estimator (HAL-MLE) using the log-spline link function from \citet{kooperberg1992logspline}, and show that in the univariate setting the bounded sectional variation norm assumption underlying HAL coincides with the classical bounded TV assumption. This equivalence directly connects HAL-MLE to existing TV-penalized approaches such as local adaptive splines \citep{mammen1997locally}. 
We establish three new theoretical results: (i) the univariate HAL-MLE is asymptotically linear, (ii) it admits pointwise asymptotic normality, and (iii) it achieves uniform convergence at rate $n^{-(k+1)/(2k+3)}$ up to logarithmic factors for the smoothness order $k \geq 1$. These results extend existing results from \citet{van2017uniform}, which previously guaranteed only uniform consistency without rates when $k=0$. We will include the uniform convergence for general dimension $d$ in the follow-up work of this paper. The intention of this paper is to provide a unified framework for the TV-penalized density estimation methods, and to connect the HAL-MLE to the existing TV-penalized methods in the univariate case, despite that the general HAL-MLE is defined for multivariate cases. 
\end{abstract}


\else
\maketitle
\begin{abstract}

\end{abstract}

\fi

\section{Introduction}
\label{intro}
We consider nonparametric density estimation on a compact support. Given a random variable $X \sim P_0$, where $P_0$ is absolutely continuous, and i.i.d. samples $x_{1:n}$ from $P_0$, our goal is to estimate the true density function $p_0$ of the underlying distribution $P_0$. This paper aims to provide a thorough theoretical analysis of univariate density estimation with a variational penalty and its application. The statistical model for $P_0$ will be nonparametric up till assuming that $p_0$ has support on a known interval $[a,b]$ and is a cadlag function with a bounded variation norm, explained in detail below. Our framework is naturally extended to the multivariate case. 

Nonparametric density estimation methods typically include kernel-based approaches, splines, and wavelet techniques. Kernel density estimation (KDE), despite its simplicity, suffers from several drawbacks: it poorly captures densities with rapidly varying regions and faces severe challenges due to the curse of dimensionality in the multivariate case.

The first drawback of KDE arises from the fact that it is a linear smoother whose fitted values depend linearly on the observed responses. TV penalties have been introduced
into spline regression to resolve the same problem with kernel regression, smoothing splines, etc.    \citet{mammen1997locally} developed local adaptive splines (LAS) demonstrating $L^2$ convergence, while \citet{tibshirani2014adaptive} formulated restricted LAS and trend filtering (TF) as general Lasso regressions. Later, the TV penalty was also introduced into the density estimation problem \citep{bak2021penalized, sadhanala2024exponential}. The TV-penalized logspline density estimation (PLSDE) proposed in \citet{bak2021penalized} is more relevant to our problem, and it intends to tackle the oscillation issue of logspines method \citep{kooperberg1991study, kooperberg1992logspline}. \citet{kooperberg1992logspline} employs an exponential link transformation ensuring positivity:
\begin{align}
\label{Link_function}
p_{f}(x) = \frac{\exp(f(x))}{\int_a^b \exp(f(u)) d\mu(u)},
\end{align}
with splines $f(x) = \sum_j \beta_j \phi_j(x)$, typically cubic B-splines. Parameters are estimated via maximum likelihood, often guided by criteria such as AIC or BIC. A theoretical analysis of the logspline model is provided in \citet{stone1990large}.
\citet{bak2021penalized} employed TV penalty and BIC criteria to provide univariate KL-divergence convergence analysis and its generalization to bivariate case. However, their method, Penalized Log-spline Density Estimation (PLSDE), encounters difficulties generalizing to multivariate settings without assuming higher-order continuity. This is a manifestation of the curse of dimensionality, and it falls into the same problem for the multivariate spline when the TV penalty is not introduced. Another problem is that PLSDE is limited to the uniform knots, which is not preferred in the multivariate case. An argument of the knot placement problem is provided in Section~\ref{The Highly Adaptive Lasso (HAL) Assumptions}.

In contrast, the Highly Adaptive Lasso (HAL) proposed by \citet{van2017generally} and practically introduced in  \citet{benkeser2016highly} estimates multivariate càdlàg functions defined on $[0,1]^d$ using a sectional variation norm, yielding uniform consistency and pointwise asymptotic normality without dimension-enforced smoothness assumptions. The density estimator based on HAL is referred to as HAL-MLE log-splines method. In this paper, we restrict our focus to univariate HAL-MLE in order to compare the performance with the classical log-spines methods and demonstrate theoretical connections to LAS and TF approaches, even though the HAL theory applies to the multivariate case. Other than density estimation, HAL-MLE also provide the asymptotic efficiency guarantee for general pathwise-differentiable statistical estimand, i.e. moments, survival probability, or percentiles, by a simple plug-in or a single-step TMLE procedure as developed in \citet{van2017generally, van2023efficient}.

The remainder of this paper is organized as follows. Section~\ref{The Highly Adaptive Lasso (HAL) Assumptions} introduces the HAL assumption and establishes its connection to the classical bounded total variation (BTV) assumption underlying local adaptive splines. Section~\ref{HAL Density Estimation} presents the construction of HAL-MLE with the log-spline link function \citep{leonard1978density, silverman1982estimation, kooperberg1992logspline, rytgaard2023estimation}. Section~\ref{Theoretical Properties} presents the theoretical results on its univariate $L^2$ convergence, asymptotic linearity, pointwise asymptotic normality, and uniform convergence. We also propose a variance estimator for the density based on the delta method. Section~\ref{section:HAL-MLE for Pathwise Differentiable Statistical Estimands} considers the plug-in HAL-MLE and HAL-TMLE of pathwise differentiable statistical estimands, shown to achieve asymptotic efficiency with influence-curve-based variance estimation. Section~\ref{sct: opt_alg} turns to computation, where we discuss the implementation of a series of optimization algorithms tailored for HAL-MLE. Section~\ref{sec:simulation} reports simulation results, empirically verifying the theoretical guarantees mentioned in the previous sections, and comparing the finite sample performance of HAL-MLE with TF (applied to density estimation), log-splines, and KDE. In Section~\ref{sec:case-study}, we present a case study with the galaxy velocity data to visualize the convergence, confidence intervals, and targetings. We provide a Python package for HAL-MLE density estimation, available at \url{https://github.com/zhengpu-berkeley/HALDensity}, and experiment details at \url{https://github.com/yilongHou/link_HAL_MLE}.

\section{The Highly Adaptive Lasso (HAL) Assumptions} 
\label{The Highly Adaptive Lasso (HAL) Assumptions}

The HAL assumptions refer to \emph{càdlàg} functions with bounded sectional variational norm (BSVN), and the modeling of HAL is a natural extension from the exact representation of this function class. In the univariate case, a \emph{c\`adl\`ag} function could be considered as a function that can be approximated by a linear combination of simple indicator functions that jumps from 0 to 1 at a knot-point. The function class includes the linear combination of cumulative distribution functions.  \citet{gill2001inefficient} showed that the concept of SVN extends the Donsker property of univariate \emph{càdlàg} functions with bounded total variation to multivariate cases. Some follow-up works of this results appear in \citet{van2017generally, rytgaard2023estimation}. A detailed discussion on \emph{c\`adl\`ag} functions with BSVN and their measurability is provided in \citet[Section~2]{munch2024estimating}. A similar result also appears in \citet{radulovic2017weak}. And in some literature like \citet{radulovic2017weak, ki2024mars}, the SVN is referred to as the Hardy-Krause variation. We here introduce these concepts and their relationship to HAL.

\subsection{Basic Concepts.} \label{Basic Concepts}
We first introduce the function class of \emph{c\`adl\`ag} functions with BSVN (upper bounded by some \(U < \infty\)) without assuming any order of differentiability, denoted as $D^{(0)}_U([0,1]^d)$.
\begin{definition}[\emph{Càdlàg} Functions with Bounded Sectional Variational Norm (BSVN)]\label{assump:cadlag with BSVN}
The function $f\in D^{(0)}_U([0,1]^d)$ is the function that can be represented as follows:
\[
  f(x)
  = f(0)
    + \sum_{s\subset\{1,\dots,d\}}
      \int_{(0(s),\,x(s)]} f\bigl(du_s,0(-s)\bigr),
\]
where \(s\) is any nonempty subset of \(\{1,\dots,d\}\), and $-s \equiv \{1,\dots,d\} \setminus s$. Here, $x(s)$ denotes the components of $x$ with indices in $s$, while $0(s)$ is the zero vector of the same length as $x(s)$, so that $(0(s), x(s)]$ is a left-open right-closed cube within the $\vert s \vert$-dimensional unit cube. Accordingly, the SVN is defined by
\[
  \Vert f \Vert_v^*
  = \bigl|f(0)\bigr|
    + \sum_{s\subset\{1,\dots,d\}}
      \int_{(0(s),\,x(s)]}
      \Bigl|\,f\bigl(du(s),0(-s)\bigr)\Bigr|.
\]
This is a stronger assumption than the bounded total variation (BTV) assumption, since it requires bounded variation not only on the full coordinate set \(s=\{1,\dots,d\}\) but on all lower‑dimensional sections \(s\subset\{1,\dots,d\}\).  
\end{definition}

In the univariate case $D^{(0)}_U([0,1])$, these definitions simplify to
\begin{align}
    f(x)
  = f(0) + \int_{(0,x]} f(du),
  \quad
  \Vert f \Vert_v^*
  = \bigl|f(0)\bigr|
    + \int_{(0,x]}
      \Bigl|\,f(du)\Bigr|.
\end{align}

When \(f\) is also $k$-th order differentiable, we define the $k$-th order sectional variational norm, \( \Vert f \Vert_{v,k}^*\), and the function subclass of $D^{(0)}_U([0,1])$ with bounded \( \Vert f \Vert_{v,k}^*\), $D^{(k)}_U([0,1])$, as follows: 

\begin{definition}[The Function with k-th order BSVN]\label{assump:DkM}
For the smoothness order $k\ge1$ and a large enough constant $U>0$, we say $f\in D^{(k)}_U([0,1])$ if the following hold:
\begin{enumerate}
  \item $f\in D^{(0)}_U([0,1])$.
  \item For each $1\le i\le k$, the $i$‑th Lebesgue–Radon–Nikodym derivative of $f$ exists
    \[
      f^{(i)}(u)\;=\;\frac{\mathrm{d}}{\mathrm{d}u}f^{(i-1)}(u)
    \]
    and $f^{(i)}\in D^{(0)}_U([0,1])$.
  \item The $k$‑th order sectional variation norm of $f$,
    in the univariate case, simplified as
    \[
      \Vert f \Vert_{v,k}^*
      \;=\sum_{i=0}^k |f^{(i)}(0)| + \mathrm{TV}\bigl(f^{(k)}\bigr),
    \]
    satisfies
    \[
      \Vert f \Vert_{v,k}^* \;\le\; U, \text{for some $U < \infty$}
    \]
\end{enumerate}
\end{definition}
A quick note on the notation: We use $\mathrm{TV}(\cdot)$ to denote the total variation functional as defined in Section 9.1 of \citet{tibshirani2022divided}. When the function $f$ is a càdlàg function, $\mathrm{TV}(f) = \int_0^1 |df(u)|$. This is also referred to as the total variation norm (a seminorm), denoted as $|f|_v$. We then write $(x - u)_+^k = (x - u)^k\,I\{u\le x\}$ where $I$ is the indicator, and this truncated $k$-th order spline with knot point $u \in [0,1]$ is $k$-th order weakly differentiable. Notice that the following representation theorem and its derivation are specialized to the univariate case, and we include it here because it motivates and unifies the parametrization of such a function class. Here we demonstrate the exact representation for $f\in D^{(k)}_U([0,1])$ with the concepts above. 
\begin{proposition}[Univariate Exact Representation]
\label{thm:HAL_representation}

Any $f\in D^{(k)}_U([0,1])$ could be represented as follows:
\[
  f(x)
  = \sum_{i=0}^k \frac{1}{i!}\,f^{(i)}(0)\,x^i
    + \int_0^1 \frac{1}{k!}\,(x-u_k)_+^k \,d\,f^{(k)}(u_k).
\]
\end{proposition}

The multivariate or higher‑order spline exact representation is included in \citet[Section 3]{vanderlaan2023higherordersplinehighly}. 

\subsection{Relationship with TV} \label{Relationship with TV}

BTV is a classic assumption of spline methods for the regression problem. The implementation is to derive TV as a function of the coefficients of the bases. Thus, different basis systems have different representations of the TV. The HAL assumptions seem to be stricter by Definition~\ref{assump:cadlag with BSVN}, but there is actually no loss comparing to the BTV assumption in the univariate case. LAS uses the same truncated power basis as in univariate HAL, and therefore we argue the equivalency of the two assumptions through analyzing the estimation of LAS.

\paragraph{A Review of Locally Adaptive Regression Splines (LAS):}
Suppose $f$ is a $k$-th order differentiable function, and $f^{(k)}$ is the $k$-th order derivative of $f$. By convention $f^{(0)} \equiv f$, or WLOG we can adopt the term that $f^{(k)}$ is the $k$-th weak derivative of $f$ from \citet{tibshirani2014adaptive}.
\citet{mammen1997locally} proposed LAS, which  achieved optimal $L^2$ convergence rates in $\mathcal{F}_k$, where 

\[
\mathcal{F}_k
\;=\;
\bigl\{\,f : [0,1]\to\mathbb{R}
\;\big|\;
f\text{ is $k$-th order differentiable and }
\mathrm{TV}\bigl(f^{(k)}\bigr)<\infty
\bigr\}.
\]
The optimization over $\mathcal{F}_k$ is referred to as the unrestricted LAS method as follows:
\begin{equation}\label{eq:ula-spline}
\hat f \;\in\;\arg\min_{f\in\mathcal{F}_k}
\frac{1}{2}\sum_{i=1}^n \bigl(y_i - f(x_i)\bigr)^2
\;+\;\lambda\,\mathrm{TV}\bigl(f^{(k)}\bigr),
\end{equation}
Due to the difficulty of implementation, especially when $k \ge 2$, they proposed the restricted LAS method as an asymptotic solution, as follows:

\begin{align*}
\mathcal{F}_{n,k}
&=
\bigl\{\,f:[0,1]\to\mathbb{R}
\;\big|\;
f\text{ is the linear combination of the $k$-th order splines with knots in }x_{1:n},\\
&\qquad\mathrm{TV}\bigl(f^{(k)}\bigr)<\infty
\bigr\},
\end{align*}
and the optimization problem becomes
\begin{equation}\label{eq:rla-spline}
\hat f \;\in\;\arg\min_{f\in\mathcal{F}_{n,k}}
\frac{1}{2}\sum_{i=1}^n \bigl(y_i - f(x_i)\bigr)^2
\;+\;\lambda\,\mathrm{TV}\bigl(f^{(k)}\bigr).
\end{equation}

\citet{mammen1997locally} showed the asymptotic convergence of the solutions of the two problems when the distances among $x_{1:n}$ are close enough. In \citet[Theorem~9]{mammen1997locally}, they showed that, for the truth $g_{0,n}$ in a generic function class $\mathcal{G}$ with $\mathrm{BTV}$, if one can choose a good enough linear subspace $\mathcal{G}_{n,k}$ with an oracle $g_{1,n}$, then the solution $\hat{g}$ on $\mathcal{G}_{n,k}$ preserves the optimal $L^2$ convergence on $\mathcal{G}$. Thus, the asymptotic conclusion holds by choosing $\mathcal{G}$ with $\mathrm{BTV}$ to be $\mathcal{F}_{k}$, and $\mathcal{G}_{n,k}$ to be $\mathcal{F}_{n,k}$.

\paragraph{The Missing \emph{C\`adl\`ag} Perspective:}
As discussed, the restricted LAS uses a finite linear combination of truncated power basis as the estimation, and thus the estimation is always a \emph{c\`adl\`ag} function. Its generalization to non-\emph{c\`adl\`ag} functions is achieved through the equivalency of a class of functions with the same $\mathrm{TV}$. In order words, one can make a function into a \emph{c\`adl\`ag} function by shifting finite points while maintaining the same $\mathrm{TV}$, and the LAS provides the same estimation. The $L^2$ convergence still holds for this finite pointwise difference. For example, $I(x<1)$ is left continuous, moving the point at $x = 1$ up to 1 provides $I(x \le 1)$ which is \emph{c\`adl\`ag}, and they are indistinguishable under the same TV penalty. That is to say, among this set of indistinguishable functions, there must be a \emph{c\`adl\`ag} function. And the convergence to non-\emph{c\`adl\`ag} functions achieved by the LAS ignores pointwise behavior.

We claim that, by the representation in Proposition \ref{thm:HAL_representation}, $D^{(k)}_U([0,1])$ is the closure of the infinite linear combination of $k$-th order truncated power splines. Thus, we conclude that the solution based on the finite-dimensional working model $\mathcal{F}_{n,k}$ is a finite-dimensional approximation of $f$ in $D^{(k)}_U([0,1])$, where $f$ is the indistinguishable \emph{c\`adl\`ag} function of the truth $f_0$. Noticing that the pointwise behavior does not hurt $L^2$ convergence, but we need the rigorous \emph{c\`adl\`ag} assumption for proving uniform convergence in Section~\ref{Theoretical Properties}.

\paragraph{Conclusion:} In the univariate case, when $k = 0$, the SVN only requires the existence of $|f(0)|$ and BTV. And when $k \ge 1$, the \emph{c\`adl\`ag} assumption and the existence of $|f^{(i)}(0)|$ for all $1 \le i \le k$ is implied by differentiability, and so BSVN and BTV are equivalent assumptions.

\begin{remark} We could argue that we only need the Lebesgue–Radon–Nikodym derivative, and relax our assumption by a null set with Lebesgue measure 0. But they could be augmented in a similar manner as well. This is the reason \citet{tibshirani2014adaptive} introduced weak differentiability for the function class and the working model in Section 3 therein. 
\end{remark}

\section{HAL Density Estimation} \label{HAL Density Estimation}
With the HAL assumptions and the same link function in \citet{kooperberg1992logspline}, we here form the density estimation technique, HAL-MLE. The intuition of this choice of link function is to assume that the log-likelihood of the density follows a càdlàg structure with bounded sectional variation, making it well-suited for modeling via HAL. This link function is also referred to as the histo-spline method \citep{leonard1978density, silverman1982estimation}. Specifically, when using zero-order HAL basis functions, the function \( f \) can be expressed as a histogram. One major benefit of this link function is that the resulting density belongs to the exponential family, ensuring the convexity of the MLE optimization with $L_1$ norm regularization.

\subsection{HAL Construction.} \label{Estimation Construction}
Proposition \ref{thm:HAL_representation} shows that we can represent any function $f \in D^{(k)}_U([0,1])$ by an infinite linear combination of the $k$-th order splines (bases). We then could construct our estimation with a finite-dimensional working model by approximating the integral with a sum. This explains the bias of HAL model by not including all the bases. We initialize the pre-selection working model by partitioning the integral into $n$ parts with knot points $x_1, \cdots, x_n$. Let $f_0\in D^{(k)}_U([0,1])$ be the truth, and $f_n$ be the estimation based on the $n$-knot points. Note that the dimension of the working model here refers to the parameters needed for modeling rather than the structure of observations.

\[
  f_n(x)
  = \sum_{i=0}^k \frac{f^{(i)}(0)}{i!}\,x^i
    + \sum_{j =1}^n \frac{\Delta\,f^{(k)}(x_j)}{k!}\,(x-x_j)_+^k \,.
\]

And we could summarize the basis system as follows:

\[
\phi_{i,0}(x) =  x^i, \quad i = 0, \ldots, k,
\]
\[
\phi_{k,x_j}(x) = (x - x_j)_+^k, \quad j = 1, \ldots, n.
\] \label{basis_d1}

This basis naturally divides into:
\begin{itemize}
    \item The \textbf{parametric (global) part}: the global polynomial trend ($x^i$-type terms),
    \item The \textbf{nonparametric (local) part}: the truncated power functions that allow flexible deviations at specified knots $x_{1:n}$.
\end{itemize}

\begin{remark}
Such a basis system is also called the truncated power basis \citep[Theorem 8.51]{schumaker2007spline}, and each basis can be represented by a tuple of basis order and knot points, i.e., for the nonparametric part of the basis system, we index them each by a tuple $(k, u)$ for $k$-th order splines and any knot point $u \in (0,1]$. For the parametric part of the basis system, we index them each by a tuple $(i, 0)$, for $i \leq k$ and starting point 0.
\end{remark}

We denote this working model index set as $\mathcal{R}^k(n)$ with cardinality $n+k+1$. This is referred to as the initial working model, whereas the model after Lasso selection is called the post-selection working model, denoted as $\mathcal{R}_n$ with cardinality $J_n$. Notice that the initial working model involves $(n+k+1)$ parameters, and an element in this working model is of the form
\[
  f_{n,\beta}(x)
  = \sum_{i=0}^k \beta_i\,\phi_{i,0}(x)
    + \sum_{j=1}^n \beta_{k+j}\,\phi_{k,x_j}(x).
\]
Hence, its $k$-th order sectional variation norm is
$\bigl\|f_{n,\beta}\bigr\|_{v,k}^* = \|\beta\|_1.$ And for a general loss \(L\), the empirical risk minimizer is
\begin{align} \label{equation: HAL beta optimizer}
 \beta_n(M)
  = \arg\min_{\|\beta\|_1 \le M}
    P_n\,L\bigl(f_{n,\beta}\bigr),
  \qquad
  f_n(x)=f_{n,\beta_n(M)}(x).   
\end{align}
$M$ here serves as a user-specified hyperparameter as the $L_1$ norm constraint of the basis coefficients. In practice, it could be tuned by cross validation and the corresponding choice of $M$ is denoted as $M_{cv}$, and the process is referred to as CV-HAL-MLE. The undersmoothening of HAL-MLE is achieved by choosing a $M' > M_{cv}$, the post-selection working model enlarges, but still indexed by the initial working model. The undersmoothened HAL-MLE reduces bias by including more bases, but increases variance as a trade-off. One can refer to Appendix~\ref{Appendix A: Notation, Definition} for more detailed definitions.

\subsection{HAL-MLE with Link Function} \label{HAl_MLE with Link Function}

As the setup in the introduction, let the $X\sim P_0$ be defined on $[0,1]$. We index the density $p_0 = \mathrm{d}P_0/\mathrm{d}\mu$ in terms of a $p_{f_0}$ for a function $f_0 \in D^{(k)}_U([0,1])$, where $\mu$ could be chosen as the Lebesgue measure for simplicity. Thus, the statistical model is defined as $\mathcal{M} = \{ p_{f} : f \in D^{(k)}_U([0,1]) \}$. The HAL-MLE of is then denoted as $f_{n, \beta_n(M)}$, and $\beta_n(M)$ is estimated as follows:
\[
\beta_n(M) = \arg \min_{\beta: \Vert \beta\Vert_1 < M} \sum_{i=1}^{n} L(f_{n, \beta})(X_i) = \arg \max_{\beta:  \Vert \beta\Vert_1 \leq M} \sum_{i=1}^{n} \sum_{j=1}^{n+k+1} \beta_j \phi_j(x_i) - n\log C(\beta),
\] where \( x_{1:n} \sim  P_0,  i.i.d.\), $L = -\log p_f$ is the loss function, and the $C(\beta) = \int_0^1 \exp(f_{\beta}(x)) \mathrm{d}x$ is the normalizing constant. 

\subsubsection{Comparison with NPMLE}
Our set ${\cal R}^k(n)$ of $k$-th order spline basis functions is determined by data and its linear span is flexible enough to provide a strong approximation of any function in our class $D^{(k)}_U([0,1])$. Moreover, for large enough $U$, the latter class will contain the true $f_0$ identifying the true density $p_{f_0}$ of $X$. 
Therefore, the HAL-MLE over the linear combination of basis in ${\cal R}_n$ with some $L_1$ norm upper bound, denoted as $D^{(k)}_U({\cal R}_n)$ approximates an MLE over the full class $D^{(k)}_U([0,1])$ containing the true $f_0$. Consequently we can view the HAL-MLE as  a Nonparametric Maximum Likelihood Estimation (NPMLE) method for estimating the true density of the data. An important distinction between HAL‑MLE and the traditional NPMLE methods, like the empirical distribution and Kaplan–Meier (KM) estimator \citep{kaplan1958nonparametric}, lies in the functional constraint HAL imposes. HAL‑MLE produces density estimates confined to the space $D^{(k)}_U([0,1])$. In contrast, traditional NPMLE imposes no smoothness or variation‑norm constraints on the densities. As a consequence, the NPMLE is attained at the (non-density) empirical discrete distribution that puts mass $1/n$ on each observation $x_i$. This is not even a real density, even when there are reasons to believe that the underlying true density is continuous. Similarly, the NPMLE of the failure time distribution for right-censored data under the assumption of independent censoring  is attained at a discrete failure time distribution, while the HAL-MLE analogue would result in a valid density estimator of the true failure time density.  
NPMLE can still provide efficient plug-in estimators of smooth functionals of the density, such as a survivor function. The HAL-MLE, possibly extended with the HAL-TMLE, also yield efficient plug-in estimators of such smooth features, while it still also estimates the density well.
\subsubsection{Uniform Grid vs Data Adaptive Grid}
Uniform grid is a common choice in the literatures of univariate splines methods. PLSDE \citep{bak2021penalized} benefits from this choice for the representation of TV penalty with B-splines. TF \citep{tibshirani2022divided} also benefits from this choice for the simple recursion of the divided difference matrix. And theoretically, both uniform grid and data-adaptive grid could provide a strong approximation of any BTV function.

In HAL-MLE, we prefer data-adaptive initialization rather than uniform knots, even though the matter is nuisance in the univariate case. When generalized to multivariate case, the uniform grid becomes an initialization of $n^d$ knots. However, a $d$-dimensional function has $2^d - 1$ sections, and HAL initialize $n(2^d - 1)$ knots, $n$ knots per section. We have a much smaller working model to begin with. 

\section{Theoretical Properties for HAL-MLE} \label{Theoretical Properties}
With the HAL assumptions and HAL-MLE setup, we here provide a comprehensive theoretical analysis of the convergence properties of spline density estimation with variational penalty, including $L^2$-convergence, pointwise asymptotic normality, and uniform consistency.

\subsection{Rate of Convergence  of HAL-MLE in Loss-based Dissimilarity} \label{subsec:rate-of-convergence-hal-mle-loss-based-dissimilarity}

We first establish the $L^2$ convergence rate of the HAL-MLE. This result is natural in the context of univariate spline methods penalized by total variation, 
and similar conclusions have been obtained for related estimators 
\citep{mammen1997locally, tibshirani2014adaptive, bak2021penalized}. 
Our contribution is to demonstrate the $L^2$ convergence of the HAL-MLE without any assumptions, 
using a standard proof strategy based on loss-based dissimilarity. 
Recall that the loss-based dissimilarity, $d_{0}(f_{n},f_{0}) \equiv P_0{L(f_n)} - P_0{L(f_0)}$, coincides with the Kullback--Leibler divergence when $L = -\log p_f$. 
Let
\[
\mathcal F := D^{(k)}_U([0,1]), 
\qquad 
f_{0} := \arg\min_{f \in \mathcal F} P_{0}L(f),
\]
where $D^{(k)}_U([0,1])$ denotes the $k$th-order bounded variation class on $[0,1]$. 
For a user-specified $L_1$-norm constraint $M$ and a finite-dimensional index set $\mathcal{R}_n$, 
define the empirical estimator
\[
f_{n} := \arg\min_{f \in D^{(k)}_M(\mathcal{R}_n)} P_{n}L(f).
\]

The proof relies on two facts.  
\begin{enumerate}
  \item (Positivity Condition) With the link function, for all $f \in \mathcal F$, there exists $\delta > 0$ such that $p_f \geq \delta$ uniformly.
  \item (Second-order Loss Behavior Condition) When $L(f) = -\log p_f$ (negative log-likelihood), we have
  \[
  \sup_{f \in \mathcal F} 
  \frac{P_{0}\!\bigl\{L(f) - L(f_{0})\bigr\}^{2}}{d_{0}(f,f_{0})}
  = O(1) < \infty.
  \]
\end{enumerate}

\begin{remark}
For the first fact, note that under the bounded variation assumption, all 
$f \in \mathcal F$ are uniformly bounded. Consequently, the link function is 
uniformly bounded away from zero and bounded above, so the required positivity 
condition holds automatically. The second fact follows directly from 
\citet{van2004asymptotic}.
\end{remark}

\begin{remark} \label{remark:logspline-convergence}
For the ones familiar with the logspline literature, the $L^2$ and $L^{\infty}$ convergence of logsplines in \citet{stone1990large} are very different from ours. First, they deal with a parametric MLE on equally spaced knots, whereas we deal with data-adaptive knots.Second, the convergence they showed is between the estimation and the parametric oracle MLE, which is defined as the argmax of the expected value of the parametric log-likelihood.
\end{remark}

\begin{theorem}[$L^2$ convergence of HAL-MLE]
\label{thm:l2-convergence-hal-mle}
If $f_0 \in D^{(k)}_U([0,1])$ and $L(f) = -\log p_f$, then,
\[
d_{0}(p_{f_n},p_{f_{0}}) = O_P\!\left(n^{-\frac{2k+2}{2k+3}}\right),
\quad \text{and hence} \quad
\lVert p_{f_n} - p_{f_{0}}\rVert_{L^2(\mu)} = O_P\!\left(n^{-\frac{k+1}{2k+3}}\right).
\]
\end{theorem}



\textcolor{blue}{This rate agrees with the $L^2$ convergence rate for order $k=0$ and dimension $d=1$ in regression \citep{mammen1997locally, tibshirani2014adaptive, bibaut2019fast}.
Our density estimation setting preserves this rate from the regression setting, which is sharp and without any assumptions.}


\subsection{Pointwise Asymptotic Normality of the HAL-MLE}\label{subsect: Pointwise Asymptotic Normality of the HAL-MLE}
Recall the link function, $p_{\beta}=p_{f_{\beta}}=\frac{\exp(f_{\beta})}{C(\beta)}$ where $C(\beta)=\int_0^1 \exp(f_{\beta}(x)) \mathrm{d}x$. The score for the log-likelihood loss is calculated as $S_{\beta}(\phi_j)=\frac{\mathrm{d}}{\mathrm{d}\beta(j)}\log p_{\beta}$ at $\beta$ for the $j$-th entry $\beta(j)$. Note that this equals
\begin{align}
S_{\beta}(\phi_j)=\phi_j(X)- P_{\beta}\phi_j,
\end{align}
where $P_{\beta}\phi_j\equiv \int \phi_j(x)p_{\beta}(x) \mathrm{d}x = \mathbb{E}_{\beta}[\phi_j]$. We use the notation $S_{\beta}(\phi)$ to denote the score vector, and $\phi \equiv (\phi_1, \cdots, \phi_{J_n})^\top$. By the linearity of expectation functional, we note $\phi\rightarrow S_{\beta}(\phi)$ is linear in $\phi$. 

Recall that \(\mathcal R_n\subset R^k(n)\) is a data‐adaptive index set of basis functions with cardinality $J_n$ corresponding with the non-zero coefficients in the HAL-MLE fit $f_{n,\beta_n}$, using a particular  $M_n$ for the $L_1$-norm constraint. The HAL-MLE implies a $J_n$-dimensional working model
\[
D^{(k)}(\mathcal R_n)
= \Bigl\{f_{n,\beta}=\sum_{j\in\mathcal R_n}\beta_j\,\phi_j(x):\beta\Bigr\}
\;\subset\;D^{(k)}([0,1]).
\]
The HAL‐MLE $f_n=f_{n,\beta_n}$ also equals a minimizer over this working model $D^{(k)}_{M_n}({\cal R}_n)$ so that
\[
f_n \;=\;\arg\min_{f\in D^{(k)}_{M_n}(\mathcal R_n)}\;-\,P_n\log p_f .
\]

One nice property of the link function is that the information matrix defined as the negative of the partial derivative of the expectation of the score, $-d/d\beta(i) P_0 S_{\beta}(\phi_j)$, equals the covariance of the scores $S_{\beta}(\phi_j)S_{\beta}(\phi_i)$ under $P_{\beta}$, which can be viewed as an inner product between the two basis functions $\phi_i$ and $\phi_j$. This inner product is an inner product for the linear space $D^{(k)}({\cal R}_n)$ providing an orthonormal basis that will make the information matrix for the reparametrized $D^{(k)}({\cal R}_n)$ a diagonal matrix and while making the scores orthogonal, and equivalently, uncorrelated w.r.t. $P_{\beta}$. This provides us with a powerful ingredient for our formal analysis of the HAL-MLE. 

Specifically, for a basis function $\phi_i$ and its corresponding parameter $\beta(i)$, the $i$-th element of $\beta$,
\begin{eqnarray} \label{eq:inner-product-score}
-\frac{\mathrm{d}}{\mathrm{d}\beta(i)}P_0 S_{\beta}(\phi_j)
&=& 0+ \frac{\mathrm{d}}{\mathrm{d}\beta(i)} P_0 \{P_{\beta}[\phi_j]\} \nonumber = \frac{\mathrm{d}}{\mathrm{d}\beta(i)} P_{\beta}[\phi_j]\\ &=& \int \phi_j(x) \frac{\mathrm{d}}{\mathrm{d}\beta(i)} p_{\beta}(x) \, \mathrm{d}x \nonumber 
= \int \phi_j(x) S_{\beta}(\phi_i) p_{\beta}(x) \, \mathrm{d}x \nonumber \\
&=& \int (\phi_j(x) - P_{\beta} \phi_j) S_{\beta}(\phi_i) p_{\beta}(x) \, \mathrm{d}x \nonumber 
= P_{\beta} S_{\beta}(\phi_j) S_{\beta}(\phi_i) \nonumber \\
&\equiv& \langle \phi_i, \phi_j \rangle_{\beta}.
\end{eqnarray}
W.r.t. the orthonormal basis parametrization of $D^{(k)}({\cal R}_n)$ for which the information matrix is diagonal, we can obtain a straightforward linearization of the HAL-MLE w.r.t. the oracle MLE $f_{{\cal R}_n,0}\equiv \arg\min_{f\in D^{(k)}_M({\cal R}_{n})}-P_0 \log p_f$ by being able to trivially invert the information matrix, thereby avoiding the delicate understanding of eigenvalues of the original information matrix in terms of the original parametrization of $D^{(k)}({\cal R}_n)$.

Nonetheless, such an analysis runs into the following complication. 
With this analysis, we can approximate $f_{n,\beta_n}-f_{{\cal R}_n,0}$ with an empirical mean of a function of $X_i$ where this function is indexed by ${\cal R}_n$, thereby not allowing the application of a central limit theorem. In other words, the data dependence of ${\cal R}_n$ makes it hard to push towards a CLT. One could resolve this by using a a priori set of $J_n$ knot-points and select $J_n$ (e.g. uniform knots) with cross-validation. However, the Lasso provides an effective way of selecting the right set of knot-points for fitting the true density so that this would be a practically inferior procedure, and the utility of the Lasso to select the working model is even more important for the extension to multivariate density estimation.

Therefore, for proving the pointwise asymptotic normality of the HAL-MLE, we create an independent version ${\cal R}_{n,0}$ of ${\cal R}_n$ that is independent of $P_n$ and yields a sup-norm approximation of $D^{(k)}_U([0,1])$ at rate $O(1/J_{n,0}^{k+1})$ up till a $\log n$ factor. Such sets have been shown to exist in \citet{vanderlaan2023higherordersplinehighly}. We could view  ${\cal R}_n=\hat{R}(P_n)$ as an algorithm that maps the data $P_n$ into a set of basis functions. Let $P_n^{\#}$ be the empirical probability measure of an independent sample of $n$ i.i.d. observations from $P_0$. We can then define a particular independent version as 
\(\mathcal R_{n,0} = \hat{R}(P_n^{\#})\).
We use this independent version to define an independent working model $D^{(k)}({\cal R}_{n,0})$ of dimension $J_{n,0}$, which can be viewed as an approximation of the data dependent working model $D^{(k)}({\cal R}_n)$. 
This independent working model defines a different oracle MLE
\[
f_{n,0}=f_{{\cal R}_{n,0},0}
\equiv\arg\min_{f\in D^{(k)}_M(\mathcal R_{n,0})}\;-\,P_0\log p_f
\;=\;f_{{\cal R}_{n,0},\beta_{n,0}}.
\] 
To emphasize the two different working models we also use notation $f_{{\cal R}_n,\beta}$ and $f_{{\cal R}_{n,0},\beta}$ as parametrizations of the two working models $D^{(k)}({\cal R}_n)$ and $D^{(k)}({\cal R}_{n,0})$. Specifically, we also use notation $f_{{\cal R}_n,0}$ and $f_{{\cal R}_{n,0},0}$ for the two oracle MLEs,  where the latter is also denoted with $f_{n,0}$.
We now analyze the HAL-MLE w.r.t.  the independent oracle MLE $f_{{\cal R}_{n,0},0}$ instead of $f_{{\cal R}_n,0}$, relying on showing that the HAL-MLE over $D^{(k)}({\cal R}_n)$ indeed can be viewed as an approximate MLE over $D^{(k)}({\cal R}_{n,0})$.

Although \(f_n\) is not a full MLE over \(D^{(k)}(\mathcal R_{n,0})\), it can be shown to act as an approximate MLE solving its score equations at close enough approximation so that $f_n$ not only acts as an MLE over the data dependent working model $D^{(k)}({\cal R}_n)$ but also as an MLE over the independent working model $D^{(k)}({\cal R}_{n,0})$. The approximation error in solving these score equations is a term in our analysis that is formally addressed in Appendix~\ref{App C: Score Analysis} by introducing the $L_1$ constrained score space. The HAL-MLE, $f_n$, has one $L_1$ norm constraint on the coefficients, and thus solves $J_n - 1$ score equations. We name the linear span of these $J_n - 1$ scores as the $L_1$ constrained score space, and this space is a closed linear subspace of $D^k(\mathcal{R}_n)$. 

Let \(\{\phi_j^*: j\in\mathcal R_{n,0}\}\) be an orthonormal basis of \(\operatorname{span}\{\phi_j: j\in\mathcal R_{n,0}\}\) under the inner product
\(\langle f,g\rangle_{\beta_{n,0}}\). We define a normalized score vector at \(x\) by
\(
\bar{\phi}_{n,x}(X)\equiv J_{n,0}^{-1/2}\sum_{j\in {\cal R}_{n,0}}\phi_j^*(x)\phi_j^*(X)\).
Notice that  for each value $x\in [0,1]$, $\bar{\phi}_{n,x}$ is a function on $[0,1]$. Notice that $J_{n,0}^{-1/2} \bar{\phi}_{n,x}(X) = J_{n,0}^{-1}\sum_{j\in {\cal R}_{n,0}}\phi_j^*(x)\phi_j^*(X)$, which is a linear combination of the orthonormal basis functions $\phi_j^*$, and we believe it to be uniformly bounded. Additionally, we use notation $O^+(r(n))$ if the term is $O(r(n)(\log n)^p)$ for some finite $p$, and similarly we define $O_p^+(r(n))$. Note that $p$ can be negative, and usually we can determine this $p$ based on undersmoothening, and it does not influence the main rate $r(n)$.
\begin{theorem}[Asymptotic Linearity of HAL-MLE]
\label{thm:asymptotic-linearity-hal-mle}
Let $J_{n,0}$ satisfy $J_{n,0}^{-(k+1)}=o((J_{n,0}/n)^{1/2})$.
Suppose the following conditions hold for an integer $k \geq 1$:
\begin{enumerate}
  \item (Basis boundedness) $J_{n,0}^{-1/2}\,\Vert  \bar\phi_{n,x}\Vert_{\infty}=O(1)$;
  \item (Smoothness) $f_0\in D^{(k)}_U([0,1])$ for some $U<\infty$;
  \item (Uniform Approximation of $D^{(k)}_U([0,1])$ by $D^{(k)}(\mathcal R_{n,0})$) 
\[
\sup_{f\in D^{(k)}_U([0,1])}\inf_{g\in D^{(k)}(\mathcal R_{n,0})}\Vert f-g \Vert_{\infty}
=O^+\bigl(J_{n,0}^{-(k+1)}\bigr).
\]
  \item ($L^2(\mu)$ Approximation of $D^{(k)}_U([0,1])$ by $D^{(k)}({\cal R}_n)$)
\[
\sup_{f\in D^{(k)}_U([0,1])}\inf_{g\in D^{(k)}(\mathcal R_{n})}\Vert f-g \Vert_{L^2(\mu)}
=O^+\bigl(J_{n}^{-(k+1)}\bigr),
\]
and can be extended to the corresponding $L_1$ constrained score space in $D^{(k)}(\mathcal R_{n})$.

\end{enumerate}
Then, with cross‐validated or slightly undersmoothened $f_n$, for each fixed $x\in[0,1]$,
\[
  \sqrt{\frac{n}{J_{n,0}}}\bigl(f_n(x)-f_{{0}}(x)\bigr)
  =\sqrt{n}\,(P_n-P_0)\bigl[S_{f_{n,0}}(\bar\phi_{n,x})\bigr]
  +o_P(1).
\]
The latter term equals $n^{1/2}$-scaled empirical mean of independent mean zero random variables $S_{f_{n,0}}(\bar{\phi}_{n,x})$ with finite variance. $\sqrt{J_{n,0}}S_{f_{n,0}}(\bar\phi_{n,x})$ serves as the influence curve of the HAL-MLE $f_n(x)$ considered as an estimator of the oracle $f_{n,0}(x)$.
\end{theorem}

Here we explain the plausibility and understanding for these approximation assumptions. The bias is identified as the uniform approximation error between the finite-dimensional working model $D^{(k)}(\mathcal R_{n,0})$ and $D^{(k)}_U([0,1])$. \citet[Appendices~D and~F]{vanderlaan2023higherordersplinehighly} show that such a set ${\cal R}_{n,0}$ of size $J_{n,0}$ exists for all sizes $J_{n,0}$. Therefore, we could simply select ${\cal R}_{n,0}$ as such a set and choose it for $J_{n,0}=J_n$.
Moreover, in this article it is argued that, due to HAL achieving the $L^2$ rate of convergence $O_p(n^{-(k+1)/(2k+3)})$, the sets of non-zero coefficients in the HAL fit should satisfy an adaptive version of this approximation condition allowing this same approximation error w.r.t. true target function, and that under some undersmoothing it will satisfy the (non-adaptive) uniform approximation condition. 
And the $L^2$ approximation error is a weaker condition on $D^{(k)}(\mathcal R_{n})$. Its extension to the $L_1$ constrained score space can be explained by approximating one basis function by other basis functions in the score space.

\begin{remark}[The Necessity for Cross-Validation and Undersmoothing]
The $L_1$-norm in the HAL-MLE drives the number of basis functions in ${\cal R}_n$. This number $J_n$ needs to balance between the bias $O^+\bigl(J_{n}^{-(k+1)}\bigr)$ of the working model $D^{(k)}({\cal R}_n)$ and variance $\sqrt{\frac{n}{J_{n}}}$ or the MLE for this working model. Cross-validation can guarantee the optimal choice \citep{van2003unified, vaart2006oracle}. One may need some undersmoothing so that the bias is reduced by a $\log n$ factor, making it negligible relative to the standard error. 
\end{remark}


\begin{theorem}[Pointwise Asymptotic Normality for HAL-MLE]
\label{thm:density_HAL_asymptotic_normality}
Under the conditions of Theorem~\ref{thm:asymptotic-linearity-hal-mle}, let $\sigma^2_{n,0}(x)=P_0(\bar{\phi}_{n,x}-P_{\beta_{n,0}}\bar{\phi}_{n,x})^2$, and then we have
\begin{equation}
    {\sigma}_{n,0}^{-1} \left( \frac{n}{J_{n,0}} \right)^{1/2} \bigl(f_n - f_{{0}}\bigr)(x) \Rightarrow_d N(0,1).
\end{equation}
\end{theorem}

A natural question is whether we can extend these to the corresponding density estimation. Notice that what we demonstrated is the pointwise asymptotic normality instead of the weak convergence of $(f_n-f_0)$ in function space. Thus, we cannot even apply the functional delta method. 
However, we construct the proof of the following corollary by recognizing that it is driven entirely by $f_n-f_0$, while the normalizing constant is a pathwise differentiable plug-in estimator with no contribution to the linearization.

\begin{corollary}[Delta-method for Density Estimation] \label{coro: Delta-method for Density Estimation}
    Under the conditions of Theorem~\ref{thm:asymptotic-linearity-hal-mle}, let $g(f)(x)=\log p_f(x)$, then we have,
    \[
    (n/J_{n,0})^{1/2}( g(f_n)(x)-g(f_0)(x))=(n/J_{n,0})^{1/2}(f_n(x)-f_0(x))+o_P(1).
    \]
    Therefore, $\log p_{f_n}(x)-\log p_{f_0}(x)$ behaves as $(f_n-f_0)(x)$, which proves that 
    \[
    (n/J_{n,0})^{1/2}(p_{f_n}-p_{f_0})(x)=p_0(x)(n/J_{n,0})^{1/2}(f_n-f_0)(x)+o_P(1).
    \]
\end{corollary}

This extends the pointwise asymptotic normality and the uniform convergence from $f_n$ to $p_{f_n}$. 

\subsection{Uniform Convergence for HAL-MLE} \label{subsct: uniform convergence for HAL-MLE}
The uniform convergence of HAL-MLE is a direct consequence of the pointwise asymptotic normality.
\begin{theorem}[Uniform Convergence for HAL-MLE]\label{thm:uniform-density-hal}
Extending the Basis Boundedness Assumption to hold uniformly for $x \in [0,1]$ a.s., while maintaining the other assumptions of Theorem~\ref{thm:asymptotic-linearity-hal-mle}, then
\[
\sup_{x\in[0,1]}
\sqrt{n/J_{n,0}}\,\lvert f_n(x)-f_0(x)\rvert
= O_P(\sqrt{\log n}).\]
By selecting $J_{n,0}\asymp n^{1/(2k+3)}$  up till $\log n$-factors, it follows
\quad
\[
\|f_n - f_0\|_\infty
= O_P^+\Bigl(n^{-(k+1)/(2k+3)}\Bigr).
\]
\end{theorem}

\begin{remark}[Zero‐order Case]
When \(k=0\), the uniform rate above need not hold, but one still has
\(\|f_n - f_0\|_\infty = o_P(1)\) based on \citet{van2017uniform}.
\end{remark}

In Remark~\ref{remark:logspline-convergence}, we mentioned the difference between the log-spline convergence and HAL-MLE convergence. Here we provide more details. First, in \citet{stone1990large}, they deal with a parametric MLE on equally spaced knots, whereas we deal with data-adaptive knots. Second, the convergence they showed is between the estimation and the parametric oracle MLE. Third, their uniform convergence rate relies on the assumption of the size of the post-selection working model to be $o(n^{\frac12 - \epsilon})$ for some $\epsilon > 0$; however, there is no theory suggesting that AIC or BIC can select the right size of the post-selection working model, unlike the theoretical guarantee in \citet{van2003unified} for cross-validation. Here, HAL-MLE shows a stronger result considering the bias between the oracle and the truth with the simplest setting of just using the data-adaptive knot points and cross-validation.

\subsection{Generalizability of the Proofs} 
\label{subsect: Generalizability of the Proofs}
Another interesting discussion is the generalizability of these proofs to different basis functions but asymptotically spanning the same function space, such as the B-splines and falling factorial basis. For the B-splines, they are linear combinations of the truncated power basis \citep{de1976splines}. For the falling factorial basis, they are identical to the truncated power basis up to a constant when $k=0$ and $k=1$. When $k \geq 2$, the proximity of the truncated power basis and the falling factorial basis is discussed in Chapter 10.1 of \citet{tibshirani2022divided}. Nevertheless, when $k \geq 1$, the falling factorial basis is continuous and therefore trivally càdlàg. Thus, these three basis systems should span the same space. The key assumptions of HAL-MLE proofs rely on i) bounding the SVN, ii) being an MLE on a fine enough working model. This is the reason why we have the same $L^2$ convergence rate as in the regression case in \citet{fang2021multivariate, ki2024mars}, even though the multivariate basis system is different. So our guess is that these proofs can be generalized to the B-splines and falling factorial basis under the same setting. However, the representation of TV is complicated for the B-splines with data-adaptive knots as mentioned in Section 3 of \citet{wang2014falling}.

\subsection{Variance Estimation for Density}  \label{density_sd_section}
Based on pointwise asymptotic normality, we propose a variance estimator using the Delta method, 
analogous to the variance estimator for generalized linear models (GLMs). 
Consider the parametric working model for HAL-MLE, i.e., 
\[
f_n(x) = \phi(x)^{\top}\beta_n,
\]
where $\phi$ denotes the vector of basis functions in $\mathcal{R}_n$. Consider this working model directly as a parametric model for the density itself with a truth $\beta_0$, we have a finite-dimensional parameter $\beta_n$, and $\beta_n \rightarrow^p \beta_0$. The density is then parameterized by $\beta_n$, and we denote $p_{f_n}(x) = p_{\beta_n}(x)$. Using the Delta method, we get the first-order Taylor expansion of $p_{f_n}(x)$ around $\beta_n$,
\[
p_{f_n}(x) = p_{\beta_n}(x) = p_{\beta_0}(x) + \nabla_\beta p_{\beta_n}(x) (\beta_n - \beta_0) + o_P(\Vert \beta_n - \beta_0 \Vert),
\]
where $\nabla_\beta p_{\beta_n}(x)$ is the gradient of $p_{\beta_n}(x)$ with respect to $\beta_n$.
Ignoring the second-order term, and taking the variance of both sides, we get
\[
\Var(p_{f_n}(x)) = \Var(p_{\beta_n}(x) - p_{\beta_0}(x)) \approx \Var(\nabla_\beta p_{\beta_n}(x) (\beta_n - \beta_0)).
\]



Denoting $\Var(p_{\beta_n}(x))$ as $\sigma^2_{n}(x)$, we obtain a plug‑in estimator of the variance \(\sigma^2_{n,0}(x)\) defined in Theorem~\ref{thm:density_HAL_asymptotic_normality}, as follows:
\begin{align} \label{eq:delta-method-variance-estimator}
\sigma_n^2(x)
=& (\nabla_\beta p_{\beta_n}(x))^\top
\; \widehat{\cov}(\beta_n)\;
(\nabla_\beta p_{\beta_n}(x)) \nonumber \\
=& (p_{\beta_n}(x))^2\,
\Bigl[S_{f_n}(\phi)(x)\Bigr]^\top 
\widehat{\cov}(\beta_n)\,
\Bigl[S_{f_n}(\phi)(x)\Bigr]. 
\end{align}
And the confidence interval can be constructed accordingly based on the $z$-statistic.
The above formula $\sigma^2_{n,0}$ in Theorem \ref{thm:density_HAL_asymptotic_normality} for the asymptotic variance relies on the orthonormal basis formulation, which makes it not practical to estimate it. This delta-method formula is a standard sandwich variance estimator for a function of the parameters of the working model. The covariance matrix of $\beta_n$ is estimated with the empirical covariance of the inverse of the information matrix applied to the score vector $(S_{f_n}(\phi)$, where $\phi \equiv (\phi_1, \cdots, \phi_{J_n})^\top$, corresponds to robust estimation of the covariance matrix of the MLE for a misspecified parametric working model.
In our case, note that this observed information matrix actually also equals the covariance of the scores under $P_{f_n}$:
$I_n(\beta_n)(j_1,j_2)=P_{f_n}S_{f_n}(\phi_{j_1})S_{f_n}(\phi_{j_2})$. Since our working model $D^{(k)}({\cal R}_n)$ will approximate the true $f_0$, we could also replace $I_n(\beta_n)$ by $I_{{n,emp}(\beta_n)}=P_n S_{f_n}(\phi)S_{f_n}(\phi)^{\top}$. We claim that the latter is more robust than the one using $I_n(\beta_n)$ and is thus recommended.


\subsubsection{\texorpdfstring{Covariance Estimation for $\beta_n$}{Covariance Estimation for beta\_n}} \label{covariance_beta}
Our delta-method variance estimator relies on an inverse of the information matrix, such as  its empirical version
\[
I_{n,emp}(\beta_n) = \frac{1}{n}\sum_{i=1}^{n} S_{f_n}(\phi)(x_i)\,S_{f_n}(\phi)(x_i)^\top.
\]

The inverse can become unstable when some of the basis functions are highly correlated. 
For numerical stability, we might add a ridge regularization to the diagonal of $I_n(\beta_n)$ before inverting, in case it is nearly singular. Following the uniform convergence theory of the HAL, the convergence rate is $\sqrt{\frac{J_{n,0}}{n}}$. Thus, the $\sigma_n^2(x)=O_p(\frac{J_{n,0}}{n})$, and this implies $I_n(\beta_n)^{-1}=O_p(J_{n,0})$. So, its minimum eigenvalue should be of order $\frac{1}{J_{n,0}}$.


For the first order HAL, $J_{n,0}\asymp^+(n^{\frac{1}{5}})$. Thus, we should add a regularization parameter at least smaller than $O_p(n^{-\frac{1}{5}})$. On the other hand, this delta-method variance estimator ignores the bias of the working model so that some overestimation of this variance would be welcome for better coverage. 
Overall, the robust selection of such a constant is an interesting problem to be addressed in future work. In our practical simulations, we added a regularization parameter of size  $O_p(n^{-\frac{1}{5}})$.  

\section{HAL-MLE for Pathwise Differentiable Statistical Estimands} \label{section:HAL-MLE for Pathwise Differentiable Statistical Estimands}
One could use the density estimator for statistical estimands such as moments, the median, the CDF, or the survival probability. Let $\Psi:{\cal M}\rightarrow \R$ be a pathwise differentiable target parameter with canonical gradient (also called efficient influence curve) $D^*_P$ at $P$ where ${\cal M}=\{P_f: f\in D^{(k)}_U([0,1]^d \}$, then, for a path $P_{\varepsilon}^h$ through $P$ with direction $h$, $\frac{\mathrm{d}}{\mathrm{d} \varepsilon} \Psi(P_{\varepsilon}^h) = \mathbb{E}_P[ D^*_P \, h]$. Let $R(P, P_0)\equiv \Psi(P)-\Psi(P_0)+P_0D^*_P$ be the exact remainder implied by the canonical gradient, which is known to be a second-order difference due to it being the exact remainder in a first order Tailor expansion $R(P,P_0)=\Psi(P)-\Psi(P_0)-(P-P_0)D^*_P$.
We have $\Psi(P)=\Psi(P_f)=\Psi^F(f)$ for a $\Psi^F: D^{(k)}_U([0,1]^d)\rightarrow \R$.
We can represent $D^*_P=D^*_f$ with $f=f(P)$.

\subsection{"Naive" Approach: Plug-in} \label{Plug-in HAL-MLE}
Plugging in the density would be a natural approach, and we will show that it yields an asymptotically efficient estimator, possibly achieved by using undersmoothing. We here provide a standard TMLE-type analysis establishing asymptotic efficiency under the assumption that $D^{(k)}({\cal R}_n)$ approximates $D^{(k)}_U([0,1])$. 

The tangent space generated by paths through $f$ is given by closure of linear span $\{S_f(\phi): \phi\in D^{(k)}_U([0,1])\}$. Thus, we can represent 
$D^*_f=S_f(\phi_f)$ for some $\phi_f\in D^{(k)}_M([0,1])$ for $\forall k$ including 0. Therefore, if the basis ${\cal R}_n$ is chosen so that $D^{(k)}({\cal R}_n)$  yields a $O^+(1/J_n^{k+1})$ $L^2$ approximation of $D^{(k)}_U([0,1])$, then we can find a projection $\tilde{\phi}_f$ of $\phi_f$ onto $D^{(k)}({\cal R}_n)$ that approximates $\phi_f$ in $L^2$-norm within distance $O^+(1/J_n^{k+1})$.

For simplicity, firstly, consider the case that $f_n$ is a relaxed HAL-MLE so that it equals an MLE over $D^{(k)}({\cal R}_n)$, while we assume that it has been established that $\Vert \beta_n\Vert=o_p(1)$ as it would be for the HAL-MLE. Then, $P_n S_{f_n}(\phi_j)=0$ for all $j\in {\cal R}_n$. Let $\tilde{D}_{f_n}$ be the projection in $L^2(\mu)$ of $D^*_{f_n}$ onto $\{S_{f_n}(\phi_j):j\in {\cal R}_n\}$. Then, $P_n \tilde{D}_{f_n}=0$. In that case if $\Vert \tilde{D}_{f_n}-D^*_{f_n}\Vert_{\mu}=o_p^+(1/J_n^{k+1})$, then it follows that $P_n D^*_{f_n}=o_p(n^{-1/2})$ and the efficiency theorem below  applies. For the HAL-MLE we already know we control $\Vert \beta_n\Vert_1=O(1)$, but we would need to show that the scores $P_n S_{f_n}(\phi_j)$, $j\in {\cal R}_n$ are solved good enough so that $P_n \tilde{D}_{f_n}=o_p(n^{-1/2})$. In the Appendix~\ref{app_f: Asymp Efficiency}, we show that this can generally be arranged by extending the approximation condition to the $L_1$ constrained score space.

\begin{theorem}[Efficient Influence Curve Approximation]
\label{thm:Efficient Influence Curve Approximation}
Suppose the sieve \(D^{(k)}(\mathcal R_n)\) satisfies the \(L^2\)‐approximation and it can be extended to its corresponding $L_1$ constrained score space, and maintain the Smoothness Assumptions in Theorem~\ref{thm:asymptotic-linearity-hal-mle}. 
Let \(f_n\) be the cross‐validated HAL–MLE over \(D^{(k)}(\mathcal R_n)\). Then
\[
P_n\,D^*_{f_n}
= o_p\!\bigl(n^{-1/2}\bigr).
\]
\end{theorem}

This directly leads to the asymptotic efficiency theorem. 
\begin{theorem}[Asymptotic Efficiency of Plug‐in HAL‐MLE]
\label{thm:plugin_HAL_asymptotic_efficiency}
Suppose the following conditions hold:
\begin{enumerate}
  \item (L\(^2\)‐Approximation) \(D^{(k)}(\mathcal R_n)\) satisfies the \(L^2\)‐approximation and it can be extended to its corresponding $L_1$ constrained score space.
  \item (Negligible Remainder) \(R(P_{f_n},P_0)=o_p\bigl(n^{-1/2}\bigr)\).
  \item (\(L^2(P_0)\)-continuity of EIC) \(P_0\{(D^*_{f_n}-D^*_{f_0})^2\}
      = o_p\bigl(1\bigr)\).
\end{enumerate}
Then the plug‐in estimator \(\Psi^F(f_n)\) satisfies
\[
\Psi^F(f_n) - \Psi^F(f_0)
= (P_n - P_0)\,D^*_{f_0} \;+\; o_p\!\bigl(n^{-1/2}\bigr),
\]
and is therefore asymptotically efficient.
\end{theorem}

We discussed the plausibility of the first assumption in Section~\ref{subsect: Pointwise Asymptotic Normality of the HAL-MLE}. And here we demonstrate that the second and third assumptions are not only reasonable but also very weak with the examples of survival probability and moments. We leave the argument of median, which is slightly complicated, to Appendix~\ref{app_f: Asymp Efficiency}. 
For survival probability at $x_0$, we have 
\[P_0{(D^*_{f_n} - D^*_{f_0})} = \int I(x>x_0)(p_{f_n} - p_{f_0}) \mathrm{d}x =O(1)o_p(\Vert p_{f_n} - p_{f_0} \Vert_{\mu}) = o_p(n^{-\frac{k+1}{2k+3}}),
\]
by the Cauchy-Schwartz inequality. This implies \(P_0\{(D^*_{f_n}-D^*_{f_0})^2\} =o_p(n^{-\frac{2k+2}{2k+3}}) = o_p\bigl(1\bigr)\). 
\[R(P_{f_n}, P_0) \equiv \Psi(P_{f_n}) - \Psi(P_0) + P_0 D_P = S_{f_n}(x_0) - S_0(x_0) + P_0 (I(x >x_0) - P_{f_n}(x>x_0)) = 0.
\] 
The argument is basically the same for moments, \[P_0{(D^*_{f_n} - D^*_{f_0})} =\int x^k (p_{f_n} - p_{f_0}) \mathrm{d}x = o_p(n^{-\frac{k+1}{2k+3}}),\] by Cauchy-Schwartz inequality. \[R(P_{f_n}, P_0) = \mathbb{E}_{P_{f_n}}(x^k) - \mathbb{E}_{P_{0}}(x^k) + P_0 (x^k - \mathbb{E}_{P_{f_n}}(x^k)) = 0.\]

\subsection{HAL-TMLE} \label{HAL-TMLE}
We can exploit the information about the target parameter more effectively by incorporating the TMLE framework, namely HAL-TMLE. We found that the targeting step is very compatible with the link function, and each target parameter corresponds to a different targeting. 

\textbf{Targeting Step:} We construct a one-dimensional submodel $\{p_{f_{\varepsilon}^h}:\varepsilon\}$ indexed by a direction $h$ of a path $\{f_{\varepsilon}^h:\varepsilon\}$ through $f$ that satisfies
\begin{align} \label{lfm}
  \left  .  \frac{\mathrm{d}}{\mathrm{d}\varepsilon} \log p_{f_{\varepsilon}^h}\right |_{\varepsilon=0} = D^*_f(x),
\end{align}

\textbf{MLE Step:} Specifically, we select $f_{\varepsilon}^h = f + \varepsilon h$, and thus we have
\[
p_{f_{\varepsilon}^h} = \frac{\exp\bigl(f + \varepsilon h\bigr)}{\displaystyle \int \exp\bigl(f + \varepsilon h\bigr)\mathrm{d}x},
\]
where $f$ plays the role of $f_{n, \beta_n(M)}$. 
We can simply select $h=h_f=D^*_f$ to make this path a local least favorable path (LLFP). Let's denote the resulting path with $f_{n,\varepsilon}$
Then, the first step TMLE update would be defined by 
$f_{n,\varepsilon_n^1}$ with $\varepsilon_n^1=\arg\min_{\varepsilon}-P_n \log p_{f_{n,\varepsilon}}$. One can set $f_n^1=f_{n,\varepsilon_n^1}$ and define $f_{n,\varepsilon}^1$ as this LLFP but now with off-set $f_n^1$, and define $\varepsilon_n^2=\arg\min_{\varepsilon}-P_n \log p_{f_{n,\varepsilon_n^2}^1}$. We can iterate this TMLE-updating till $P_n D^*_{f_n^k}\leq \sigma_n/(n^{1/2}\log n)$, where
$\sigma_n^2 = (P_n D^*_{f_n})^2$ is the estimator of the sample variance of the efficient influence curve at the initial estimator. The final update satisfying this criterion is then the iterative TMLE $f_n^*$, implying the plug-in TMLE $\Psi^F(f_n^*)$ of $\Psi^F(f_0)$. 

A theoretical concern in iterative TMLE is the preservation of the function class.
The HAL-MLE $f_n$ belongs to a function space whose induced distribution $P_{f_n}^0$ lies in a 
Donsker class with covering number equivalent to that of $D^{(k)}([0,1])$. When applying a finite number of targeting steps, this property is preserved. However, performing infinitely many TMLE iterations may drive the estimator outside of this Donsker class.

According to \citet{van2017generally}, a single targeting step already suffices, because the initial HAL-MLE $f_n$ converges to $f_0$ at a rate faster than $n^{-1/4}$ and we also satisfy the positivity condition, the single-step update $f_n^1$ obtained by fluctuating along the least favorable submodel satisfie 
\[
P_n D^*_{f_n^1} = o_p(n^{-1/2}).
\]
Consequently, $f_n^1$ also attains the same asymptotic efficiency, while preserving the convergence rate and sectional variation norm of the initial HAL-MLE. One could also construct a universal least favorable path as in \citet{van2016one} and Chapter 5 in \citet{van2018targeted}, which guarantees that the single step TMLE update already solves $P_n D^*_{f_n^1}=0$. We formally state the theorem as follows.

\begin{theorem} [Asymptotic Efficiency of Single-step HAL‐TMLE] \label{thm:single-step-hal-tmle}
Suppose we impose a single targeting step to $f_n$ and yield $f_n^1$. And suppose that the second and third assumptions hold for $f_n^1$, that is,
\begin{enumerate}
\item (Negligible Remainder) \(R(P_{f_n^1},P_0)=o_p\bigl(n^{-1/2}\bigr)\).
\item (\(L^2(P_0)\)-continuity of EIC) \(P_0\{(D^*_{f_n^1}-D^*_{f_0})^2\}
      = o_p\bigl(1\bigr)\).
\end{enumerate}
Then $P_n D^*_{f_n^1} = o_p(n^{-1/2})$, and thus the single-step HAL-TMLE is asymptotically efficient.
\end{theorem}

The following are some examples of statistical estimands for which we can compute this TMLE:

\begin{example}[Survival function and CDF]
For the marginal survival function $\Psi^F(f)=\int_{x_0}^1 p_f(y)dy$ at $x_0$, we have \(D^*_f(x) = I(x > x_0) - \mathbb{E}_{P_{f}}\bigl[I(x > x_0)\bigr].\) We can thus select $h = I(x > x_0)$. Note that, technically, we need $\mathbb{E}_{P_{f_n}}[h] = 0$, but the constant will cancel out by the normalizing constant. Similarly, for CDF $F_f$ at $x_0$ we can select $h=I(x\leq x_0)$.
\end{example}

\begin{example}[Median]
For the median $\Psi^F(f)=F_f^{-1}(0.5)$ we have 
\(D^*_f(x) = \frac{\frac12 - I(x < F_f^{-1}(0.5))}{f(F_f^{-1}(0.5)) },\) and thus $h = I(X< F_f^{-1}(0.5))$. We need to estimate the median based on the initial estimation and then do the targeting. Notice that this formula of the median is trivially generalized to any percentile. Since the density estimation changes from $f_n$ to $f_n^k$ after the $k$ steps targeting, we can iteratively update the median. In Theorem~\ref{thm:single-step-hal-tmle}, we show that a single step targeting already guarantees $P_n D^*_{f_n^1} = o_p(n^{-1/2})$, and thus the asymptotic efficiency holds.
\end{example}

\begin{example}[$k$-th order moments]
For the $k$-th order moment $\Psi^F(f)=\int x^k p_f(x)\mathrm{d}x$ we have
\(D^*_f(x) = x^k - \Psi^F(f),\) and thus $h = x^k$. Notice that this is just the parametric basis. This implies a motivation for not penalizing the parametric part in the $m$-th order HAL, like LAS and the classic TF, in the sense that the resulting HAL fit would be targeting all its k-th order moments, $k=1,\ldots,m$.\\
\end{example}


Notice that the LLFPs for survival probability and moments are actually their universal least favorable paths (ULFP), and the single-step HAL-TMLE in this case is equivalent to one-step HAL-TMLE in \citet{van2017generally}, and thus we have $P_n D^*_{f_n^1} = 0$ exactly. Also, by \citet{van2017generally}, $f_n^1$ and $f_n$ preserve the same convergence rate to $f_0$. Thus, we can extend the arguments of the assumptions of the plug-in HAL-MLE to the single-step HAL-TMLE directly. 

\subsection{Theoretical Property and Variance Estimation}
\label{subsec: EIC-based-variance-estimation}
As established in Section~\ref{Plug-in HAL-MLE} and Section~\ref{HAL-TMLE}, the score equation is solved at least at level $o_p(n^{-1/2})$. Thus, both estimators are asymptotically efficient and asymptotically linear, with their proofs unified in Appendix~\ref{app_f: Asymp Efficiency}. Since TMLE is asymptotically linear with influence curve $D^*_{f_n^1}$, variance estimation for $\Psi^F(f_n)$ can be obtained directly by computing the sample variance of the estimated $D^*_{f_n^1}(x_i)$.
\section{Optimization Algorithm} \label{sct: opt_alg}
Recall that we are optimizing the problem with the following form:
\begin{align}\label{constraint opt}
   \beta_n(M) = \arg\max_{\beta:\ \|\beta\|_1 \leq M} \sum_{i=1}^{n} \sum_{j=1}^{n+k+1} \beta_j \phi_j(x_i) - n\log C(\beta), 
\end{align} 
and its Lagrangian form is as follows:
\begin{align} \label{lag opt}
   \beta_n(M) = \arg\max_{\beta} \sum_{i=1}^{n} \sum_{j=1}^{n+k+1} \beta_j \phi_j(x_i) - n\log C(\beta) - \lambda \|\beta\|_1, 
\end{align} 

Tuning the upper bound $M$ of the $L_1$ norm in the constraint form is more statistically meaningful, because the final choice of $M$ should reflect the magnitude of the SVN of the truth. However, we prefer the Lagrangian form because the algorithm for the constraint form always involves projection to the $L_1$ ball. In practice, we find that soft thresholding provides a more robust dimension reduction through iterations than projection to the $L_1$ ball.

\subsection{Candidate Algorithms} \label{subsct: cand_alg}
While generic convex solvers such as \texttt{MOSEK}, \texttt{ECOS}, or \texttt{SCS} within the \texttt{CVXPY} package can, in principle, handle our convex optimization problem, their general strategy is to map any instance into a conic programming framework. This approach, though flexible, is not tailored to the structural features of our estimator. Moreover, as the sample size increases to the range of 3,000–5,000 observations, we find that the default knot point selection of these solvers becomes less effective: a large number of unnecessary knot points remain active, primarily due to accuracy constraints in the optimization process. A demonstration of this is shown in Figure~\ref{fig:opt-algorithms-tn}. These limitations motivate the development of specialized algorithms that exploit the structure of the problem and incorporate effective knot point selection strategies, yielding both improved scalability and enhanced statistical efficiency. 

We refer to the matrix of initiated basis evaluated at the data points as the basis matrix. The basis matrix is an $n \times (n+k+1)$ matrix. 
\paragraph{Standard Algorithm:} When using the truncated power basis, the Hessian constructed from the basis matrix is usually ill-conditioned and thus first-order methods like Fast Iterative Shrinkage-Thresholding Algorithm (FISTA) would converge slowly in iterations. Therefore, we propose a standard algorithm combining proximal Newton's method and the coordinate descent algorithm. An advantage of proximal Newton is that we can calculate the Hessian in closed form. Also, we could warm start the algorithm when using cross-validation to choose $\lambda$ in order to keep Newton's method within the quadratic convergence range. 

\paragraph{Large-scale Adjustment:} The proximal Newton's method fails for large-scale problems due to memory issues. When the sample size n is too large, such as 10,000 or 100,000, the initialized Hessian will be $O(n^2)$ and too large to be memorized. For this problem, we propose the Quasi-Newton method (L-BFGS) and coordinate descent. An easy implementation is just memorizing the most recent pair of differences in parameters and their corresponding gradient. Thus, the approximation of the Hessian is a scaled identity matrix, and we update the scale at each iteration. The memory reduces from $O(n^2)$ to $O(n)$. Pragmatically, we find out that including the coordinate descent is more robust for large-scale problems.

Another algorithm we propose is a variant of AdaGrad using the inner product defined in Equation~\ref{eq:inner-product-score}. We just compute the diagonal entries of the Hessian for each iteration. Thus, the cost per iteration becomes very cheap, and the memory is just $O(n)$ and totally affordable. 

\begin{remark}
The development of the large-scale methods seems to be unnecessary in the univariate case, because, in HAL theory like Appendix~\ref{appB: $L^2$ Convergence Analysis} and Appendix~\ref{appendix_E: Uniform Convergence Rate}, we only need about $O^+(n^{\frac{1}{2k+3}})$ knots, and we could choose them based on the percentile of the original sample while remaining the performance. However, in the multivariate case, we need such algorithms because (1) the definition of multivariate percentile could be meaningless and (2) even a small sample size could initialize a large number of parameters $n(2^d - 1)$ for $ d$-dimensional data.
\end{remark}

\subsection{Exploration of Optimization Algorithms} \label{subsct: exploration_of_optimization_algorithms}
We study the performance of the optimization algorithms in some fixed settings where \texttt{CVXPY} solvers have trouble selecting knot points. We here demonstrate the performance of the optimization algorithms on the Truncated Normal Data Generating Process (DGP), and the visualizations are shown in Figure~\ref{fig:opt-algorithms-tn}. We use the final result of \texttt{CVXPY} solutions as the baseline, and we evaluate the performance of the optimization algorithms by comparing the knot selection and loss convergence per iteration and per FLOP. The FLOPs are estimated by the computational complexity of the algorithms, which is shown in Table~\ref{tab:complexity-summary}. Additional results for all six DGPs and the reasoning of the computational complexity based on our implementation is provided in Appendix~\ref{app:knot-selection}.

\begin{table}[H]
   \centering
   \begin{tabular}{ll}
   \toprule
   Algorithm & Dominant per-iteration complexity \\
   \midrule
   FISTA & $\Theta\big((n + G)\,K\big)$ \\
   Proximal AdaGrad & $\Theta\big((n + G)\,K\big)$ \\
   Proximal Newton L\,-\,BFGS & $\Theta\big((1+L_k)\,(n + G)\,K\big)$ \\
   Proximal Newton (full) & $\Theta\big(G\,K^2 + s\,K^2\big) + \Theta\big((1+L_k)\,(n + G)\,K\big)$ \\
   \bottomrule
   \end{tabular}
   \caption{Dominant per\,-\,iteration computational complexity. Here $n$ is the number of samples, $G$ the number of midpoint intervals used for normalization, $K$ the number of basis coefficients, $s$ the active set size in coordinate descent, and $L_k$ the number of line\,-\,search objective evaluations at iteration $k$.}
   \label{tab:complexity-summary}
   \end{table}

\begin{figure}[H]
\centering
\includegraphics[width=.36\textwidth]{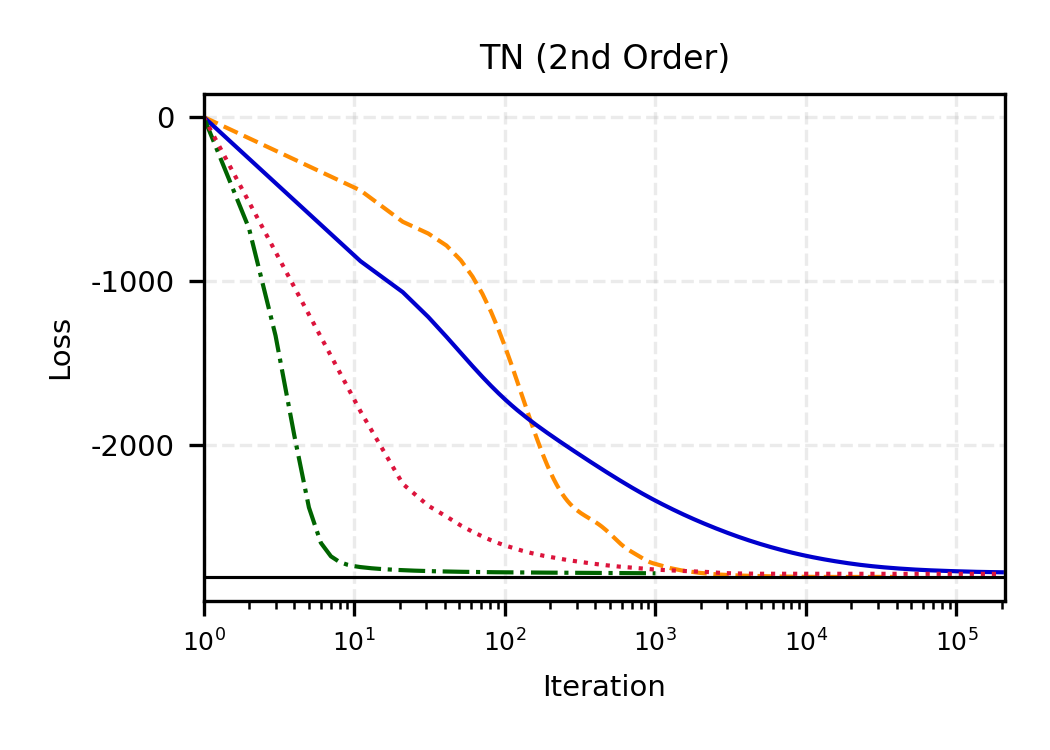}
\includegraphics[width=.36\textwidth]{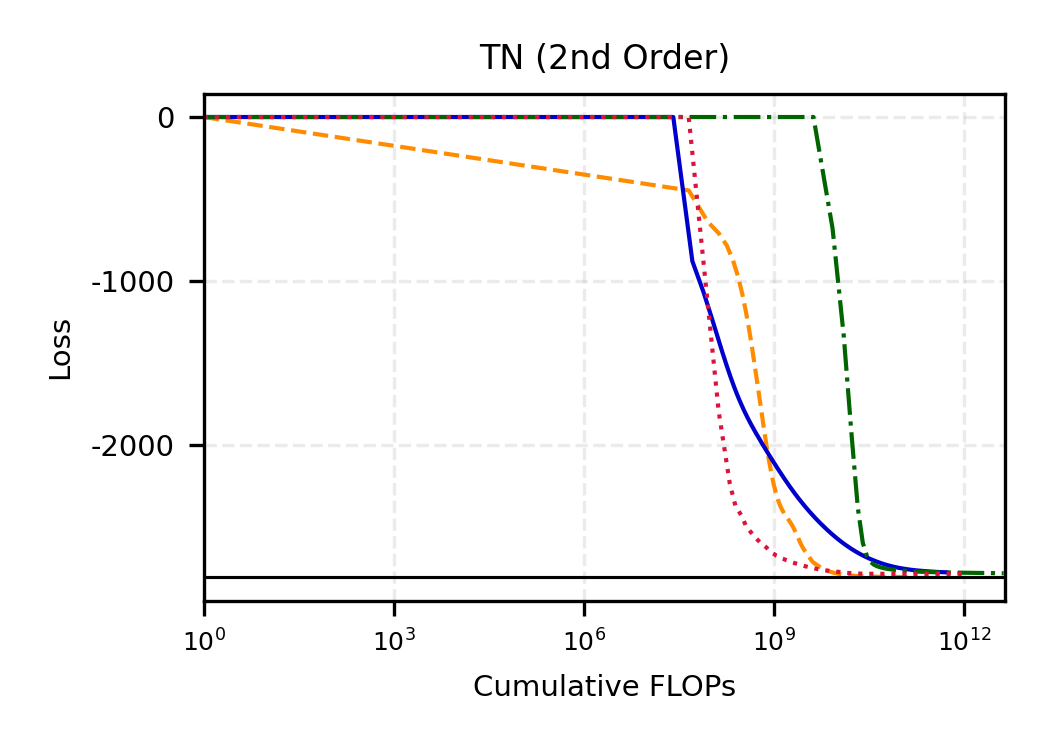}\\
\includegraphics[width=.36\textwidth]{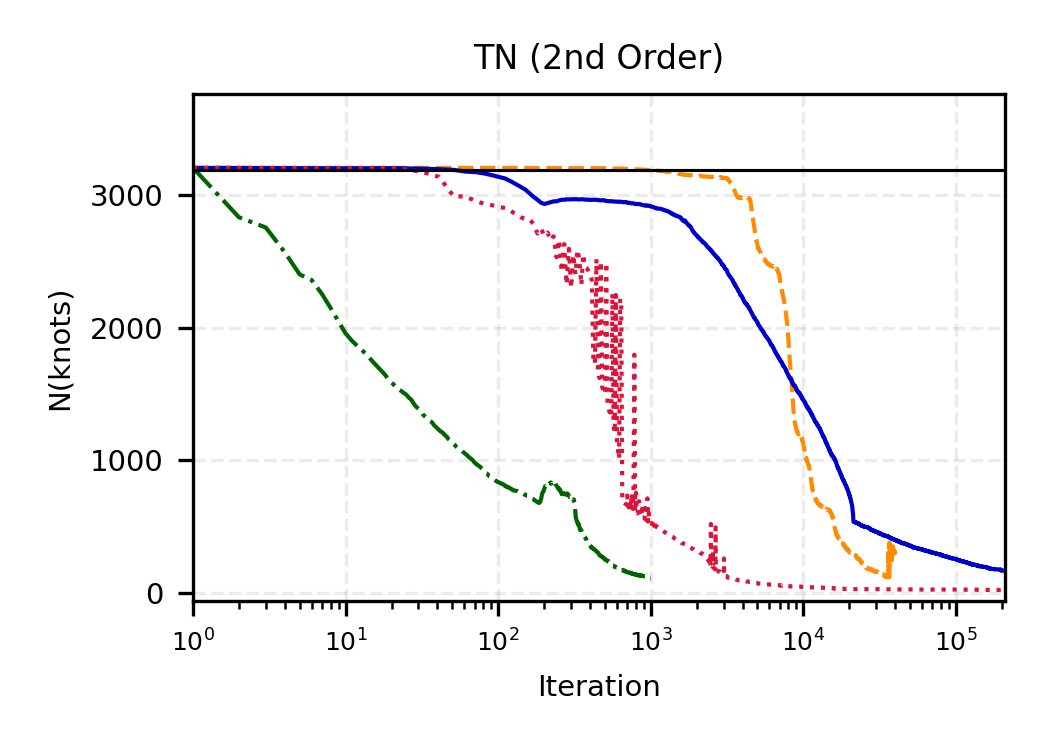}
\includegraphics[width=.36\textwidth]{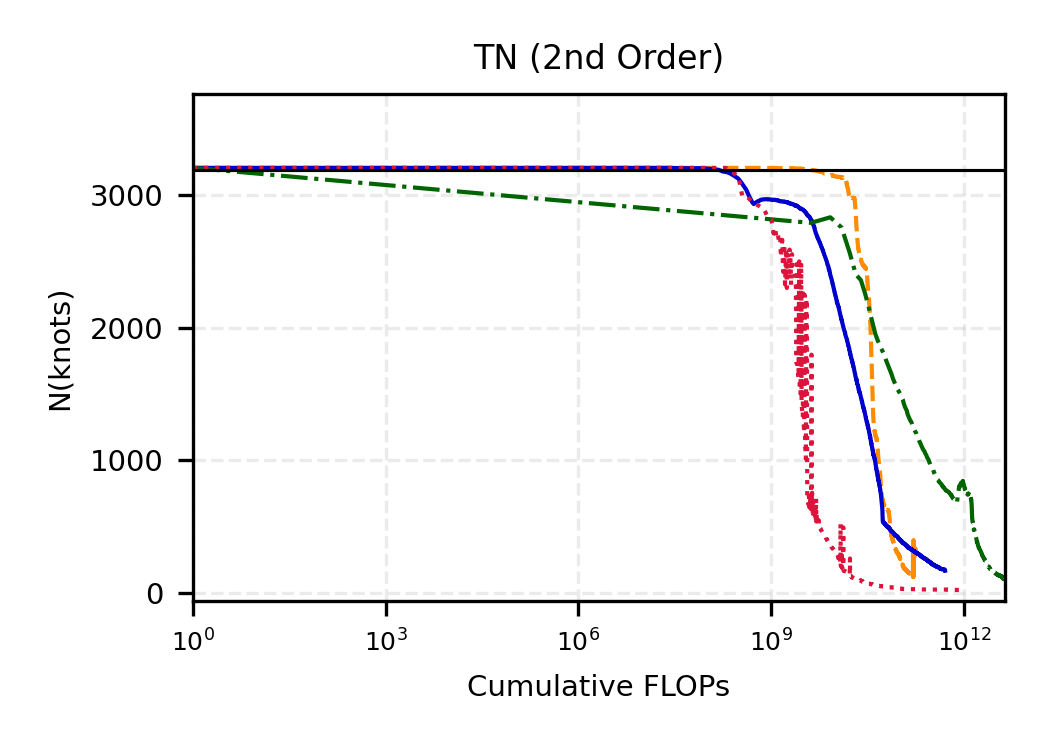}\\
\vspace{-0.5em}
\includegraphics[width=.72\textwidth]{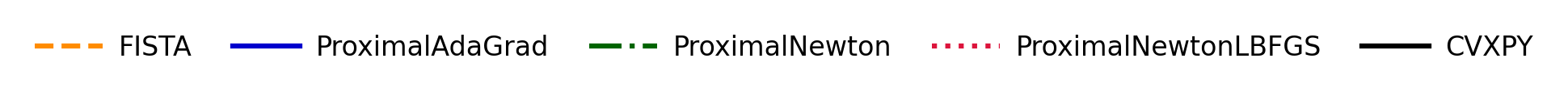}
\caption{Optimization algorithm comparison for Truncated Normal (2nd order basis): knot selection per iteration (top left), knot selection per FLOP (top right), loss convergence per iteration (bottom left), and loss convergence per FLOP (bottom right).}
\label{fig:opt-algorithms-tn}
\end{figure}

Based on the results, we find that the Proximal Newton algorithm takes the fewest iterations to converge, and it is also efficient in selecting knot points, which is expected because it is a second-order method and we have the closed-form Hessian. The Proximal Newton L-BFGS algorithm is also efficient in selecting knot points, and it is more efficient than the Proximal Newton algorithm in terms of FLOPs. However, the Proximal Newton L-BFGS algorithm is sometimes unstable in the process of knot selection. AdaGrad is very fast and efficient, but in practice, we find that it has oscillatory behavior as well. FISTA is easy to implement and robust, but it is slow as expected in terms of iterations and wall-clock time. 

Generally speaking, comparing to the \texttt{CVXPY} solvers, the optimization algorithms we proposed succeed in selecting knot points and converging to the global optimum. This aligns with the HAL-MLE theory that we only need $O^+(n^{\frac{1}{2k+3}})$ knots to achieve the desired performance, and this also informs us that the numerical problem of \texttt{CVXPY} solvers is severe when the sample size is large. However, we cannot make a fair comparison between the optimization algorithms and the \texttt{CVXPY} solvers in terms of wall-clock time because some of the \texttt{CVXPY} solvers are closed-source and some are coded in C, whereas we coded our own solvers in Python.

\begin{remark}[Computational comparison with B-splines]
\label{rem:comp_comparison}
It is well known that locally supported B-splines can be more computationally efficient, with their primary advantage stemming from the banded structure of the Hessian matrix. However, this benefit diminishes a lot in large-scale problems where the full Hessian cannot be computed explicitly and must instead be approximated. In addition, for the large-scale problem, the jumping matrix associated with the variation penalty is difficult to track when knot points are chosen in a data-adaptive manner, especially with smoothness order $k \geq 1$. By contrast, the truncated power basis—though less efficient computationally—offers the simplest and most transparent representation of sectional variation, making it particularly well suited for our theoretical analysis and algorithmic development.
\end{remark}
\section{Simulation}\label{sec:simulation}
This section documents the Monte Carlo simulation results for the HAL-MLE. We (i) test the uniform convergence and pointwise asymptotic normality of the HAL-MLE, (ii) test the asymptotic efficiency of the plug-in HAL-MLE and HAL-TMLE, and (iii) compare it with the logspline, TF and its equivalent variant with penalty on the parametric part (TFPP), and KDE methods, over a vast range of different DGPs and sample sizes.

\subsection{DGPs}
We use six univariate densities on $[0,1]$ across continuous vs discrete, concentrated vs heavy-tailed, and smooth vs oscillatory (single-scale vs multi-scale). The six DGPs are: Truncated Normal (TN), Truncated GMM with three equally spaced modes (GS3), Truncated GMM with three modes of differing scales and weights (GA3), Truncated GMM with five spikes and background (GS5), Step Function (Step), and Sinusoidal (Sine), as shown in Figure~\ref{fig:dgps}.

Each ground–truth density $p_0$ and population parameters, like moments, median, etc., are available in closed form or can be easily computed. We provide the exact parameterizations in Appendix~\ref{app:dgp_setup}.

\begin{figure}[H]
    \centering
    \includegraphics[width=0.9\textwidth]{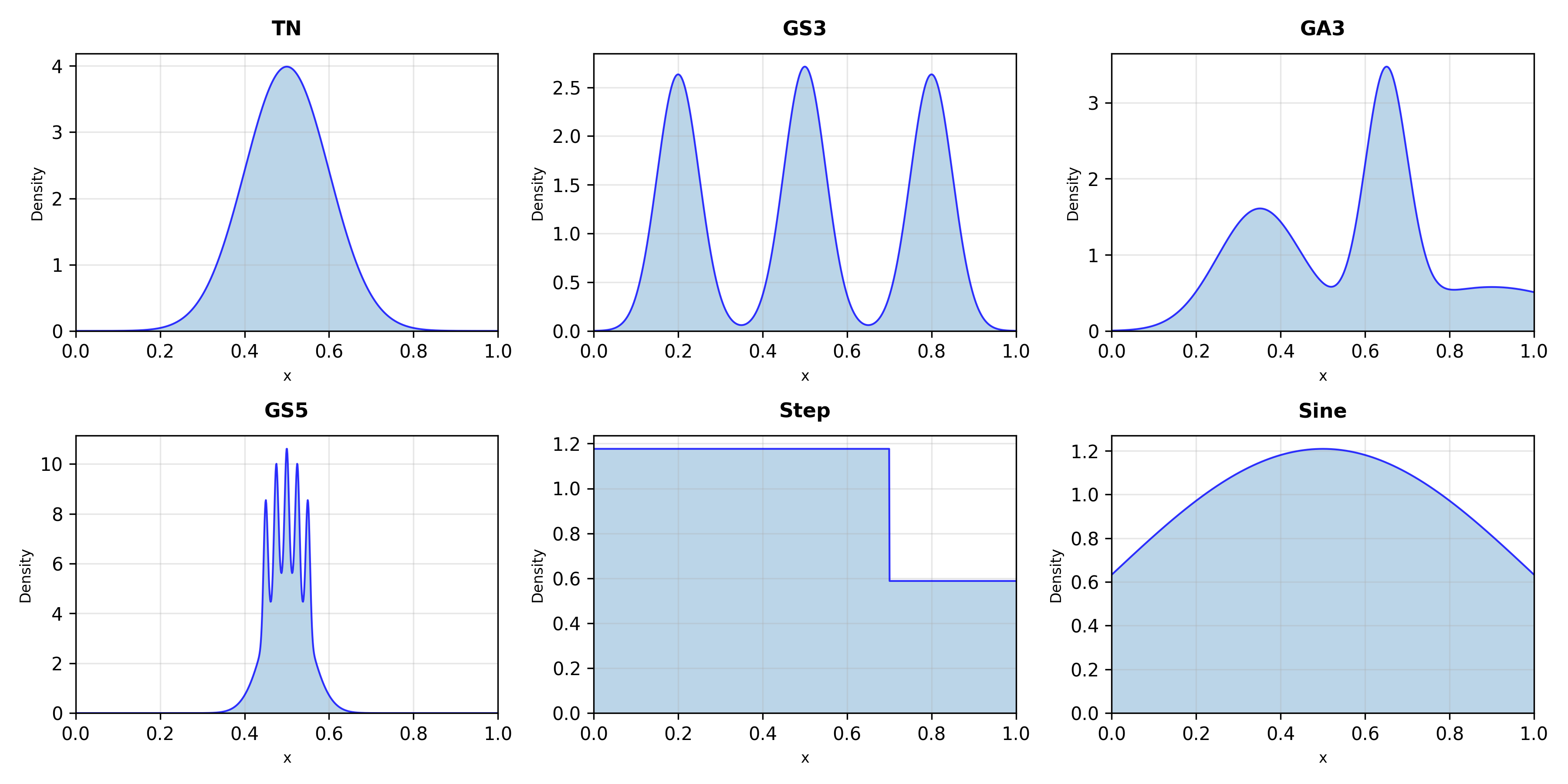}
    \caption{Six DGPs: (a) TN, (b) GS3, (c) GA3, (d) GS5, (e) Step, (f) Sine.}
    \label{fig:dgps}
\end{figure}

\subsection{Implementation of Cross-Validation with Optuna}
We use a 5-fold cross-validation to cross validate over the $L_1$ norm constraint and the basis order of the HAL-MLE. Instead of cross validating over each grid of the hyperparameter pairs, we used Optuna \citep{akiba2019optuna} to sample the hyperparameters with 50 steps. Throughout the cross-validation steps, we acknowledge the failure of the solvers occasionally and we have to switch to other solvers in \texttt{CVXPY}. In general, \texttt{MOSEK} is the fastest solver, and we use it as the default solver. \texttt{SCS} is relatively slow but more robust. The search space is $L_1$ norm constraint in $[10^{-3},10^{6}]$ (log scale) and basis order in $\{0,1,2\}$. The trial budget is 50 evaluations per estimator/DGP/sample size. As claimed in \citep{akiba2019optuna}, the performance of Optuna is very close to the performance of the grid search with a massive reduction in computational cost.

\subsection{HAL-MLE density estimation}
We here demonstrate the uniform convergence and pointwise asymptotic normality of the HAL-MLE.

\subsubsection{Uniform convergence (sup\textendash norm error).} \label{subsec:simulation_uniform_convergence}
We consider sample sizes $n \in \{25,50,100,200,400,800,1600,3200\}$ with $1000$ independent random samples per DGP. For each sample, we approximate the sup\textendash norm error $\lVert \hat p - p_0 \rVert_{\infty}$ by evaluating both the estimator and the truth on an evenly spaced grid of 201 points over $[0,1]$ and taking the maximum absolute deviation over grid points. By Theorem~\ref{thm:uniform-density-hal}, the uniform convergence rate is $O_p^+\big(n^{-(k+1)/(2k+3)}\big)$ for basis order $k\in\{0,1,2\}$. Because $k$ is selected by cross-validation over this set, we expect at least $O_p\big(n^{-1/3}\big)$ decay, even though we do not have the uniform convergence rate for $k = 0$. This can be read off from the median lines of the boxplots below as $n$ increases. If there is a convergence rate, the median line should be approximately a straight line with a decreasing slope, expected to be around or above $-\frac13$. 

We aggregate errors across samples and summarize their distributions by sample size using boxplots (both linear\textendash and log\textendash scale). The results for each DGPs are shown below in Figure~\ref{fig:uniform_convergence}.

\begin{figure}[H]
    \centering
    \includegraphics[width=.32\textwidth]{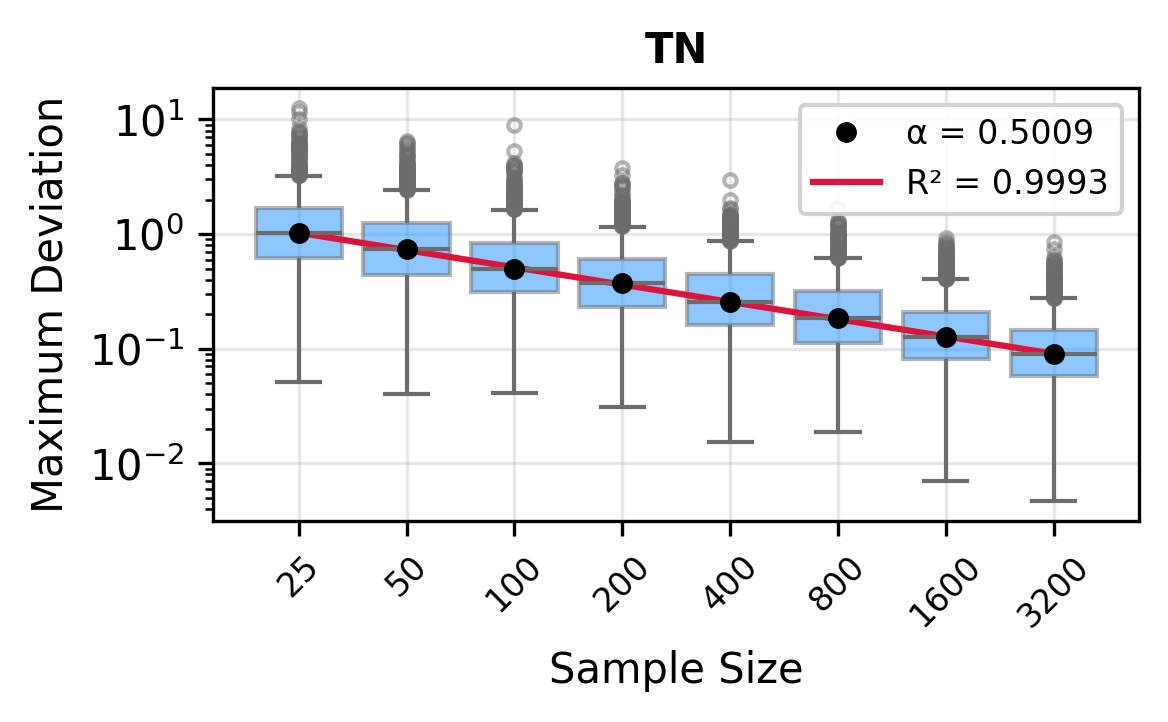}
    \hfill
    \includegraphics[width=.32\textwidth]{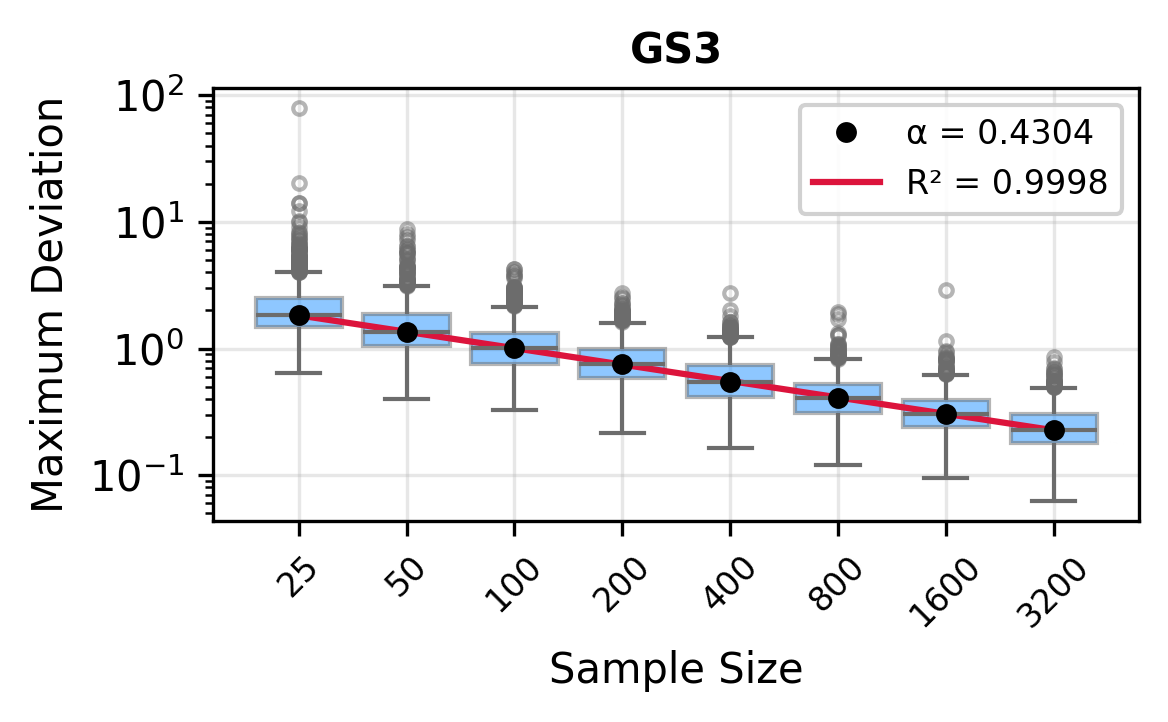}
    \hfill
    \includegraphics[width=.32\textwidth]{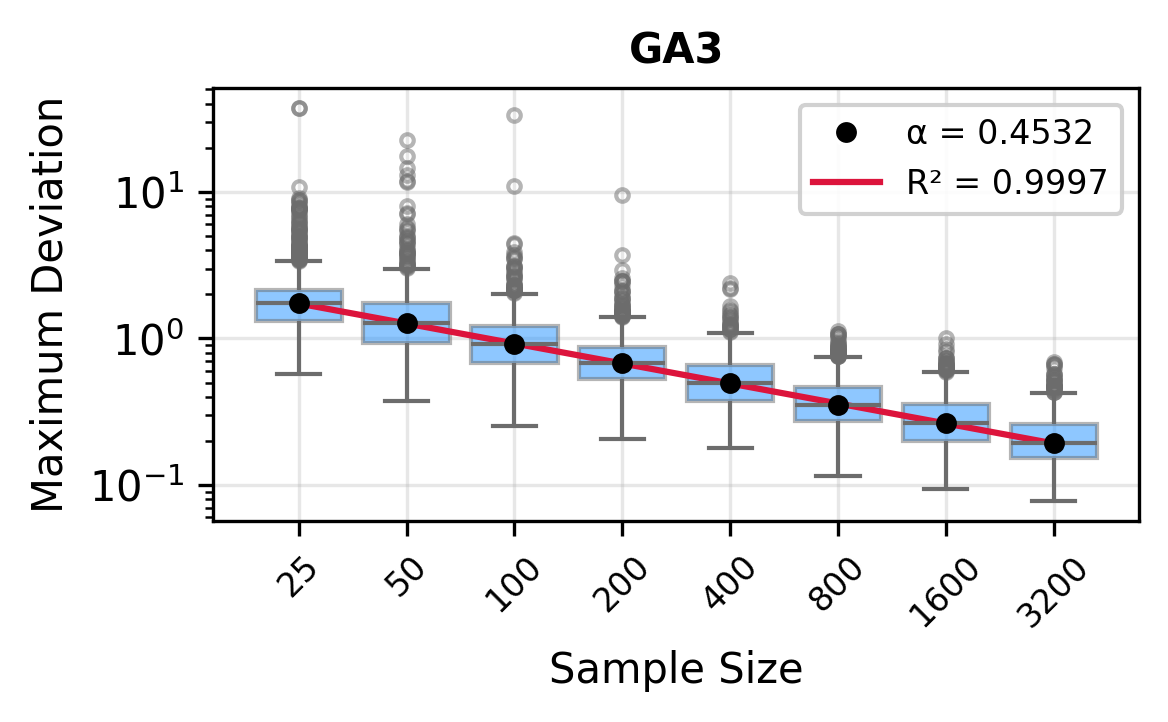}

    \vspace{0.4em}

    \includegraphics[width=.32\textwidth]{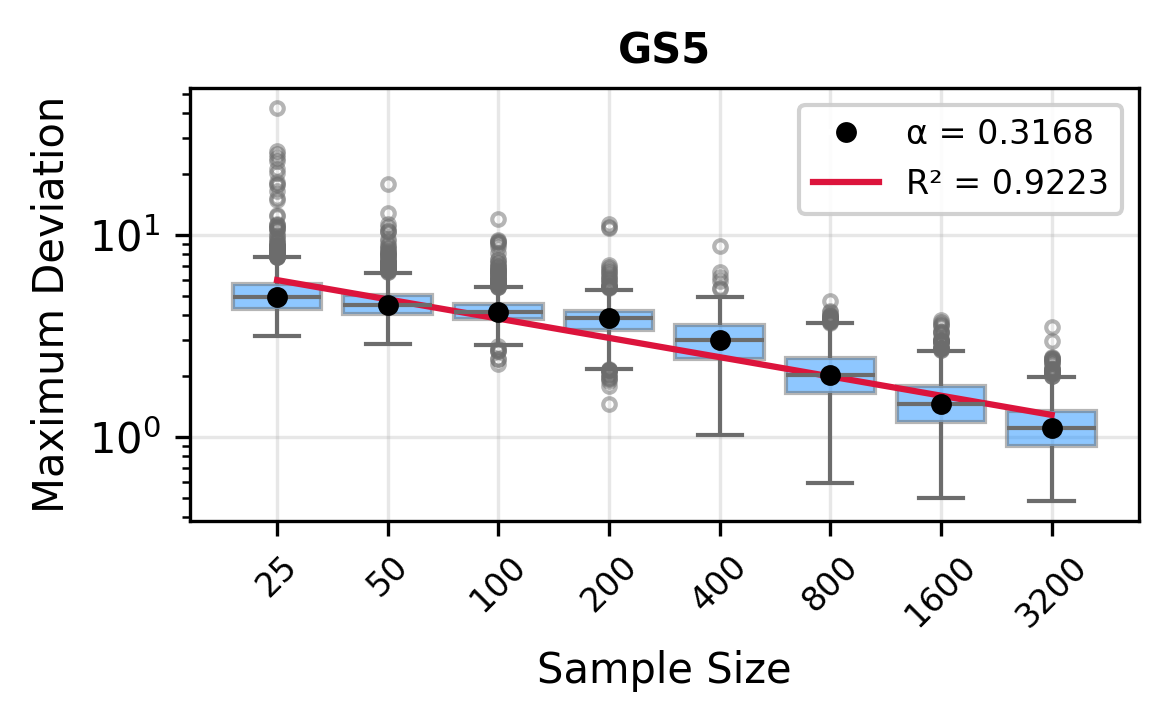}
    \hfill
    \includegraphics[width=.32\textwidth]{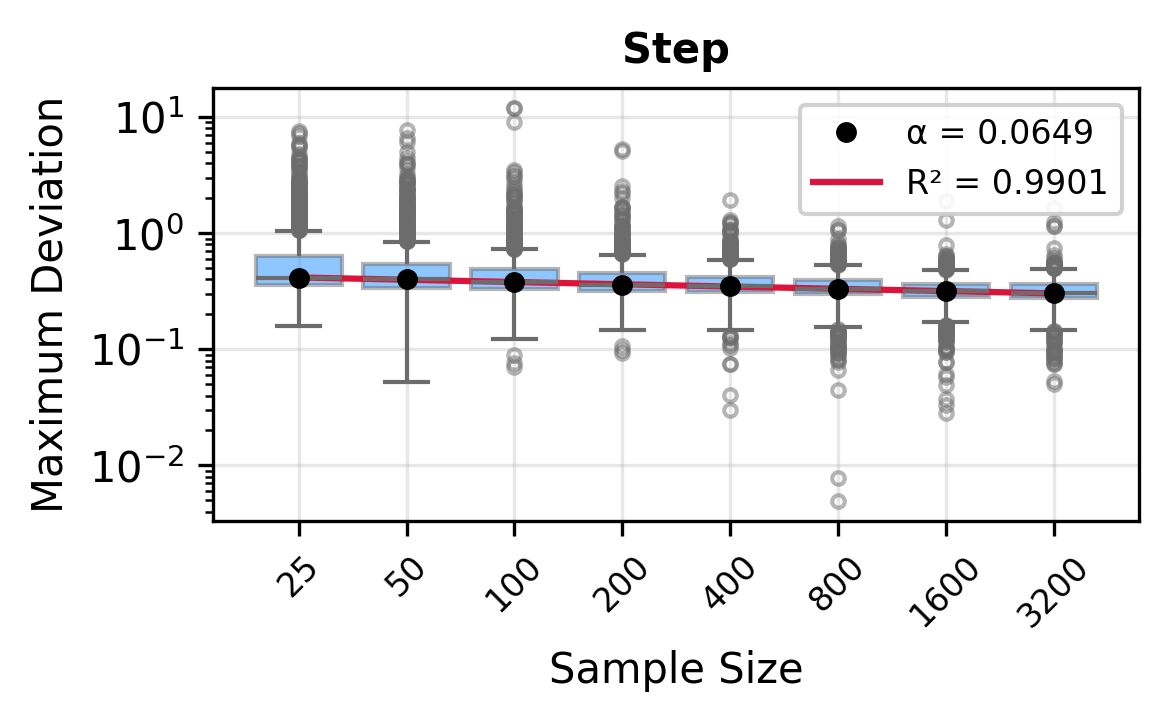}
    \hfill
    \includegraphics[width=.32\textwidth]{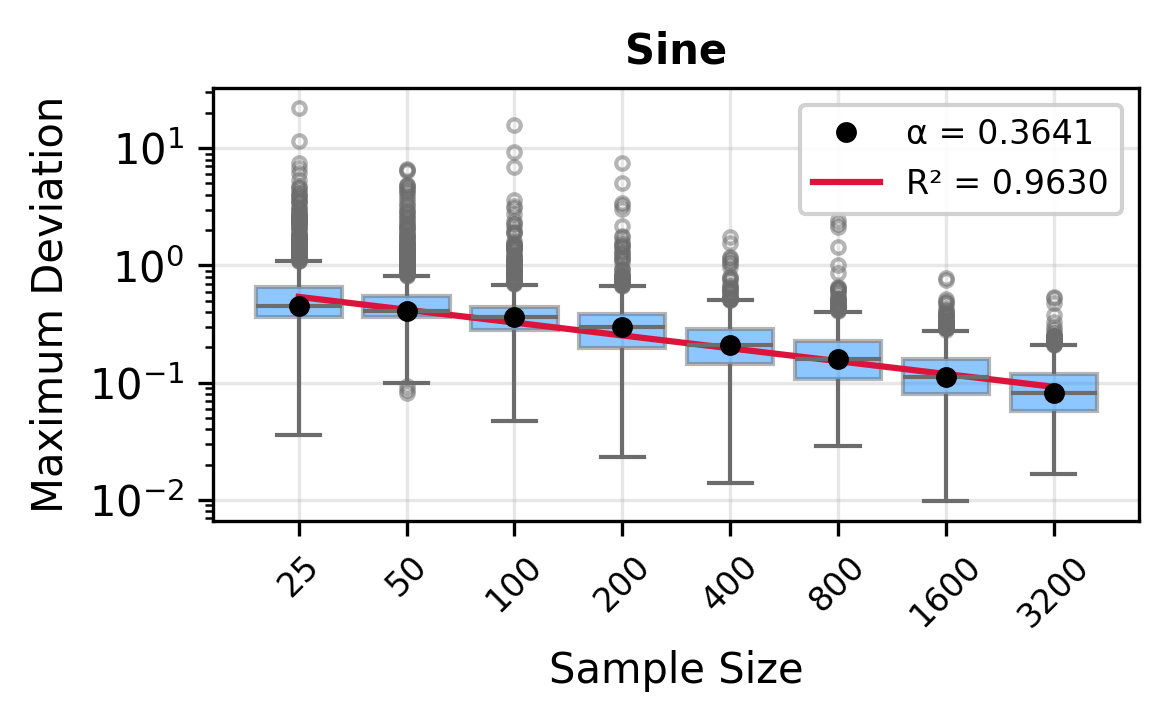}
    \caption{Uniform\textendash convergence sup\textendash norm error (log\textendash scale) by DGP. Panels (row\textendash wise, left to right): (a) TN, (b) GS3, (c) GA3, (d) GS5, (e) Step, (f) Sine.}
    \label{fig:uniform_convergence}
\end{figure}

We regress the median error on the sample size on a log-log scale and report the slope and $R^2$ in Table~\ref{tab:uniform_convergence_scaling}. The results are shown below in Table~\ref{tab:uniform_convergence_scaling}.

\begin{table}[H]
  \centering
  \small
  \setlength{\tabcolsep}{4pt}
  \renewcommand{\arraystretch}{1.2}
  \begin{tabular}{lcccccc}
    \toprule
    Metric & TN & GS3 & GA3 & GS5 & Step & Sine \\
    \midrule
    $\alpha$ (Decay Exponent) & 0.5009 & 0.4304 & 0.4532 & 0.3168 & 0.0649 & 0.3641 \\
    $R^2$ (Goodness of Fit) & 0.9993 & 0.9998 & 0.9997 & 0.9223 & 0.9901 & 0.9630 \\
    Significant ($p < 0.05$) & Yes & Yes & Yes & Yes & Yes & Yes \\
    \bottomrule
  \end{tabular}
  \caption{Uniform convergence scaling analysis. $\alpha$ represents the decay exponent in the power law $\text{error} \propto n^{-\alpha}$. $R^2$ measures goodness of fit on the log-log scale. Significance tests $H_0: \text{slope} = 0$ vs $H_1: \text{slope} \neq 0$ at $\alpha = 0.05$.}
  \label{tab:uniform_convergence_scaling}
\end{table}

This Table~\ref{tab:uniform_convergence_scaling} verifies that the adaptive HAL-MLE that selects the smoothness $k$ with cross-validation is converging at an adaptive rate. And the adaptive rates we are seeing align with the  smoothness of the DGPs. The Step Function might be too easy and does not have a huge space for improvement.

\subsubsection{Asymptotic normality: pointwise SE, CI width, and coverage.} \label{subsec:simulation_asymptotic_normality}
We assess the pointwise asymptotic normality of the HAL-MLE density mentioned in Theorem~\ref{thm:density_HAL_asymptotic_normality} and evaluate the performance of density variance estimation via the delta-method in Eq.~\ref{eq:delta-method-variance-estimator}. We assess the density variance estimation performance by introducing the oracle CI, which is the HAL-MLE point estimate \(\pm 1.96\) times the empirical standard deviation of the fitted densities across 1000 samples. Then, we evaluate the variance estimator by computing the coverage of the CI constructed by the z-statistic of the HAL-MLE point estimate. 

For each DGP and sample size \(n\), we evaluate on the 201 grid points between the first and last observation and summarize (i) interval widths and (ii) coverage of the truth \(p_0(x)\). Coverage is computed pointwise as the fraction of replicates whose 95\% CI contains the truth at each grid point. The coverage of the Delta-method variance estimator is the mean coverage over the 201 grid points. The full results, including the CI widths and the coverage of each grid point, are included in Appendix~\ref{app_i_subsec:asymptotic_normality_and_var_est}.

\noindent Table~\ref{tab:coverage_analysis} reports mean coverage summaries by DGP and sample size, shown as (plug\textendash in, oracle). Abbreviations: TN, GS3/GA3, GS5, Step, Sine. Each cell shows (estimated coverage, oracle coverage). Oracle coverage uses the empirical standard deviation across Monte Carlo Simulations.

\begin{table}[H]
  \centering
  \scriptsize
  \setlength{\tabcolsep}{2pt}
  \renewcommand{\arraystretch}{1.1}
  \begin{tabular}{lcccccccc}
    \toprule
    DGP & N=25 & N=50 & N=100 & N=200 & N=400 & N=800 & N=1600 & N=3200 \\
    \midrule
    TN & (94.7, 93.6) & (93.8, 92.9) & (93.7, 93.2) & (93.7, 92.1) & (92.6, 92.2) & (95.1, 92.1) & (95.5, 92.2) & (91.2, 92.0) \\
    GS3 & (83.7, 90.9) & (89.9, 91.2) & (90.0, 90.8) & (90.5, 91.3) & (91.2, 91.1) & (93.3, 91.4) & (92.2, 91.4) & (91.7, 91.4) \\
    GA3 & (81.5, 91.9) & (85.2, 92.4) & (87.7, 92.6) & (88.6, 92.9) & (90.9, 92.9) & (94.9, 92.8) & (95.6, 92.6) & (95.2, 92.8) \\
    GS5 & (93.8, 95.4) & (91.2, 94.1) & (88.9, 92.7) & (90.4, 93.7) & (95.3, 96.5) & (95.6, 96.5) & (94.5, 96.6) & (94.7, 96.4) \\
    Step & (90.0, 94.2) & (85.9, 92.2) & (81.9, 90.6) & (80.3, 90.3) & (82.9, 90.0) & (90.0, 90.1) & (93.2, 90.5) & (94.2, 91.0) \\
    Sine & (92.5, 94.3) & (91.0, 94.4) & (88.5, 94.7) & (85.5, 94.1) & (89.5, 92.8) & (94.2, 93.5) & (95.9, 93.5) & (96.2, 93.5) \\
    \bottomrule
  \end{tabular}
  \caption{Coverage probabilities (\%) for 95\% confidence intervals across DGPs and sample sizes. Each cell shows (estimated coverage, oracle coverage). Estimated coverage uses the method proposed in Section~\ref{density_sd_section}. Oracle coverage uses the empirical standard deviation across Monte Carlo replicates.}
  \label{tab:coverage_analysis}
\end{table}

We can see that the oracle CI coverage is always around 95\%, which means the density estimation is pointwise asymptotically normal. For most of the DGPs, the Delta-method variance estimator is also very close to the oracle CI coverage. A limitation of the Delta-method is that it does not work as well on the extrapolated region, where we have no data. For example, in the GS5 DGP, the distribution is very concentrated and the data points are centered around [0.2,0.8]. And the Delta-method variance estimator is not working as well when testing the coverage on the margins, even though the density estimation is still asymptotically normal.

\subsubsection{Asymptotic efficiency of plug-in HAL-MLE and HAL-TMLE.} \label{subsec:simulation_asymptotic_efficiency_of_plug_in_hal_mle_and_hal_tmle}
We here evaluate the asymptotic efficiency of the plug-in HAL-MLE and HAL-TMLE by comparing their performance in estimating some statistical estimands such as the mean, median, survival at 0.5, and second moment, against the existing asymptotically efficient estimators.

Trivially, the asymptotically efficient estimator of the population mean is the sample mean. And the asymptotically efficient estimator of the population median is the sample median. The asymptotically efficient estimator of population survival probability at 0.5 is the empirical survival proportion at 0.5. The asymptotically efficient estimator of the population second moment is the sample second moment.

Across DGPs and sample sizes, we compare the bias, variance, MSE, and the ratio bias/sd of plug-in HAL-MLE and HAL-TMLE with the asymptotically efficient estimators. The results of GA3 are shown below in Figure~\ref{fig:efficiency-comparison-ga3}. The full results are included in Appendix~\ref{app_i_subsec:asymptotic_efficiency}.

\begin{figure}[H]
  \centering
  \includegraphics[width=\textwidth]{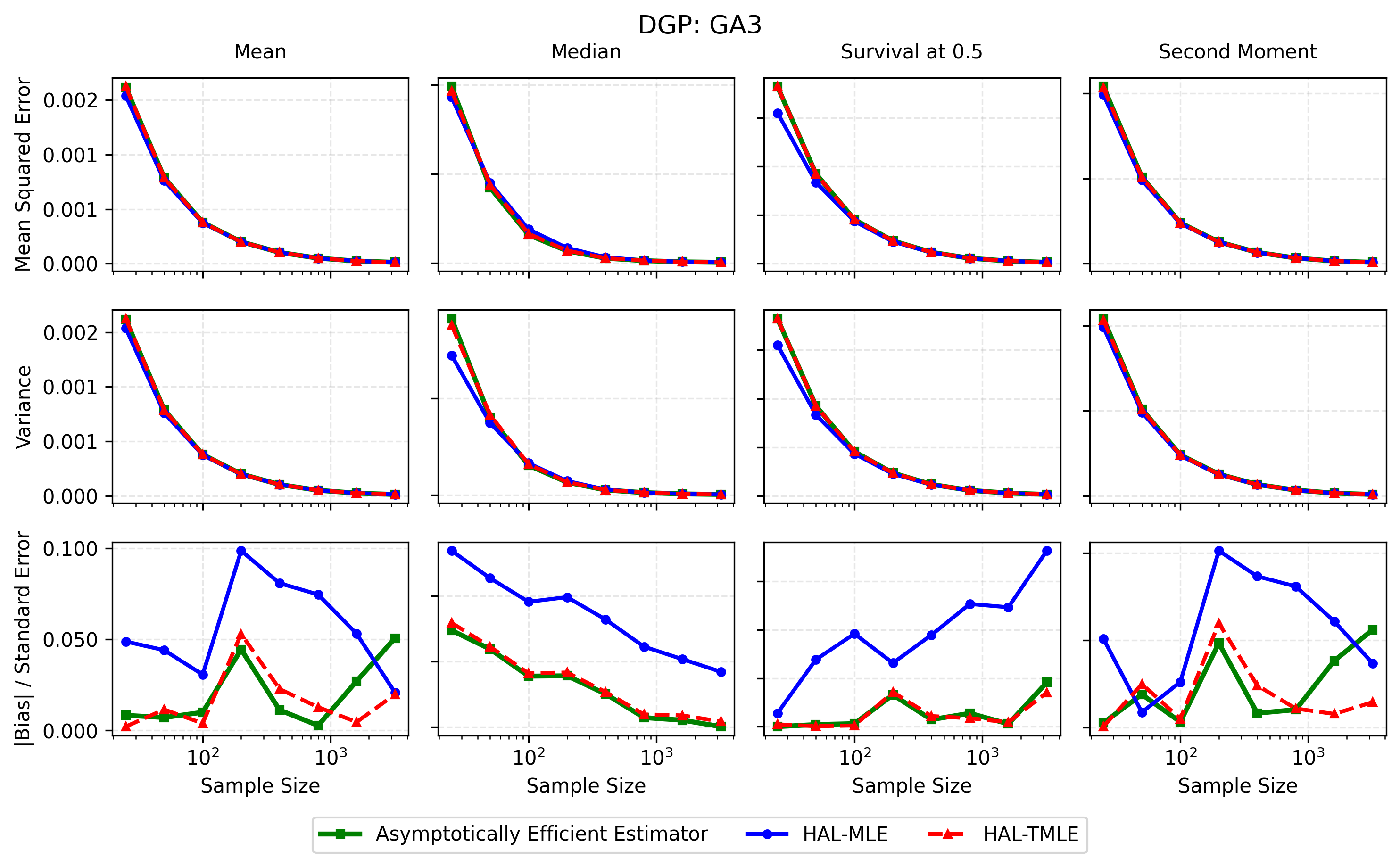}
  \caption{Representative asymptotic\textendash efficiency comparison (DGP: GA3). Columns show four statistical estimands (Mean, Median, Survival at 0.5, Second Moment) and rows show three metrics (MSE, Variance, $|$Bias$|$/SE). Curves compare the asymptotically efficient estimator (green), HAL\textendash MLE (blue), and HAL\textendash TMLE (red).}
  \label{fig:efficiency-comparison-ga3}
\end{figure}

They are comparable and thus verify the asymptotic efficiency of the plug-in HAL-MLE and HAL-TMLE. We can see that sometimes the plug-in HAL-MLE has super-efficiency over the asymptotically efficient estimator when the sample size is small. But after targeting, the EIC equation is solved, and the HAL-TMLE performance is the same as the asymptotically efficient estimator uptill some numerical errors. In the univariate case, we cannot expect the performance of Plug-in HAL-MLE and HAL-TMLE to be significantly better than their corresponding asymptotic efficient estimators, but one major advantage is that we can provide nonparametric inference using the EIC-based confidence interval.

\subsection{EIC\textendash based variance estimation for statistical estimands.} \label{subsec:simulation_eic-based-variance-estimation}
We evaluate the EIC-based variance estimation introduced in Section~\ref{subsec: EIC-based-variance-estimation} for the statistical estimands by computing the plug-in standard errors via the efficient influence curve and comparing the coverage of the oracle. We report the coverage of the 95\% confidence interval for four estimands—mean, second moment, survival at 0.5, and median—across six DGPs and eight sample sizes in Tables~\ref{tab:coverage_targeting_mean}, \ref{tab:coverage_targeting_second_moment}, \ref{tab:coverage_targeting_survival_0_5}, and \ref{tab:coverage_targeting_median}. Each table shows, for every DGP and sample size, the pair (EIC-based coverage, oracle coverage). EIC-based coverage uses plug-in standard errors derived from the efficient influence curve, while oracle coverage uses the empirical standard deviation of estimates across Monte Carlo simulations.

The full figure panels of CI width and coverage for each DGP are provided in Appendix~\ref{app_i_subsec:eic-based-variance-estimation}.

\begin{table}[H]
  \centering
  \scriptsize
  \setlength{\tabcolsep}{2pt}
  \renewcommand{\arraystretch}{1.1}
  \begin{tabular}{lcccccccc}
    \toprule
    DGP & N=25 & N=50 & N=100 & N=200 & N=400 & N=800 & N=1600 & N=3200 \\
    \midrule
    TN & (93.3, 94.6) & (94.4, 96.0) & (94.9, 95.2) & (95.8, 95.4) & (94.7, 94.7) & (93.9, 95.1) & (95.3, 95.2) & (94.3, 94.3) \\
    GS3 & (95.1, 95.9) & (95.2, 95.8) & (95.7, 95.9) & (95.5, 94.9) & (94.3, 95.1) & (95.3, 94.9) & (95.1, 94.6) & (96.3, 96.1) \\
    GA3 & (94.8, 95.4) & (94.7, 95.3) & (95.6, 94.6) & (95.3, 95.0) & (94.9, 95.2) & (94.3, 94.2) & (96.2, 94.9) & (94.8, 94.6) \\
    GS5 & (95.5, 96.0) & (94.3, 94.7) & (96.2, 96.0) & (94.7, 94.5) & (93.1, 93.6) & (94.9, 94.8) & (95.5, 94.4) & (94.6, 94.0) \\
    Step & (93.5, 95.1) & (95.6, 96.4) & (92.6, 93.7) & (95.6, 95.3) & (94.0, 94.6) & (94.0, 95.1) & (95.2, 95.0) & (94.4, 94.3) \\
    Sine & (92.8, 93.6) & (95.1, 95.1) & (94.5, 94.7) & (95.8, 95.4) & (94.7, 94.9) & (94.3, 95.3) & (95.3, 95.3) & (94.4, 94.1) \\
    \bottomrule
  \end{tabular}
  \caption{Coverage for mean}
  \label{tab:coverage_targeting_mean}
\end{table}
\begin{table}[H]
  \centering
  \scriptsize
  \setlength{\tabcolsep}{2pt}
  \renewcommand{\arraystretch}{1.1}
  \begin{tabular}{lcccccccc}
    \toprule
    DGP & N=25 & N=50 & N=100 & N=200 & N=400 & N=800 & N=1600 & N=3200 \\
    \midrule
    TN & (93.2, 94.7) & (94.8, 95.4) & (94.1, 94.8) & (95.6, 95.1) & (94.5, 94.9) & (93.9, 95.1) & (95.2, 95.2) & (94.5, 94.3) \\
    GS3 & (94.9, 95.3) & (94.9, 95.6) & (95.6, 95.7) & (95.4, 95.0) & (94.0, 94.4) & (94.6, 94.8) & (95.2, 94.7) & (96.7, 95.1) \\
    GA3 & (94.3, 95.6) & (94.4, 95.1) & (95.8, 95.1) & (95.0, 95.0) & (95.5, 95.8) & (93.9, 94.3) & (95.8, 95.1) & (95.6, 95.1) \\
    GS5 & (95.5, 96.0) & (94.2, 94.8) & (96.0, 96.1) & (95.3, 95.1) & (93.2, 94.1) & (94.8, 94.8) & (95.2, 94.4) & (94.9, 94.1) \\
    Step & (92.8, 96.0) & (94.8, 96.9) & (91.4, 94.4) & (94.9, 95.7) & (93.5, 94.0) & (94.0, 95.1) & (95.2, 95.1) & (94.2, 93.8) \\
    Sine & (93.3, 94.2) & (94.6, 95.0) & (94.1, 94.6) & (95.2, 95.4) & (93.4, 93.5) & (94.2, 95.1) & (94.7, 94.7) & (94.7, 94.0) \\
    \bottomrule
  \end{tabular}
  \caption{Coverage for second moment}
  \label{tab:coverage_targeting_second_moment}
\end{table}
\begin{table}[H]
  \centering
  \scriptsize
  \setlength{\tabcolsep}{2pt}
  \renewcommand{\arraystretch}{1.1}
  \begin{tabular}{lcccccccc}
    \toprule
    DGP & N=25 & N=50 & N=100 & N=200 & N=400 & N=800 & N=1600 & N=3200 \\
    \midrule
    TN & (95.4, 95.4) & (94.4, 94.4) & (95.0, 95.0) & (95.4, 95.4) & (94.4, 94.7) & (95.4, 95.5) & (95.5, 95.5) & (95.2, 94.6) \\
    GS3 & (95.9, 95.9) & (93.6, 93.6) & (94.8, 94.8) & (95.0, 95.0) & (94.5, 94.7) & (95.8, 95.5) & (93.8, 93.8) & (96.5, 94.5) \\
    GA3 & (94.2, 96.5) & (96.0, 95.3) & (93.2, 94.8) & (93.9, 94.9) & (95.3, 95.3) & (95.6, 95.0) & (96.1, 95.4) & (95.8, 95.5) \\
    GS5 & (96.1, 96.1) & (93.1, 94.3) & (94.4, 95.1) & (95.9, 95.7) & (94.3, 95.1) & (94.8, 95.4) & (94.8, 94.7) & (95.1, 94.8) \\
    Step & (93.4, 96.2) & (95.1, 96.8) & (96.2, 95.8) & (94.8, 95.5) & (93.9, 94.6) & (94.3, 94.5) & (96.0, 96.0) & (95.0, 94.2) \\
    Sine & (95.4, 95.4) & (94.5, 94.5) & (95.1, 95.1) & (95.4, 95.4) & (94.4, 94.4) & (94.9, 95.6) & (95.3, 95.4) & (95.2, 94.6) \\
    \bottomrule
  \end{tabular}
  \caption{Coverage for survival at 0.5}
  \label{tab:coverage_targeting_survival_0_5}
\end{table}
\begin{table}[H]
  \centering
  \scriptsize
  \setlength{\tabcolsep}{2pt}
  \renewcommand{\arraystretch}{1.1}
  \begin{tabular}{lcccccccc}
    \toprule
    DGP & N=25 & N=50 & N=100 & N=200 & N=400 & N=800 & N=1600 & N=3200 \\
    \midrule
    TN & (91.5, 94.3) & (94.0, 94.8) & (95.5, 95.9) & (95.0, 94.9) & (94.7, 95.5) & (94.8, 95.5) & (95.4, 95.6) & (95.4, 94.5) \\
    GS3 & (93.8, 92.8) & (97.0, 95.8) & (97.7, 94.7) & (96.6, 95.3) & (94.7, 94.9) & (94.9, 95.2) & (93.6, 94.1) & (96.2, 95.1) \\
    GA3 & (89.8, 91.3) & (92.4, 93.8) & (93.7, 94.5) & (96.3, 94.7) & (96.7, 94.4) & (96.1, 94.8) & (96.7, 95.1) & (95.0, 94.8) \\
    GS5 & (91.2, 94.8) & (92.7, 93.4) & (95.6, 94.7) & (96.7, 94.9) & (96.5, 94.0) & (96.3, 95.9) & (95.2, 94.8) & (96.4, 94.6) \\
    Step & (91.8, 95.2) & (93.3, 95.5) & (94.7, 96.3) & (93.7, 94.6) & (94.2, 95.4) & (95.1, 95.6) & (95.8, 96.3) & (95.8, 95.2) \\
    Sine & (91.9, 94.0) & (93.7, 94.1) & (95.1, 94.5) & (94.2, 94.5) & (95.2, 95.2) & (94.7, 95.1) & (95.7, 96.0) & (96.1, 95.3) \\
    \bottomrule
  \end{tabular}
  \caption{Coverage for median}
  \label{tab:coverage_targeting_median}
\end{table}

We can see that the EIC-based variance estimator and the oracle coverage are all very close to 95\%.

\subsection{Comparison with Existing Methods} \label{subsec:comparative_benchmarks}
We compare the HAL-MLE with the existing methods, including the TF, TFPP, logspline, and KDE methods. TF(PP) are carried out here with \texttt{CVXPY} as well for the sake of fairness. However, we also implemented an ADMM solver following \citet{ramdas2016fast} for TF on the same uniform grid, with similar performance to the \texttt{CVXPY} implementation; see Appendix~\ref{app:tf_admm}. We evaluate the performance of the existing methods at $n=800$ over a common grid of 19 points on $[0.05,0.95]$; and report the bias, variance, and MSE in Figure~\ref{fig:methods-bias-n800}, Figure~\ref{fig:methods-variance-n800}, and Figure~\ref{fig:methods-mse-n800}. 

\subsubsection{Hyperparameter selection.}
The hyperparameter selection for the TF(PP), logspline, and KDE methods are as follows: (i) The TF and TFPP methods have a $\lambda$ in $[10^{-3},10^{6}]$ (log scale)for the $\ell_1$ penalty and $k$ for the $k$th order finite differences ($k\in\{0,1,2\}$). (ii) The logspline method is a spline basis for the log-density with default hyperparameter choices in the R package \texttt{logspline}. (iii) The KDE method is a kernel smoother with jointly tuned kernel (Gaussian, Epanechnikov, tophat) and bandwidth $h$ in $[10^{-3},10^{6}]$ (log scale).

\begin{figure}[H]
    \centering
    \includegraphics[width=\textwidth]{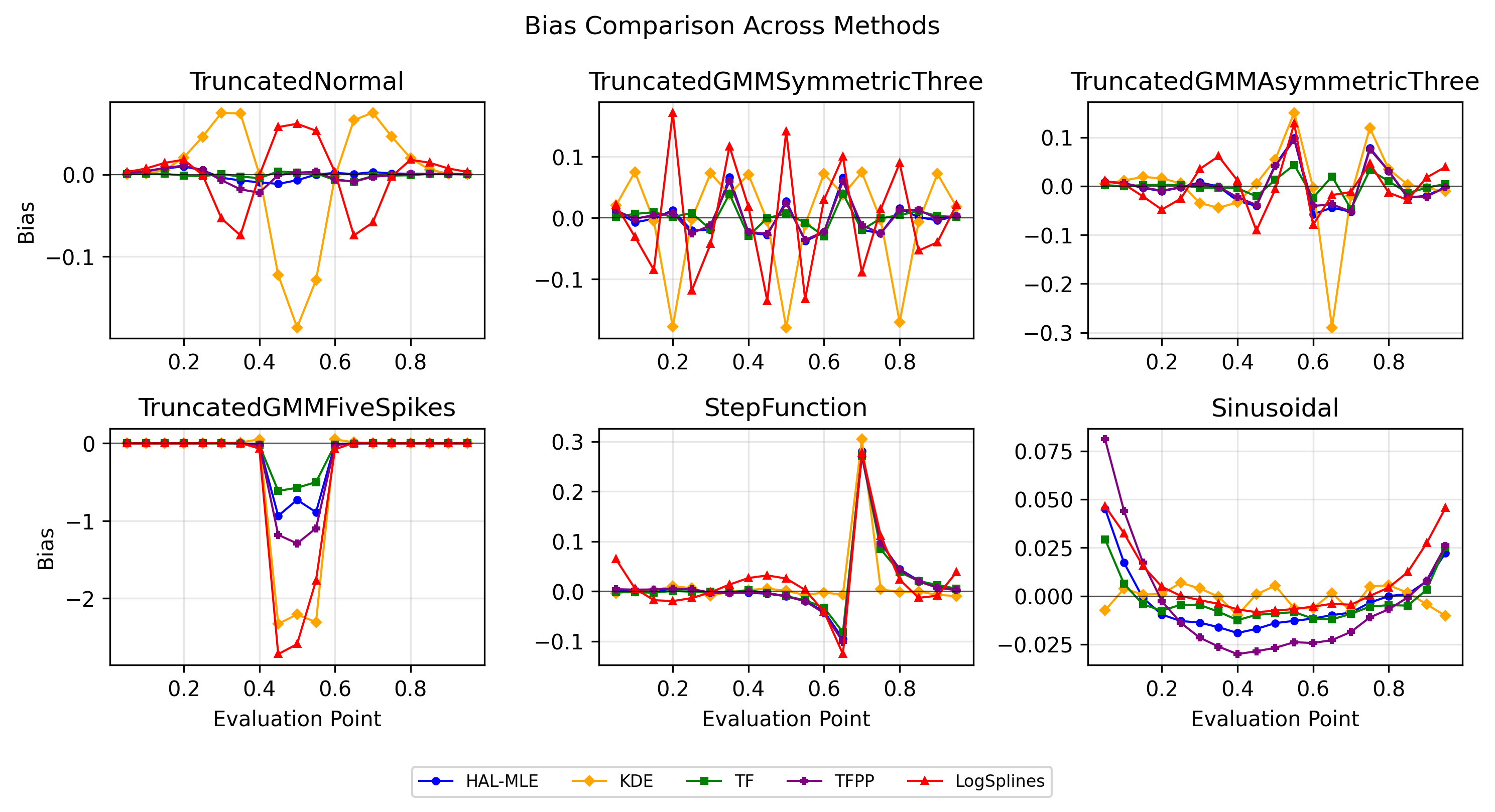}
    \caption{Bias at $n=800$ across six DGPs (columns) for five estimators (legend).}
    \label{fig:methods-bias-n800}
\end{figure}

\begin{figure}[H]
    \centering
    \includegraphics[width=\textwidth]{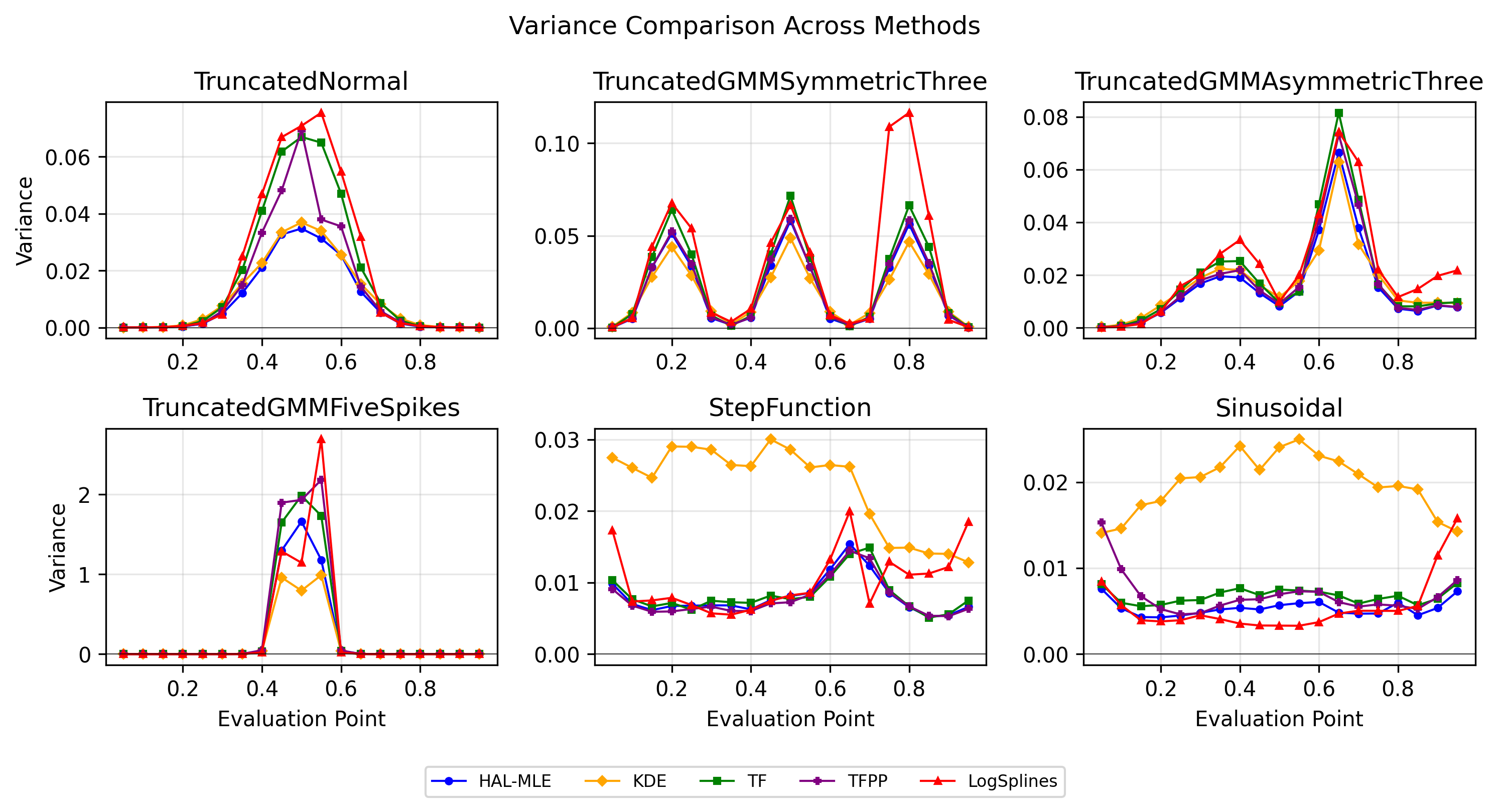}
    \caption{Variance at $n=800$ across six DGPs (columns) for five estimators (legend).}
    \label{fig:methods-variance-n800}
\end{figure}

\begin{figure}[H]
    \centering
    \includegraphics[width=\textwidth]{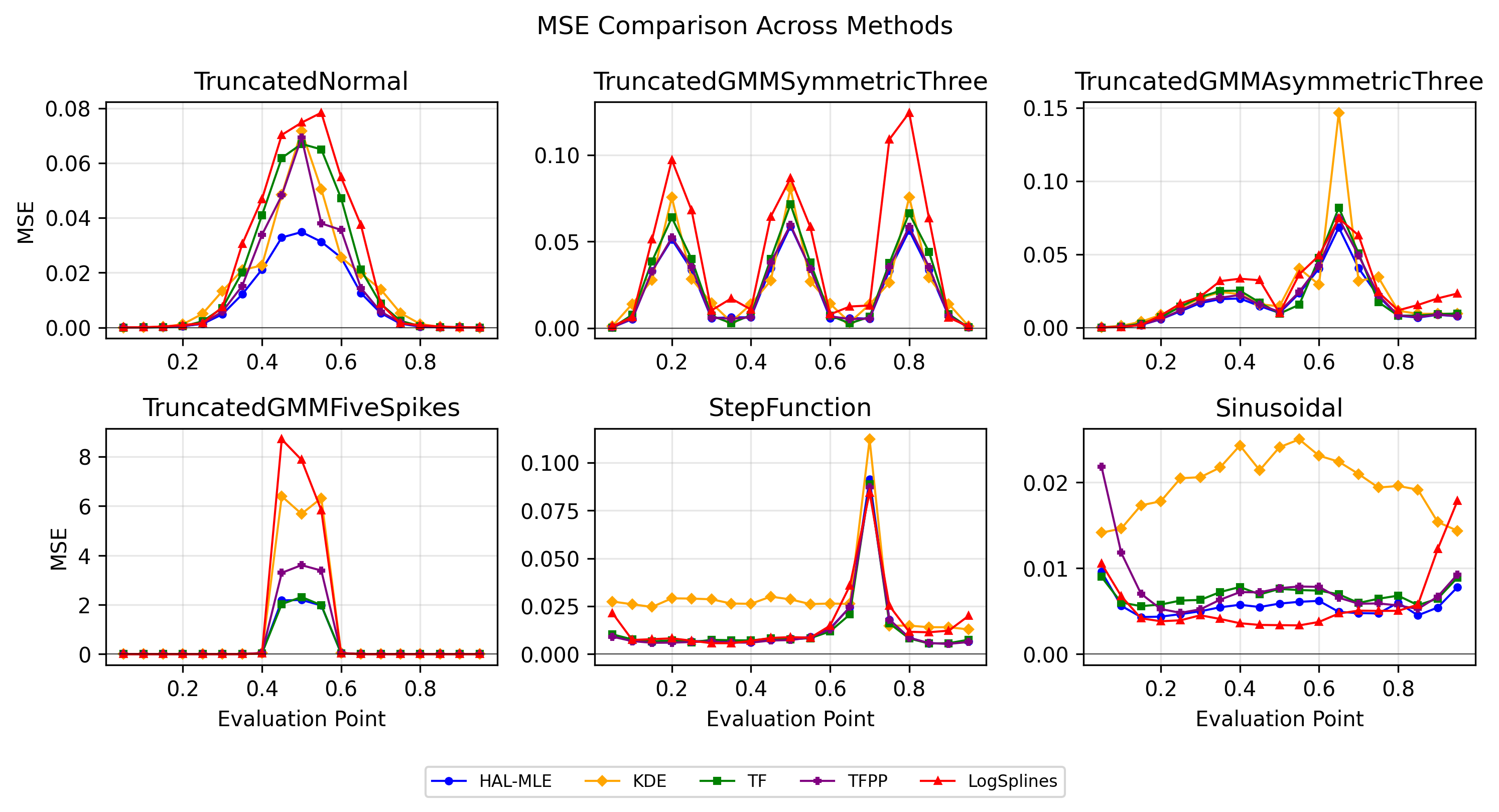}
    \caption{MSE at $n=800$ across six DGPs (columns) for five estimators (legend).}
    \label{fig:methods-mse-n800}
\end{figure}

Some observations of the comparison are as follows: (i) Comparing with logspline, which used the same link function but without controlling for the variation norm, the HAL-MLE is less biased and has a much lower MSE for the highly oscillatory DGPs, while preserving a slightly better performance on the other DGPs. (ii) Comparing with the KDE, the linear smoother, the HAL-MLE is less biased and has a much lower MSE and variance for the highly oscillatory, sudden change, and heavy-tailed DGPs. (iii) Comparing with the TF and TFPP, which controlled for the variation norm and the sectional variation norm respectively, but with a different basis system, the HAL-MLE are comparable in terms of bias and MSE. Theoretically, the 0th and 1st order falling factorial basis and the 0th and 1st order truncated power basis are equivalent \citep{tibshirani2022divided}, but the performance of TFPP and HAL-MLE here are comparable but not equivalent. One reason is that the implementation of TF(PP) uses the uniform knots instead of data-adaptive knots as in the HAL-MLE, even though using the data-adaptive knots is plausible as described in \cite{wang2014falling}. The second reason is that our implementation of TF as described in Appendix~\ref{app:tf_admm} deals with the normalizing constant brutally and thus the performance is not equivalent as expected. We acknowledge that more effort is needed, and it is beyond the scope of this paper. However, we do argue that the current implementation of TF and TFPP with uniform knots is the more computationally efficient than the HAL-MLE with a comparable performance. We verified that the divided difference matrix and its variant with penalty on the parametric part, as introduced in \cite{wang2014falling}, using uniform knots or data-adaptive knots, have better condition number than the truncated power basis respectively. The performance gives us the confidence of generalizing the pointwise asymptotic normality and uniform convergence properties of HAL-MLE to TF.

\section{Case Study}\label{sec:case-study}

This section presents a real data analysis using the classic Galaxies recession–velocity measurements. We follow the same dataset choice as \citet{bak2021penalized} and organize the discussion in three parts: (i) data description and preprocessing, (ii) density fitting and variance estimation for the density itself, and (iii) inference for several statistical estimands via plug–in HAL–MLE and HAL–TMLE.

\subsection{Data and preprocessing}
We analyze the \emph{Galaxies} data (R package \texttt{MASS}; \citealp{venables2002modern}), consisting of velocities for $n=82$ galaxies in the Corona Borealis region measured from six conic sections of space \citep[][Table~1]{postman1986probes}. If galaxies are clustered, the velocity density is expected to exhibit three to seven modes that correspond to superclusters \citep{roeder1990density}. 

For the preprocessing, we apply a monotone affine rescaling of the observed velocities to $[0,1]$. Throughout the figures, the horizontal axis represents the velocity as a proportion of the speed of light $c$: $x\in[0,1]$ corresponds to $v/c$. 

\subsection{Density fit and variance estimation}
We fit the HAL–MLE density on the rescaled axis. The resulting fit captures all prominent modes and the troughs between them, reflecting the heterogeneous galaxy population. We estimate the standard errors for the density via a delta–method as in Section~\ref{density_sd_section} and construct the confidence intervals. We benchmark these standard errors against a nonparametric bootstrap. This is a visualization of how bootstrap results differ from the delta-method results on the same small dataset. However, the bootstrap is much more computationally expensive, and thus we do not recommend it for practical use at this point. The shape of the density is similar to \citet{bak2021penalized} even though the data is rescaled.

\begin{figure}[H]
    \centering
    \includegraphics[width=\textwidth]{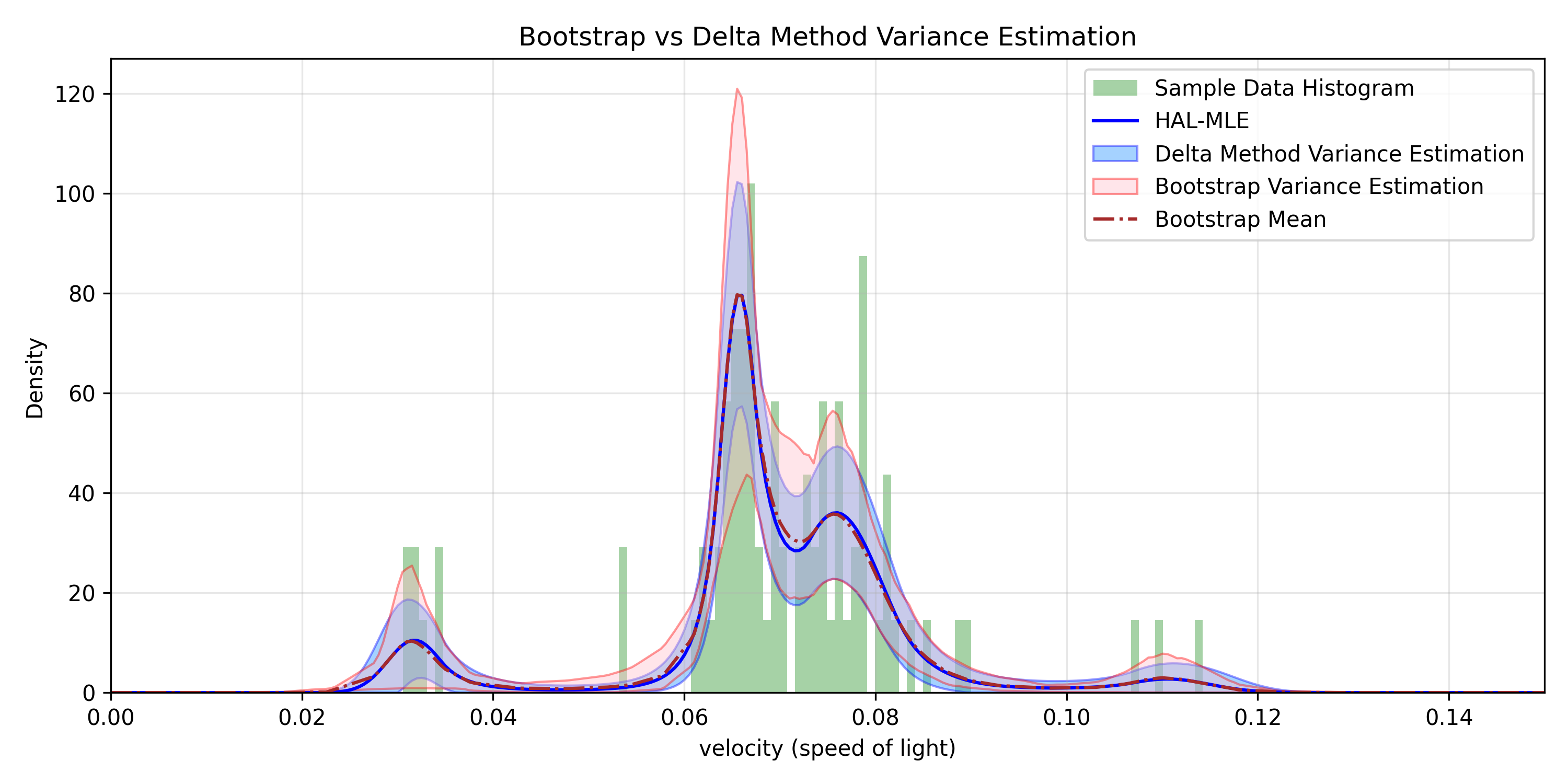}
    \caption{Galaxy velocities: histogram of the data with 50 bins, HAL–MLE density, and variance estimation for the density via the delta method (blue band) compared with a nonparametric bootstrap (red band). The $x$--axis is velocity as a proportion of the speed of light.}
    \label{fig:galaxy-density-variance}
\end{figure}

Overall ,the delta–method bands agree with the bootstrap, but there are regions where they are slightly narrower. One reason for this is that the parametric assumption of the Delta Method might cause under-estimation of variance. Another reason is that with the bootstrap, some bootstrap samples encounter the edge case of the solver, and the estimation is not accurate and varies a lot.

\subsection{Plug-in HAL-MLE versus HAL-TMLE}
We visualize the performance of the plug-in HAL-MLE and HAL-TMLE for the mean, median, and survival function. And we also visualize what targeting does to a fitted density based on the different statistical estimands. The results are shown below in Figure~\ref{fig:galaxy-mean}, Figure~\ref{fig:galaxy-median}, and Figure~\ref{fig:galaxy-survival}.

\begin{figure}[H]
    \centering
    \includegraphics[width=0.9\textwidth]{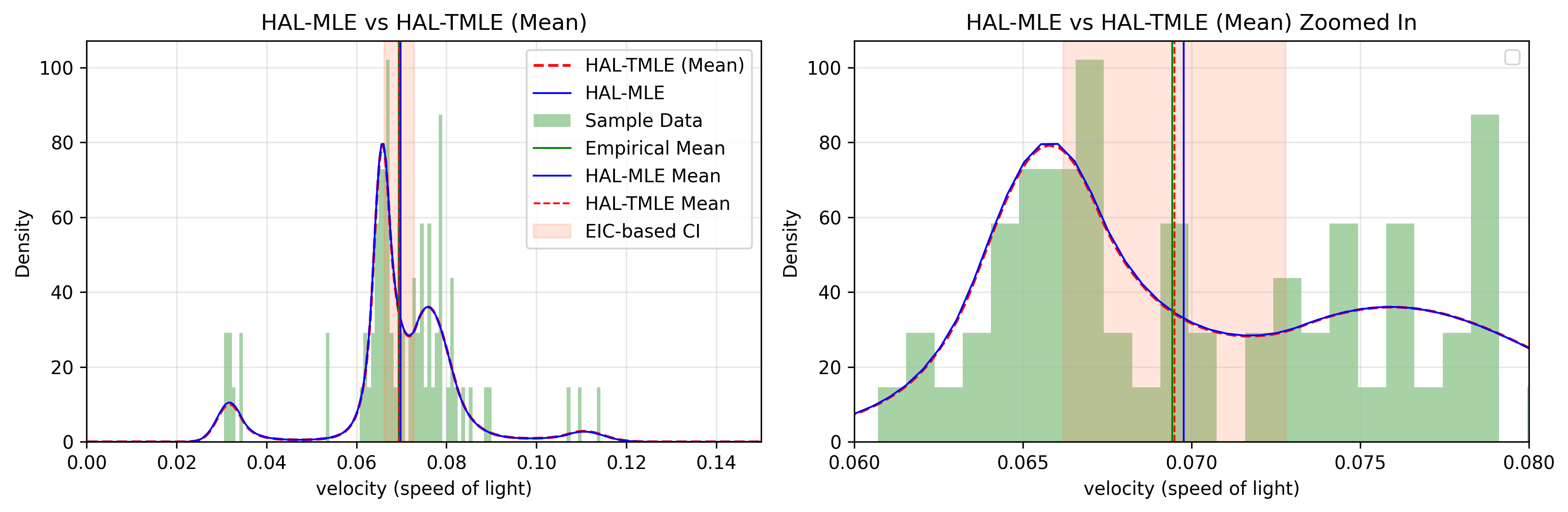}
    \caption{Mean functional: plug–in HAL–MLE versus HAL–TMLE with the sample mean shown as the asymptotically efficient benchmark. Left: full scale; right: zoom around the dominant mode.}
    \label{fig:galaxy-mean}
\end{figure}

\begin{figure}[H]
    \centering
    \includegraphics[width=0.9\textwidth]{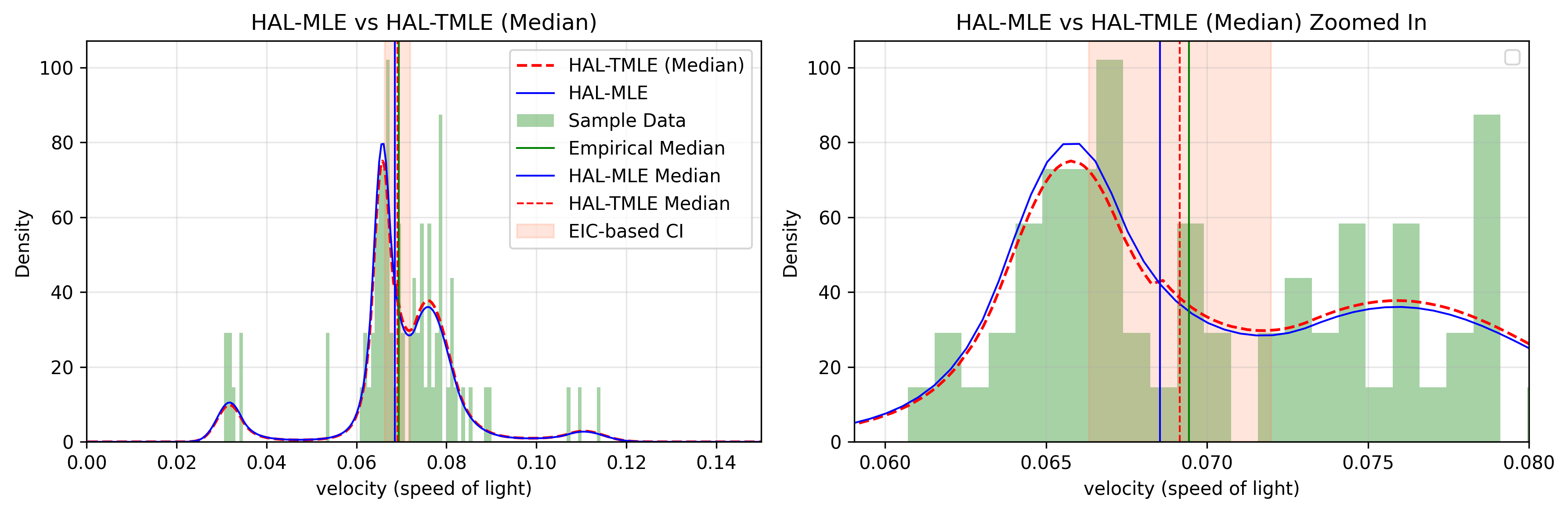}
    \caption{Median functional: plug–in HAL–MLE versus HAL–TMLE with the sample median shown for reference. Targeting nudges the estimate to align with the efficient score for the quantile.}
    \label{fig:galaxy-median}
\end{figure}

\begin{figure}[H]
    \centering
    \includegraphics[width=0.9\textwidth]{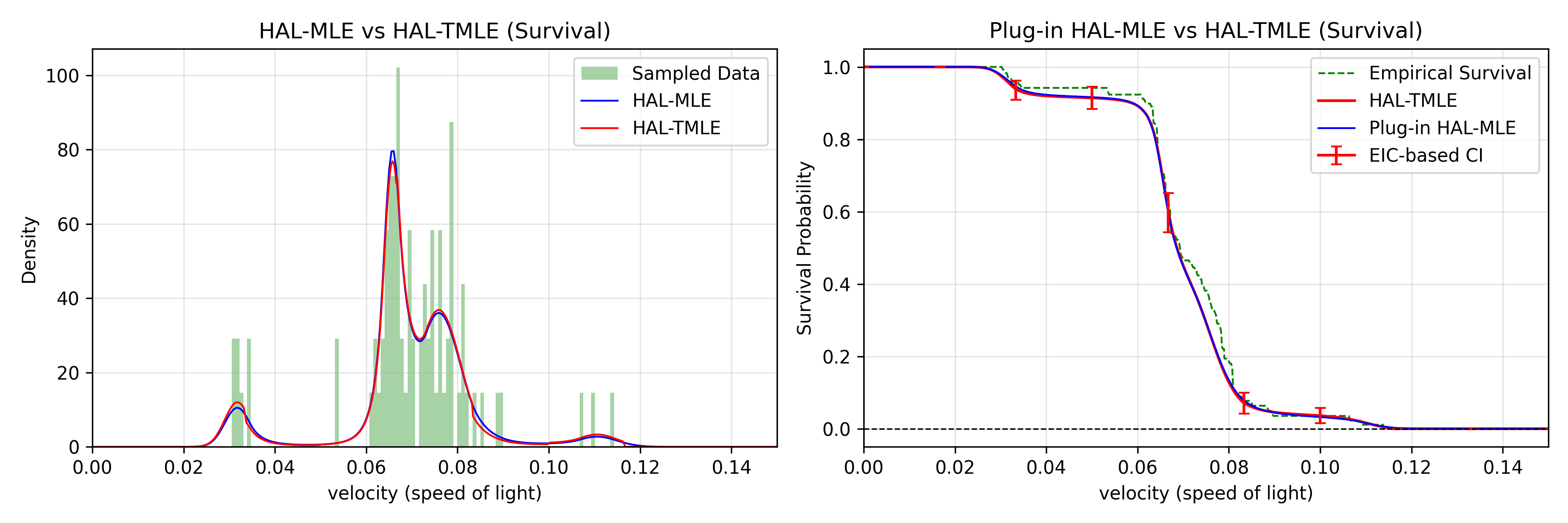}
    \caption{Survival functional $S(t)$: plug–in HAL–MLE versus HAL–TMLE, compared to the empirical survival curve. Right panel overlays EIC–based intervals for representative $t$.}
    \label{fig:galaxy-survival}
\end{figure}

The TMLE of the survival function at 0.5 shifts the entire tail after 0.5 up and down. It solves the EIC equation, but it breaks the smoothness of the density. The TMLE of mean and median debiases the estimation from the plug-in HAL-MLE towards the sample mean and median, while still preserving a fitted density slightly different from the plug-in HAL-MLE.

\section{Discussion} \label{sec:discussion}
We constructed a nonparametric density estimation technique based on HAL-MLE. We proved and verified the pointwise asymptotic normality and the uniform convergence for the univariate case when \(k\ge1\); we also give a practical delta‑method confidence interval for the density and show that plug‑in HAL‑MLE and HAL‑TMLE are asymptotically efficient for pathwise differentiable statistical estimands. We proposed an EIC–based variance estimation accordingly. Specialized proximal‑Newton/L‑BFGS/AdaGrad solvers scale and select knots effectively. We compared the performance of  different univariate density estimators, including KDE, TF(PP), log-spline, and HAL-MLE, over the same benchmark. We used a galaxy‑velocity case study to visualize the convergence, confidence intervals, and targetings. The limitation of this paper is that we only focused on the univariate case to compare and connect the classical density estimation methods in such setting. In the future, we can consider the generalization of the assumptions and extend the results to the multivariate case. We can also consider the adaptation to different data structures, such as censored/coarsened data.

\section{Acknowledgments} \label{sec:acknowledgments}
We would like to first thank Professor Ryan Tibshirani for his valuable comments and suggestions on the implementation of TF and TFPP. We had multiple discussions with him on the theoretical and practical aspects of the falling factorial basis and truncated power basis. We also learned a lot about convex optimization and the specialized ADMM algorithm for TF from him. We thank Zhexiao Lin for the detailed feedback and editing. We also acknowledge the anonymous reviewers for their helpful comments and suggestions.

\clearpage
\bibliography{ref}
\ifdefined\isarxiv%
\bibliographystyle{template/iclr2024_conference}
\else
\bibliographystyle{template/IEEE}
\fi

\newpage
\appendix
\newpage
\section{Notation, Definition, and Representation} \label{Appendix A: Notation, Definition}
\subsection{Notation}
\paragraph{Acronym}~\\
TV: Total Variation \\
BTV: Bounded Total Variation \\
SVN: Sectional Variational Norm \\
BSVN: Bounded Sectional Variational Norm \\
HAL: Highly Adaptive Lasso \\
TMLE: Target Maximum Likelihood Estimator (or Minimum Loss)\\
HAL-TMLE: TMLE that uses HAL for estimation of nuisance functions \\
HAL-MLE: Highly Adaptive Lasso Maximum Likelihood Estimator
KDE: Kernel Density Estimation \\
LAS: Local Adaptive Splines \\
TF: Trend Filtering \\
TFPP: Trend Filtering with Penalty on the Parametric Part \\
LSDE: Logspline Density Estimation \\
PLSDE: TV-penalized Logspline Density Estimation \\
KM: Kaplan-Meier Estimator \\
CV: Cross Validation \\
cadlag function: multivariate real-valued function on (say) unit cube $[0, 1]^d$ that is right-continuous with left-hand limits \\
CDF: generalized cumulative distribution function.\\
DGP: Data-Generating Process. \\
LLFP: local least favorable path \\
ULFP: universal least favorable path \\

\paragraph{Mathematical Notation}~\\
$\sim$: We say $X \sim P_0$ if the random variable $X$ has distribution $P_0$. \\
$\lesssim$: We say $a \lesssim b$ if $a \leq Cb$ for some constant $C$.\\
$\asymp$: We say $a \asymp b$ if $cb \leq a \leq Cb$ for some constant $c$ and $C$. \\
$\asymp^+$: We say $a \asymp^+ b$ if $cb \leq a \leq 
Cb(\log n)^p$ for some constant $c$, $C$, and finite $p$. Note that $p$ can be negative. \\
$\equiv$: Definition.
$\mathrm{TV}(\cdot)$: Total variation functional as defined in Section 9.1 of \citet{tibshirani2022divided}. When the function $f$ is a càdlàg function, $\mathrm{TV}(f) = \int_0^1 |df(u)|$. This is also referred to as the total variation norm (a seminorm), denoted as $|f|_v$. \\
$X$: Euclidean valued observed random variable. \\
$P_0$: True data-generating distribution of X or the truth expectation operator in empirical process. \\
$P_n$: Empirical data-generating distribution of X or the empirical expectation operator in empirical process. \\
$P_{f_n}$: Expectation operator with measure $Q_n$. \\
$P_{f_0}$: Expectation operator with measure $f_0$, which is the truth $P_0$. \\
$(\mathbb{G}_n(f) : f \in \mathcal{F})$: denotes empirical process $\mathbb{G}_n = n^{1/2}(P_n - P) \in \ell^{\infty}(\mathcal{F})$ indexed by a class of functions $\mathcal{F}$.\\
$\mathcal{M}$: Statistical model, the set of possible probability distributions for $P_0$.\\
$O^+(r(n))$: We use notation $O^+(r(n))$ if the term is $O(r(n)(\log n)^p)$ for some finite $p$, and the $+$ notation is also used in $O_p^+$. In words, it is described as "up till a $\log n$ factor".\\
$d_0(f, f_0)$: loss-based dissimilarity between $f$ and $f_0$ defined by $d_0(f, f_0) = P_0L(f) - P_0L(f_0)$ for some general loss $L$.\\
$P_0 \in \mathcal{M}$: $P_0$ is known to be an element of the statistical model $\mathcal{M}$.\\
$\mathrm{d}f/\mathrm{d}\mu$: Radon-Nikodym derivative of measure implied by cadlag function $f$ w.r.t. Lebesgue measure $\mu$.\\
$p = \mathrm{d}P/\mathrm{d}\mu$: Density of $P$ w.r.t. dominating measure $\mu$\\
$p_0$: True density of data-generating distribution $P_0$ w.r.t. appropriate dominating measure $\mu$\\
$p_f$: The density $p$ indexed by a function $f$ through a link function.\\
$\| f \|_v$: variation norm of cadlag function $f$.\\
$f^{(k)}$: k-th order Lebesgue-Radon-Nikodym derivative of cadlag function $f$. \\
$N(\varepsilon, \mathcal{F}, \| \cdot \|)$: the covering number of class $\mathcal{F}$ w.r.t. some metric defined as the minimal number of balls with radius $\varepsilon$ needed to cover $\mathcal{F}$.\\
$J(\delta, \mathcal{F}, \| \cdot \|)$: the entropy integral for class $\mathcal{F}$ w.r.t. some metric defined as $\int_0^{\delta} \sqrt{\log N(\varepsilon, \mathcal{F}, \| \cdot \|)} \mathrm{d}\varepsilon$. $J_{\infty}(\delta, \mathcal{F})$ is used when the norm is the supremum norm. $J_{2}(\delta, \mathcal{F})$ is used when the norm is the $L^2$ norm.\\
$\Pi$: projection operator.\\
$S_f$: score operator at $f$.\\
$D^*(P)$: canonical gradient (efficient influence curve) of a pathwise differentiable parameter at $P$.\\

\paragraph{Definition}~\\
$f_n$: we use subscript n for estimations, like $\beta_n$ or $C_n$. Here in this paper, $f_n$ means the HAL-MLE.\\
$f_{n, \beta_n(M)}$: the HAL-MLE with a user-specified M. So it is parameterized by the coefficient $\beta_n$ under the $L_1$ constraint less than M. $f_{n, \beta_n(M)}$ is the MLE over $D^{(k)}_M(\mathcal{R}_n)$\\
$M_{cv}$: the hyperparameter M chosen by CV.\\
$f_n^{rel}$: Relaxed HAL-MLE. Relax the HAL-MLE $f_{n, \beta_n(M)}$ by loosening the $L_1$ constraint with a larger constant $C > M$, such that it is a MLE over $D^{(k)}_C(\mathcal{R}_n)$. Notice that we use the word "fully relaxed" when we drop the $L_1$ constraint completely \\
Undersmooth: The behavior that we intentionally include more basis to reduce the bias, which will increase variance. Practically, we can choose $M' > M_{cv}$ to realize this. It represents the trade off between bias and variance. \\
$f_{n,0}$: The oracle HAL-MLE over a sample of size n. Such that $f_{n,0}=\arg\min_{f\in D^{(k)}_M({\cal R}_{n,0} )}-P_0 \log p_{f}$. \\
$\tilde f_n$: The projection of $f_n$ to some other linear subspace. Typically, we used $\tilde f_n$ to denote the projection of $f_n$ on $ D^{(k)}_M({\cal R}_{n,0})$ in Appendix~\ref{appendix_D: proof of asymptotic normality}. \\
$f_n^k$: the k-th step TMLE update, shorthand notation of $f_{n, \varepsilon_n^k}$. $f_n$ is also $f_n^0$. \\
$f^{(k)}$: the k-th order derivative of f. \\
$C_n$: The normalizing constant of $\tilde f_n$, also denoted as $C(\tilde f_n)$, or $C(\tilde \beta_n)$. \\
$C_{n,0}$: The oracle normalizing constant, also denoted as $C(f_{n,0})$, or $C(\beta_{n,0})$. \\
$\| f \|_v^*$: sectional variation norm $\| f \|_v^* = \int_{[0,1]^d} | df(u) |$.\\
$D^{(0)}_U([0, 1]^d)$: Space of real valued cadlag functions on $[0, 1]^d$ with BSVN for some $U < \infty$. We use $D^{(0)}_U([0, 1])$ for our univariate case.\\
$\| f \|_{v,k}^*$: $k$-th order sectional variation norm of f.\\
$D^{(k)}_U([0, 1])$: functions $f$ in $D^{(0)}_U([0, 1])$ for which all k-th order Lebesgue-Radon-Nikodym derivatives exist and $\| f \|_{v,k}^* < U < \infty$\\
$\mathcal{R}_0$: complete infinite index set of basis.\\
$\mathcal{R}_n$: a finite data-adaptive index set of basis. The active set after Lasso selection. For the nonparametric part of the basis system, we index them each by a tuple $(k, u)$ for $k$-th order splines and any knot point $u \in (0,1]$. For the parametric part of the basis system, we index them each by a tuple $(i, 0)$, for $i \leq k$ and starting point 0. As such, we could define the total index set as a union of all these. \\
$D^{(k)}(\mathcal{R}_n)$: for a given subset $\mathcal{R}_n \subset \mathcal{R}_0$, this is defined as the space of functions that are in closure of linear span of $\{\phi_{k,u} : (k, u) \in \mathcal{R}_n\}$. \\
$D^{(k)}_M(\mathcal{R}_n)$: the subset of $D^{(k)}(\mathcal{R}_n)$ with a user-specified $M$ as the $L_1$ norm constraint of the basis.\\
$D^{(k)}_M(\mathcal{R}_{n,0})$: an independent oracle space with a user-specified $M$ as the $L_1$ norm constraint of the basis. A technique used for proving pointwise asymptotic normality.\\
$J_n$: the effective dimension of $D^{(k)}_M(\mathcal{R}_n)$. The number of active basis. \\
$J_{n,0}$: the effective dimension of $D^{(k)}_M(\mathcal{R}_{n,0})$. The number of active basis. \\
$\phi_j^*$: some orthonormal basis function on $D^{(k)}_M(\mathcal{R}_{n,0})$.\\
$\bar \phi_{n,x}$: a special basis function at point $x$ defined as $\bar \phi_{n,x} = J_{n,0}^{-1/2} \sum_{j \in \mathcal{R}_{n,0}} \phi_j^*(x) \phi_j^*$. Used interchangably with $\bar \phi$.\\
$\tilde \phi_j$: the projection of $\phi_j^*$ onto some linear span, like $D^{(k)}_M(\mathcal{R}_n)$. In general, we use the tilde sign above an element to denote its projection on some other spaces, e.g. $\tilde \phi_{n,x}$, $\tilde f_n$.\\

\subsection{Representation Theorem} \label{App:representation thm}
This is the proof for Proposition \ref{thm:HAL_representation}

\begin{proof}
\textit{Base case:} For $k = 0$ (with $u_0 = u$),
\[
  f(x)
  = f(0) + \int_{(0, x]} \mathrm{d}\,f(u)
  = f(0) + \int_{(0, 1]} I\{u \le x\}\,\mathrm{d}\,f(u)
  = f(0) + \int_0^1 (x - u)_+^0 \,\mathrm{d}\,f(u).
\]

\textit{Induction:} If, for a $k$‑times differentiable function $f$, the representation holds true, we now show it for a $(k+1)$‑times differentiable $f$.

Since $f^{(k)}$ is differentiable,
\[
  \mathrm{d}\,f^{(k)}(u_k) = f^{(k+1)}(u_k)\,\mathrm{d}u_k,
  \quad f^{(k+1)}(u_k)\in D^{(0)}_U([0,1]).
\]
Starting from the inductive hypothesis,
\begin{align*}\label{HAL_representation_formula}
  f(x)
  &= \sum_{i = 0}^k \frac{1}{i!}f^{(i)}(0)\,x^{i}
    + \int_0^1 \frac{1}{k!}(x - u_k)_+^k\,\mathrm{d}\,f^{(k)}(u_k) \\
  &= \sum_{i = 0}^k \frac{1}{i!}f^{(i)}(0)\,x^{i}
    + \int_0^1 \frac{1}{k!}(x - u_k)_+^k\,f^{(k+1)}(u_k)\,\mathrm{d}u_k \\
  &= \sum_{i = 0}^k \frac{1}{i!}f^{(i)}(0)\,x^{i}
    + \int_0^1 \frac{1}{k!}(x - u_k)_+^k \\
  &\quad \times \Bigl(f^{(k+1)}(0)
    + \int_0^{u_k} \mathrm{d}\,f^{(k+1)}(u_{k+1})\Bigr)\,\mathrm{d}u_k \\
  &= \sum_{i = 0}^k \frac{1}{i!}f^{(i)}(0)\,x^{i}
    + \int_0^1 \frac{1}{k!}(x - u_k)_+^k\,f^{(k+1)}(0)\,\mathrm{d}u_k \\
  &\quad + \int_0^1 \frac{1}{k!}\int_0^{u_k}(x - u_k)_+^k\,\mathrm{d}\,f^{(k+1)}(u_{k+1})\,\mathrm{d}u_k \\
  &= \sum_{i = 0}^k \frac{1}{i!}f^{(i)}(0)\,x^{i}
    + f^{(k+1)}(0)\int_0^x \frac{1}{k!}(x - u_k)^k\,\mathrm{d}u_k \\
  &\quad + \int_0^1 \frac{1}{k!}\int_0^{u_k}(x - u_k)_+^k\,\mathrm{d}\,f^{(k+1)}(u_{k+1})\,\mathrm{d}u_k \\
  &= \sum_{i = 0}^{k+1} \frac{1}{i!}f^{(i)}(0)\,x^{i}
    + \int_0^1 \frac{1}{k!}\int_0^{u_k}(x - u_k)_+^k\,\mathrm{d}\,f^{(k+1)}(u_{k+1})\,\mathrm{d}u_k \\
  &= \sum_{i = 0}^{k+1} \frac{1}{i!}f^{(i)}(0)\,x^{i} \\
  &\quad + \int_0^1 \frac{1}{k!}\int_0^1(x - u_k)^k I\{u_{k+1}\le u_k\le x\}\,\mathrm{d}\,f^{(k+1)}(u_{k+1})\,\mathrm{d}u_k \\
  &= \sum_{i = 0}^{k+1} \frac{1}{i!}f^{(i)}(0)\,x^{i} \\
  &\quad + \int_0^1 \frac{1}{k!}\int_0^1(x - u_k)^k I\{u_{k+1}\le u_k\le x\}\,\mathrm{d}u_k\,\mathrm{d}\,f^{(k+1)}(u_{k+1}) \\
  &\qquad\text{(by Fubini's theorem)}\\
  &= \sum_{i = 0}^{k+1} \frac{1}{i!}f^{(i)}(0)\,x^{i}
    + \int_0^1\frac{1}{k!}\Bigl(\int_{u_{k+1}}^{x}(x - u_k)^k\,\mathrm{d}u_k\Bigr)\,\mathrm{d}\,f^{(k+1)}(u_{k+1})\\
  &= \sum_{i = 0}^{k+1} \frac{1}{i!}f^{(i)}(0)\,x^{i}
    + \int_0^1\frac{1}{(k+1)!}(x - u_{k+1})_+^{k+1}\,\mathrm{d}\,f^{(k+1)}(u_{k+1}).
\end{align*}

Then, by induction, the representation holds for all \(k\).
\end{proof}

\newpage
\section{\texorpdfstring{$L^2$}{L2} Convergence Analysis} \label{appB: $L^2$ Convergence Analysis}
We start with the $L^2$ convergence proof, which is a standard proof based on the convergence of loss-based dissimilarity. Let \(\mathcal F := D^{(k)}_U([0,1])\), and define \(f_{0} \;=\; \arg\min_{f\in\mathcal F} P_{0}L(f)\). Let $M$ be the user-supplied $L_1$ norm constraint and $\mathcal{R}_n$ be the index set of the finite-dimensional working model, and define 
\(f_{n} \;=\; \arg\min_{f\in D^{(k)}_M(\mathcal{R}_n)} P_{n}L(f).\)
Define the loss-based dissimilarity to be
\(d_{0}(f_{n},f_{0}) \;=\; P_{0}L(f_{n}) - P_{0}L(f_{0}).\) Notice that $d_{0}(f_{n},f_{0})$ is just the KL divergence of the true density $p_{f_{0}}$ and the estimated density $p_{f_{n}}$ when $L$ is $-\log p_f$. Knowing the fact that KL divergence is equivalent to the squared $L^2$-norm when the positivity assumption in Theorem~\ref{thm:density_HAL_asymptotic_normality} holds true, we can show the $L^2$ convergence of HAL-MLE by showing the convergence of $d_{0}(f_{n},f_{0})$. Again, the positivity assumption is guaranteed by the link function.

\subsection{Proof of Theorem~\ref{thm:l2-convergence-hal-mle}}
We show the convergence by discussing the covering number of $D^{(k)}_U([0,1]^d)$ and its relationship to the covering number of $\mathcal{F}_L := \{L(f) - L(f_0): f \in \mathcal F\}$.
The sup–norm covering number of $D^k_M[0,1]^d$ is covered in Lemma 29 of \citep{vanderlaan2023higherordersplinehighly}, and this implies the $L^2$ covering number of $D^k_M[0,1]^d$. However, for the univariate case, it is sufficient—and conventional—to rely directly on existing \(L^{2}\) results for functions whose \(k^{\text{th}}\) derivative has bounded total variation. The equivalency is already stated in Section~\ref{Relationship with TV}. The earliest source we have been able to locate is \citep{birman1967piecewise}, which establishes the sharp bound
\(
\log N\!\Bigl(\varepsilon,\; D^{(k)}_U([0,1]),\; L^2(\mu)\Bigr)
\;\lesssim\;
\Bigl(\tfrac{C}{\varepsilon}\Bigr)^{\frac{1}{\,k+1\,}},
\) where $C$ is a constant.
Similar results and follow-up applications also appear in \citep{mammen1997locally}, \citep{tibshirani2014adaptive}, and \citep{sadhanala2019additive}. Consequently, we adopt the one-dimensional \(L_{2}\) entropy bound throughout; it is fully compatible with—and indeed implied by the broader sup–norm result in \citep{vanderlaan2023higherordersplinehighly}.

To show the connection between $N\!\Bigl(\varepsilon,\; \mathcal{F},\; L^2(\mu)\Bigr)$ and $N\!\Bigl(\varepsilon,\; \mathcal{F}_L,\; L^2(P_0)\Bigr)$, we list two key assumptions.

\begin{assumption}[Quadratic loss-based dissimilarity assumption]\label{as1}
For the loss function \(L\) and loss–based dissimilarity 
\(d_{0}(f,f_{0}) = P_{0}L(f) - P_{0}L(f_{0})\),
assume 
\[
\sup_{f \in \mathcal{F}}
\frac{P_{0}\!\bigl\{L(f) - L(f_{0})\bigr\}^{2}}
     {d_{0}(f,f_{0})}
\;= O(1) <\;\infty.
\]
\end{assumption}

\begin{remark}
When \(L(f) = -\log p_f\) (negative log-likelihood), then Assumption~\ref{as1} becomes a fact \citet{van2004asymptotic}.
\end{remark}

\begin{assumption}[Norm comparison]\label{as2}
Let \(\lVert\cdot\rVert\) be the norm used for  $N\!\Bigl(\varepsilon,\; \mathcal{F},\; \lVert\cdot\rVert\Bigr)$.  
Assume
\[
\sup_{f \in \mathcal{F}}
\frac{P_{0}\!\bigl\{L(f) - L(f_{0})\bigr\}^{2}}
     {\lVert f - f_{0}\rVert^{2}}
\; = O(1) <\;\infty.
\]
\end{assumption}

\begin{remark}
If \(\lVert\cdot\rVert\) is chosen as either \(L_{2}(\mu)\) or \(L_{\infty}\),
then Assumption~\ref{as1} implies Assumption~\ref{as2}, based on the equivalency of KL divergence and the squared $L^2$ distance under positivity assumption.
\end{remark}

\begin{theorem}[Covering number equivalency]\label{Covering number equivalency}
For a choice of \(\lVert\cdot\rVert\) such that Assumption~\ref{as2} holds (e.g. $L^2$ and $L_{\infty}$), Assumption~\ref{as2} implies that $N\!\Bigl(\varepsilon,\; \mathcal{F},\; \lVert\cdot\rVert\Bigr) \asymp N\!\Bigl(\varepsilon,\; \mathcal{F}_L,\; L^2(P_0)\Bigr)$.
\end{theorem}

\begin{proof}
Let $\{f_j^\varepsilon: j = 1, \cdots, N(\varepsilon)\}$ be an $\varepsilon$-covering number on $\mathcal{F}$, then there $\exists j$, such that for any $f \in \mathcal{F}$, $\lVert f_j^\varepsilon - f\lVert < \varepsilon $. Notice that for negative log-likelihood loss, there is a one-to-one relationship between $\mathcal{F}$ and $\mathcal{F}_L$.

And this implies that any element in $\mathcal{F}_L$ can be expressed as $L(f) - L(f_0)$. Constructing a net $\{L(f_j^\varepsilon) - L(f_0): j = 1, \cdots, N(\varepsilon)\}$ on $\mathcal{F}_L$ based on that covering number, then $\lVert((L(f_j^\varepsilon) - L(f_0)) - (L(f) - L(f_0))\lVert_{L^2(P_0)}^2 = \lVert L(f_j^\varepsilon) - L(f)\lVert_{L^2(P_0)}^2 = P_0\{L(f_j^\varepsilon) - L(f) \}^2 =O(\lVert f_j^\varepsilon - f\rVert^{2}) = O(\varepsilon^2)$. The second last step is implied by Assumption~\ref{as2}. So this net is a covering number on $\mathcal{F}_L$ which indicates the equivalence of the covering number of both classes.
\end{proof}

Now we have developed all the tools needed for proving Theorem~\ref{thm:l2-convergence-hal-mle}. 

\begin{proof}
Let \( J(\delta, \mathcal{F}_L, L^2(P_0)) \) be the entropy integral of $\mathcal{F}_L$; for brevity we write \(J_2(\delta)\). Based on Theorem~\ref{Covering number equivalency}, $\mathcal{F}_L$ and $\mathcal{F}$ shares the same entropy integral \(J_2(\delta) = \delta^\frac{2k+1}{2k+2}\).

Define \( \mathbb{G}_n = n^{1/2} (P_n - P_0) \) to be the empirical process. \(P_{n}\) is the empirical measure and
\(P_{0}\) the true data-generating measure. Formal definitions of covering numbers and entropy integrals can be found in
Section 2.2 of \citet{van1996weak}.

\begin{align}
0
\;\le\;
d_{0}(f_{n},f_{0})
&= -\bigl(P_{n}-P_{0}\bigr)\bigl\{L(f_{n})-L(f_{0})\bigr\}
    + P_{n}\bigl\{L(f_{n})-L(f_{0})\bigr\} \notag\\
&\le -\bigl(P_{n}-P_{0}\bigr)\bigl\{L(f_{n})-L(f_{0})\bigr\} \text{, because $f_n$ is the empirical risk minimizer} \notag\\ 
&\le n^{-1/2}\,
   \sup_{\substack{f\in\mathcal F_{L}\\ \lVert f\rVert_{L^{2}(P_{0})}\leq\delta}}
   \bigl\lvert \mathbb{G}_{n}(f) \bigr\rvert \notag \\
&\lesssim n^{-1/2} J_2(\delta) \left( 1 + \frac{J_2(\delta)}{\delta^2 n^{1/2}} \right) \text{Section 3.4.2 of \citep{van1996weak}.}  \notag
\end{align}
\( J_2(\delta) \) dominates \( \frac{J_2(\delta)^2}{\delta^2 n^{1/2}} \) when \( J_2(\delta) \lesssim \delta^2 n^{1/2} \). For any \( \delta \geq \delta_n \asymp n^{-\frac{k+1}{2k+3}} \), we can use \( J_2(\delta) \) as the upper bounded for \( \mathbb{E}_{P_0}\sup_{\substack{f\in\mathcal F_{L}\\ \lVert f\rVert_{L_{2}(P_{0})}<\delta}}
   \bigl\lvert \mathbb{G}_{n}(f) \bigr\rvert\).

The sketch of the following proof is to create a sequence of $\delta \geq \delta_n$, denote as $\delta_k$, to iteratively solve for the optimal rate. We start with the trivially large radius \(\delta_{0}=1\) (and the corresponding $J(\delta_0) = O_p(1)$), which yields  
\(
d_{0}(f_{n},f_{0}) = O_{p}\!\bigl(n^{-1/2}\bigr).
\)
By Assumption~\ref{as1},
\[
P_{0}\!\bigl\{L(f_{n})-L(f_{0})\bigr\}^{2}=O_{p}\!\bigl(d_{0}(f_{n},f_{0})\bigr)=O_{p}\!\bigl(n^{-1/2}\bigr),
\qquad
\bigl\lVert L(f_{n})-L(f_{0})\bigr\rVert_{L_{2}(P_{0})}
        =O_{p}\!\bigl(n^{-1/4}\bigr).
\]
Hence we update the radius to \(\delta_{1}\asymp n^{-1/4}\), which is still
larger than or equal to \(\delta_n\).
We repeat the argument, iteratively shrinking \(\delta\) until convergence.
At convergence the \((k+1)\)-st update equals the \(k\)-th, i.e. 
\(\delta_{k+1}=\delta_{k}=\delta\), and the fixed point solves
\[
\delta
=\sqrt{n^{-1/2}J_2(\delta)}
= n^{-1/4}\,\delta^{\frac{2k+1}{4k+4}}
=\delta_n.
\]
Plugging \(\delta_n\) back into \(J_2(\delta)\) gives
\[
d_{0}(f_{n},f_{0})
\;=\;
O_{p}\!\bigl(n^{-1/2}J(\delta_n)\bigr)
\;=\;
O_{p}\!\bigl(n^{-\frac{2k+2}{2k+3}}\bigr).
\]
The positivity of density holds, which implies that 
\(d_{0}(f_{n},f_{0})\asymp\lVert p_{f_{n}}-p_{f_{0}}\rVert_{L^{2}(\mu)}^{2}\);
therefore
\[
\lVert p_{f_{n}}-p_{f_{0}}\rVert_{L^{2}(\mu)}
=O_{P}\!\bigl(n^{-\frac{k+1}{2k+3}}\bigr).
\]

\end{proof}




\newpage
\section{Score Equation Analysis for HAL-MLE} \label{App C: Score Analysis}
HAL-MLE, unlike the log-splines, is not a full parametric MLE on the working model. Instead, we have a $L_1$ norm constraint on the coefficients. So we pay a price that the  empirical score equation is not exactly 0. In this appendix, we will show that the scores we need can be approximately solved with a negligible error.

Recall the definition of the score. Let $S_{\beta}(\phi_j)=\frac{\mathrm{d}}{\mathrm{d}\beta(j)}\log p_{\beta}$ be the score at $\beta$ for $\beta(j)$, the $j$-th coefficient of $\beta$. We also denote this score as $S_{f_{\beta}}(\phi_j)$.
Note that this equals
\[
S_{\beta}(\phi_j)=\phi_j(X)-P_{\beta}\phi_j, \]
where $P_{\beta}\phi_j\equiv \int \phi_j(x)p_{\beta}(x) \mathrm{d}x$.
We note $\phi\rightarrow S_{\beta}(\phi)$ is linear. HAL-MLE has one $L_1$ norm constraint and solves $J_n - 1$ scores. We name the linear span of these $J_n - 1$ scores as the $L_1$ constrained score space, and notice that this space is a closed linear subspace of $D^k(\mathcal{R}_n)$. In the latter proof in Appendix~\ref{appendix_D: proof of asymptotic normality} and Appendix~\ref{appendix_E: Uniform Convergence Rate}, we will construct $\bar{\phi}_{n,x}$ as defined before Theorem~\ref{thm:asymptotic-linearity-hal-mle} and we want to elaborate on the rate of $P_n S_{f_n}(\bar{\phi}_{n,x})$.

Let $\tilde{\phi}_{n,x}$ be the projection of $\bar{\phi}_{n,x}$ onto the $L_1$ constrained score space, that is, the linear span of scores $S_{f_n}(\tilde{\phi}_j)$, $j\in {\cal R}_n$ that HAL solves exactly so that $P_n S_{f_n}(\tilde{\phi}_j)=0$ for $j\in {\cal R}_n$. Then, $P_n S_{f_n}(\tilde{\phi}_{n,x})=0$, and thus we have
\[
P_n S_{f_n}(\bar{\phi}_{n,x})=P_n \{S_{f_n}(\bar{\phi}_{n,x})-S_{f_n}(\tilde{\phi}_{n,x})\}.
\]
We note that the $L^2(P_{f_{n,0}})$-norm of $\bar{\phi}_{n,x}$ is bounded:
\begin{align} \label{eq:bar-phi-norm-L2P0}
\Vert \bar{\phi}_{n,x} \Vert_{\beta_{n,0}}=J_{n,0}^{-1}\sum_{j\in {\cal R}_{n,0}}\{\phi_j^*(x)\}^2=O(1).
\end{align}
We have $p_{f_{n,0}}>\delta>0$ on $[0,1]$, so that we also have 
\begin{align} \label{eq:bar-phi-norm-mu-Lebesgue}
\Vert \bar{\phi}_{n,x} \Vert_{\mu}=O(1).
\end{align}

Recall the $L^2$ Approximation Error Assumption between $D^k(\mathcal{R}_n)$ and $D^k_U([0,1])$ in Theorem~\ref{thm:asymptotic-linearity-hal-mle} 
\[
\sup_{f\in D^{(k)}_U([0,1])}\inf_{g\in D^{(k)}(\mathcal R_{n})}\|f-g\|_{\mu}
=O^+\bigl(J_{n}^{-(k+1)}\bigr).
\]

The idea of this assumption is that we can select a rich enough set of basis functions that can approximate the truth in $L^2(\mu)$-norm. The $L_1$ constrained score space has only one less basis than $D^k(\mathcal{R}_n)$, so it is reasonable to extend this assumption to the $L_1$ constrained score space. 

\begin{theorem}[Negligibility of Score Approximation Error for HAL-MLE] \label{thm:score-approximation-negligibility-hal-mle}
Suppose that the $L^2(\mu)$-Approximation Error Assumption can be extended to the $L_1$ constrained score space, then the empirical score equation satisfies
\[
n^{1/2}P_n S_{f_n}(\bar{\phi}_{n,x}) = o_P(1).
\]
\end{theorem}

\begin{proof}
If the $L^2(\mu)$-Approximation Error Assumption holds for $\bar{\phi}_{n,x}$ and its projection $\tilde{\phi}_{n,x}$ as well, then,  
\[
\Vert \bar{\phi}_{n,x} - \tilde{\phi}_{n,x} \Vert_{L^2(\mu)} = O^+(J_n^{-(k+1)}),
\]

\begin{align*}
P_n S_{f_n}(\bar{\phi}_{n,x}) &= P_n \{S_{f_n}(\bar{\phi}_{n,x})-S_{f_n}(\tilde{\phi}_{n,x})\} \\
&= (P_n - P_0) \{\bar{\phi}_{n,x}-\tilde{\phi}_{n,x}\} - (P_{f_n} - P_0) \{\bar{\phi}_{n,x}-\tilde{\phi}_{n,x}\} \\
\end{align*}    

The first term can be bounded by an empirical process, where we shorthanded $\|\cdot\|_{L^2(\mu)}$ by $\|\cdot\|_{\mu}$,
\begin{align*}
    (P_n - P_0) \{\bar{\phi}_{n,x}-\tilde{\phi}_{n,x}\} &= J_{n,0}^{\frac{1}{2}} n^{-\frac12} \sup_{\substack{f\in\mathcal D^{(k)}_U([0,1])\\ \lVert f\rVert_{L^{2}(P_{0})}\leq\delta}}
    \bigl\lvert \mathbb{G}_{n}(f) \bigr\rvert   \\
    &= J_{n,0}^{\frac{1}{2}} n^{-\frac12} J(\delta_n) \\
    &= J_{n,0}^{\frac{1}{2}} n^{-\frac12} \delta_n^{\frac{2k+1}{2k+2}} \\
    &= O_P^+(J_{n,0}^{\frac{1}{2}} n^{-\frac12} (J_n^{-1/k+1})^{(2k+1)/(2k+2)}) \\
    &= O_P^+(n^{-\frac{1}{4k+6}} n^{-\frac12} (n^{-\frac{2k+1}{4k+6}})) \\
    &= O_P^+(n^{-\frac{4k+5}{4k+6}}) \\
    &= o_p(n^{-\frac{1}{2}}) \\
\end{align*}
We here used the same empirical process as in Appendix~\ref{appB: $L^2$ Convergence Analysis}, and the $\delta_n$ is the is the $L^2(P_0)$ convergencerate of $J_n^{-\frac{1}{2}}(\bar{\phi}_{n,x}-\tilde{\phi}_{n,x})$ going to 0. This rate is implied by the $L^2$ Approximation Error Assumption and the positivity of the true density $p_{0}$.

The second term can be bounded by Cauchy-Schwarz inequality.
\begin{align*}
    (P_{f_n} - P_0) \{\bar{\phi}_{n,x}-\tilde{\phi}_{n,x}\} &\leq \|p_{f_n} - p_0\|_{\mu} \cdot \left\|\bar{\phi}_{n,x}-\tilde{\phi}_{n,x}\right\|_{\mu} \\
    &= \|p_{f_n} - p_0\|_{\mu} \cdot \left\|\bar{\phi}_{n,x}-\tilde{\phi}_{n,x}\right\|_{\mu} \\
    &= O_P^+(n^{-\frac{k+1}{2k+3}} J_n^{-(k+1)})  \\
    &= O_P^+(n^{-\frac{2k+2}{2k+3}}) \\
    &= o_p(n^{-\frac{1}{2}}) \\
\end{align*}
And thus, we have
\begin{align*}
    P_n S_{f_n}(\bar{\phi}_{n,x}) &= O_P^+(n^{-\frac{2k+2}{2k+3}}),
\end{align*}
and 
\begin{align*}
    n^{1/2}P_n S_{f_n}(\bar{\phi}_{n,x}) &= o_P(1).
\end{align*}
\end{proof}


\newpage
\section{Proof of Pointwise Asymptotic Normality} \label{appendix_D: proof of asymptotic normality}

This appendix includes the proof of Theorem~\ref{thm:asymptotic-linearity-hal-mle}, Theorem~\ref{thm:density_HAL_asymptotic_normality}, and Corollary~\ref{coro: Delta-method for Density Estimation}.

\begin{proof}
Again, recall that the HAL‐MLE over $D^{(k)}_M(\mathcal R_n)$ by
\[
f_n \;=\;\arg\min_{f\in D^{(k)}_M(\mathcal R_n)}\;-\,P_n\log p_f = f_{\beta_n}.
\]
For the proof, we need to build an auxiliary tool with a set ${\cal R}_{n,0}$ that is independent of $P_n$ instead of ${\cal R}_n$. Let \(\mathcal R_{n,0} = \mathcal R(P_n^{\#})\), using the same selection algorithm on an independent sample \(P_n^{\#}\sim P_0\).
Define the oracle MLE over $D^{(k)}_M(\mathcal R_{n,0})$ as
\[f_{n,0}=\arg\min_{f\in D^{(k)}_M(\mathcal R_{n,0})}\;-\,P_0\log p_f \;=\;f_{\beta_{n,0}},\] the oracle is represented as $f_{n,0}=\sum_{j=1}^{J}\beta_{n,0,j}\,\phi_j^{\ast}$, and it solves the score equation at $f_{n,0}$, \[P_0 S_{f_{n,0}}(\phi_j^{\ast}) = 0\;\;(\forall j).\]
Suppose \(D^{(k)}_M(R_{n,0})\) has effective dimension \(J_{n,0}\) with orthonormal basis
\(\phi^* \equiv \{\phi_j^{\ast}\}_{j=1}^{J_{n,0}}\), and we can project $f_n =  \Pi(f_n\mid D^{(k)}_M({\cal R}_{n,0}))$ from \(D^{(k)}_M(R_{n,0})\) to \(D^{(k)}_M(R_{n})\), denoted by $\tilde f_n$. Since $\tilde f_n$ is an element on $D^{(k)}_M(\mathcal R_{n,0})$, we can also represent $\tilde f_n $ in terms of \(\{\phi_j^{\ast}\}_{j=1}^{J_{n,0}}\), and we have $\tilde f_n = \sum_{j=1}^{J_{n,0}}\tilde \beta_{n}(j)\,\phi_j^{\ast}$. 

\subsection{Outline}
We here give an outline of the proof first and then argue the rates of each terms along the process. Denote the score equation at $f_n$ as $r_n(\cdot) \;=\; P_n S_{f_n}(\cdot)$, and denote the score equation at $\tilde f_n$ as $\tilde r_n(\cdot) \;=\; P_n S_{\tilde f_n}(\cdot)$. Define a linear operator T of an vector with $J$ elements such that \(T(b)(\cdot)=\sqrt{\frac{n}{J}}\sum_{j=1}^{J} b_j\,\phi_j^{\ast}(\cdot)\). Previously in Section~\ref{HAl_MLE with Link Function},  we define the normalized score vector at \(x\) by
\(
\bar{\phi}_{n,x}(\cdot)\equiv J_{n,0}^{-1/2}\sum_{j\in {\cal R}_{n,0}}\phi_j^*(x)\phi_j^*(\cdot)\).
Notice that $T(\phi^*(x))(\cdot) = \bar\phi_{n,x}$, and we shorthand it to be $\bar\phi$ in the following outline.


\begin{align*}
&\sqrt{\frac{n}{J_{n,0}}}\bigl(\tilde f_n(x) - f_{n,0}(x)\bigr) \\
&\quad= T(\tilde \beta_n - \beta_{n,0})(x) \\
&\quad= T\!\Bigl(
      -\frac{\mathrm{d}}{\mathrm{d}\beta_{n,0}}
      P_0 S_{f_{n,0}}(\phi^{\ast})
      \,( \tilde \beta_n - \beta_{n,0})
    \Bigr)(x) &&\text{(Identity Outer Product)} \\[2pt]
&\quad= T\!\bigl(
        - P_0\{S_{\tilde f_n}-S_{f_{n,0}}\}(\phi^{\ast})
        + R_{1n}(\phi^{\ast})\bigr)(x) &&\text{(Taylor Expansion)} \\[2pt]
&\quad= T\!\bigl(
        -P_0\{S_{\tilde f_n}-S_{f_n}\}
        - P_0\{S_{f_n}-S_{f_{n,0}}\}
        + R_{1n}
      \bigr)(\phi^{\ast}) &&\text{(Add\&subtract $P_0 S_{f_n}$)} \\[2pt]
&\quad= T\!\bigl(
        -\{P_{\tilde f_n}-P_{f_n}\}
        + (P_n-P_0)S_{f_n}
        - r_n + R_{1n}
      \bigr)(\phi^{\ast}) &&\text{(Score expression)} \\[3pt]
&\quad= -n^{1/2}\bigl((P_{\tilde f_n} - P_{f_n})(\bar\phi)\bigr) + n^{1/2}(P_n-P_0)S_{f_{n,0}}(\bar\phi) +  \\
&\qquad\quad (P_n-P_0)\bigl\{S_{f_n}(\bar\phi)-S_{f_{n,0}}(\bar\phi)\bigr\} - n^{1/2}r_n(\bar\phi)
     + n^{1/2}R_{1n}(\bar\phi) &&\text{(Linearity of $T$)} \\[3pt]
&\quad= - o_P(1)
     + n^{1/2}(P_n-P_0)S_{f_{n,0}}(\bar\phi) + 0\\
&\qquad\quad - o_P(1) + o_P(1) &&\text{(Rate substitution)} \\[3pt]
&\quad = n^{1/2}(P_n-P_0)S_{f_{n,0}}(\bar\phi) + o_P(1).
\end{align*}

$\sqrt{J_{n,0}}S_{f_{n,0}}(\bar\phi)$ serves as the influence curve of HAL-MLE, and it is a function of $x$ since $\bar\phi$ is a function of $x$. For an application of the CLT we need that the variance under $P_0$ of $S_{\beta_{n,0}}(\bar{\phi})$ is $O(1)$. We have \[S_{\beta_{n,0}}(\bar{\phi})=\bar{\phi}(X)-P_{\beta_{n,0}}\bar{\phi}.\] Therefore this variance can be bounded by $\| \bar{\phi}\|_{P_0}^2$. 
Recall the Positivity Condition for Theorem~\ref{thm:l2-convergence-hal-mle}, the density $p_f$ is bounded away from zero uniformly for all $f \in \mathcal F$. And the BSVN condition implies that the density $p_f$ is uniformly bounded from above by some constant as well, and then this should generally be true for the oracle MLE. Then, it follows that 
\[
\| \bar{\phi}\|_{P_0}=O(\| \bar{\phi}\|_{P_{\beta_{n,0}} } )=O( J_{n,0}^{-1}\sum_{j\in {\cal R}_{n,0}}\{\phi_j^*(x)\}^2)=O(1).
\]
Let $\sigma^2_{n,0}\equiv P_0 \{S_{\beta_{n,0}}(\bar{\phi})\}^2$. We can then apply the empirical process theory here because $S_{f_{n,0}}(\bar\phi)$ is independent from the sample data based on our construction of oracle MLE.

Then the rest of the proof of the pointwise asymptotic normality from the projection $\tilde f_n$ to the oracle $f_{n,0}$ depends on the rates of the terms $n^{1/2}\bigl(P_{\tilde f_n} - P_{f_n})(\bar\phi)\bigr)$, $n^{1/2}r_n(\bar\phi)$, and $n^{1/2}R_{1n}(\bar\phi)$. By Appendix~\ref{App C: Score Analysis}, we have that $n^{1/2} r_n(\bar \phi)$ is $o_P(1)$. For the rest of the terms, we will discuss them in the following subsections.

We want the property from HAL-MLE $f_n$ to the truth $f_0$, thus we need to deal with the projection error from $f_n$ to $\tilde{f}_n$ and the oracle approximation error from $f_{n,0}$ to $f_0$. And both of the errors could be bounded by the Uniform Approximation Assumption in Theorem~\ref{thm:asymptotic-linearity-hal-mle} of this working model w.r.t. $D^{(k)}_U([0,1])$. Since $f_n\in D^{(k)}_U([0,1])$, we certainly have that $\| f_n-\tilde{f}_n\|_{\infty} = O^+(1/J_{n,0}^{k+1})$. And $\|f_{n,0} - f_0\|_{\infty} = O^+(J_{n,0}^{-(k+1)})$. They become negligible when choosing $(n/J_{n,0})^{1/2}\asymp^+ J_{n,0}^{(k+1)}$. By adding these two terms, we achieve the pointwise asymptotic normality of HAL-MLE $f_n$ to the truth $f_0$, as follows:
\begin{align*}
    \sqrt{\frac{n}{J_{n,0}}}\bigl(f_n - f_{n,0}\bigr)(x) = -\,n^{1/2}(P_n-P_0)S_{f_{n,0}}(\bar\phi)
     + o_P(1) + \sqrt{\frac{n}{J_{n,0}}}O^{+}\!\bigl(J_{n,0}^{-(k+1)}\bigr),
\end{align*}
and thus, 
\begin{align}
    \sqrt{\frac{n}{J_{n,0}}}\bigl(f_n - f_{0}\bigr)(x) &= -\,n^{1/2}(P_n-P_0)S_{f_{n,0}}(\bar\phi)
     + o_P(1) + \sqrt{\frac{n}{J_{n,0}}}O^{+}\!\bigl(J_{n,0}^{-(k+1)}\bigr) \nonumber \\
     &= -\,n^{1/2}(P_n-P_0)S_{f_{n,0}}(\bar\phi)
     + o_P(1),
\end{align}
or equivalently, $\sigma^{-1}_{n,0} n^{1/2}(P_n-P_0)S_{\beta_{n,0}}(\bar{\phi}_{n,x})$ converges in distribution to $N(0,1)$, while $\sigma^2_{n,0}=O(1)$.

\subsection*{The Term of $R_{1n}(\bar\phi)$}
Here, we argue why $n^{1/2}R_{1n}(\bar\phi)$ is $o_P(1)$.
We get the $R_{1n}(\phi_j^*)$ term as a function of some basis function $\phi_j^*$ by the Taylor expansion of the score function $S_{f_n}(\phi_j^*)$ around $f_{n,0}$.
\begin{align}
   R_{1n}(\phi_j^*) &= P_0 \{S_{\tilde \beta_n}(\phi_j^*) - S_{\beta_{n,0}}(\phi_j^*)\} - \frac{\mathrm{d}}{\mathrm{d}\beta_{n,0}}P_{\beta_{n,0}}(\phi_j^*)(\tilde \beta_n - \beta_{n,0}) \nonumber \\
   &= P_{\tilde \beta_n}(\phi_j^*) - P_{\beta_{n,0}}(\phi_j^*) - \frac{\mathrm{d}}{\mathrm{d}\beta_{n,0}}P_{\beta_{n,0}}(\phi_j^*)(\tilde \beta_n - \beta_{n,0}) \nonumber \\
   &= -\int \phi_j^* (p_{\tilde \beta_n} - p_{\beta_{n,0}}) - \frac{\mathrm{d}}{\mathrm{d}\beta_{n,0}}p_{\beta_{n,0}}(\tilde \beta_n - \beta_{n,0}) \mathrm{d}x.
\end{align}
Denoting $R_{1n}(\tilde{f}_n,f_{n,0}) \equiv p_{\tilde \beta_n} - p_{\beta_{n,0}} - \frac{\mathrm{d}}{\mathrm{d}\beta_{n,0}}p_{\beta_{n,0}}(\tilde \beta_n - \beta_{n,0})$, and using the shorthanded notation $C(\tilde f_n)=C_n$ and $C(f_{n,0})=C_{n,0}$ for the normalizing constant, then we have the following lemma.

\begin{lemma}\label{app_d_lemma:R1n_phi_j}
\begin{align*}
R_{1n}(\phi_j) &= - \int \phi_j R_{1n}(\tilde{f}_n,f_{n,0}) \mathrm{d}x\\
 &= -C_{n,0}^{-2}\int \phi_j \exp(f_{n,0})\mathrm{d}x \cdot \tfrac{1}{2} \int \exp(\xi(\tilde{f}_n,f_{n,0}))(\tilde{f}_n-f_{n,0})^2 \mathrm{d}x\\
 &\quad +(C_n-C_{n,0})^2/(C_nC_{n,0}^2)\int \phi_j \exp(f_{n,0}) \mathrm{d}x\\
 &\quad +C_{n,0}^{-1}\tfrac{1}{2}\int \phi_j \exp(\xi(\tilde{f}_n,f_{n,0}))(\tilde{f}_n-f_{n,0})^2\mathrm{d}x\\
 &\quad -(C_n-C_{n,0})/(C_nC_{n,0})\int \phi_j (\exp(\tilde{f}_n)-\exp(f_{n,0}))\mathrm{d}x .
\end{align*}
\end{lemma}

\noindent{\bf Proof:}
Let $C_n=C_n$ and $C_{n,0}=C_{n,0}$. 
We have
\begin{align*}
p_{\tilde{f}_n}-p_{f_{n,0}} 
 &= C_n^{-1}\exp(\tilde{f}_n)-C_{n,0}^{-1}\exp(f_{n,0})\\
 &= \{C_n^{-1}-C_{n,0}^{-1}\}\exp(f_{n,0})+C_n^{-1}(\exp(\tilde{f}_n)-\exp(f_{n,0}))\\
 &= (C_{n,0}-C_n)/(C_nC_{n,0}) \exp(f_{n,0})+C_n^{-1}(\exp(\tilde{f}_n)-\exp(f_{n,0}))\\
 &= - (C_n-C_{n,0})/C_{n,0}^2 \exp(f_{n,0})+(C_n-C_{n,0})^2/(C_nC_{n,0}^2)\exp(f_{n,0})\\
 &\quad + C_{n,0}^{-1}(\exp(\tilde{f}_n)-\exp(f_{n,0}))-(C_n-C_{n,0})/(C_nC_{n,0})(\exp(\tilde{f}_n)-\exp(f_{n,0})).
\end{align*}

By exact Taylor expansion of $\exp(x)=\exp(x_0)+\exp(x_0)(x-x_0)+\tfrac{1}{2}\exp(\xi(x,x_0))(x-x_0)^2$ for a $\xi(x,x_0)$ in between  $x$ and $x_0$, we have
\[
\exp(\tilde{f}_n)-\exp(f_{n,0})=\exp(f_{n,0})(\tilde{f}_n-f_{n,0})+\tfrac{1}{2}\exp(\xi(\tilde{f}_n,f_{n,0})) (\tilde{f}_n-f_{n,0})^2.\]
We have 
\begin{align} \label{eq:Taylor_exp_norm_const}
C_n-C_{n,0} &= \int (\exp(\tilde{f}_n)-\exp(f_{n,0})) \mathrm{d}x \nonumber\\
 &= \tfrac{1}{2}\int \exp(\xi(\tilde{f}_n,f_{n,0}))(\tilde{f}_n-f_{n,0})^2 \mathrm{d}x+
\int \exp(f_{n,0})(\tilde{f}_n-f_{n,0}) \mathrm{d}x.
\end{align}
Thus,
\begin{align*}
-C_{n,0}^{-2} \exp(f_{n,0})(C_n-C_{n,0}) 
 &= -C_{n,0}^{-2}\exp(f_{n,0})\int \exp(f_{n,0})(\tilde{f}_n-f_{n,0})\mathrm{d}x\\
 &\quad -C_{n,0}^{-2}\exp(f_{n,0}) \tfrac{1}{2} \int \exp(\xi(\tilde{f}_n,f_{n,0}))(\tilde{f}_n-f_{n,0})^2 \mathrm{d}x.
\end{align*}
So
\begin{align*}
p_{\tilde{f}_n}-p_{f_{n,0}} 
 &= -C_{n,0}^{-2}\exp(f_{n,0})\int \exp(f_{n,0})(\tilde{f}_n-f_{n,0})\mathrm{d}x\\
 &\quad -C_{n,0}^{-2}\exp(f_{n,0}) \tfrac{1}{2} \int \exp(\xi(\tilde{f}_n,f_{n,0}))(\tilde{f}_n-f_{n,0})^2 \mathrm{d}x\\
 &\quad +(C_n-C_{n,0})^2/(C_nC_{n,0}^2)\exp(f_{n,0})
+C_{n,0}^{-1}\exp(f_{n,0})(\tilde{f}_n-f_{n,0})\\
 &\quad +C_{n,0}^{-1}\tfrac{1}{2}\exp(\xi(\tilde{f}_n,f_{n,0}))(\tilde{f}_n-f_{n,0})^2\\
 &\quad -(C_n-C_{n,0})/(C_nC_{n,0})(\exp(\tilde{f}_n)-\exp(f_{n,0}))\\
 &\equiv \frac{\mathrm{d}}{\mathrm{d}f_{n,0}}p_{f_{n,0}}(\tilde{f}_n-f_{n,0})+R_{1n}(\tilde{f}_n,f_{n,0}),
\end{align*}
where
\begin{align*}
\frac{\mathrm{d}}{\mathrm{d}f_{n,0}}p_{f_{n,0}}(\tilde{f}_n-f_{n,0})
 &= -C_{n,0}^{-2}\exp(f_{n,0})\int \exp(f_{n,0})(\tilde{f}_n-f_{n,0})\mathrm{d}x\\
 &\quad +C_{n,0}^{-1}\exp(f_{n,0})(\tilde{f}_n-f_{n,0}).\\
R_{1n}(\tilde{f}_n,f_{n,0})
 &= -C_{n,0}^{-2}\exp(f_{n,0}) \tfrac{1}{2} \int \exp(\xi(\tilde{f}_n,f_{n,0}))(\tilde{f}_n-f_{n,0})^2 \mathrm{d}x\\
 &\quad +(C_n-C_{n,0})^2/(C_nC_{n,0}^2)\exp(f_{n,0})\\
 &\quad +C_{n,0}^{-1}\tfrac{1}{2}\exp(\xi(\tilde{f}_n,f_{n,0}))(\tilde{f}_n-f_{n,0})^2\\
 &\quad -(C_n-C_{n,0})/(C_nC_{n,0})(\exp(\tilde{f}_n)-\exp(f_{n,0})).
\end{align*}
This completes the proof of Lemma~\ref{app_d_lemma:R1n_phi_j}. 

We apply Lemma~\ref{app_d_lemma:R1n_phi_j} to the term $R_{1n}(\bar\phi)$ and bound each term separately.
\begin{align*}
R_{1n}(\bar\phi)
  &= - \int \bar\phi \Biggl\{
        p_{f_n}-p_{f_{n,0}}
        - \frac{\mathrm{d}}{\mathrm{d}f_{n,0}}p_{f_{n,0}}(f_n-f_{n,0})
      \Biggr\} \mathrm{d}x \\[6pt]
  &\equiv -\tfrac{1}{2} C_{n,0}^{-2} \int \bar\phi \exp(f_{n,0})\, \mathrm{d}x 
            \int \exp\!\bigl(\xi(\tilde{f}_n,f_{n,0})\bigr)
             (\tilde{f}_n-f_{n,0})^2\, \mathrm{d}x \\[6pt]
  &\quad + \frac{(C_n-C_{n,0})^2}{C_n C_{n,0}^2}
            \int \bar\phi \exp(f_{n,0})\, \mathrm{d}x \\[6pt]
  &\quad + \tfrac{1}{2}C_{n,0}^{-1}\int \bar\phi 
             \exp\!\bigl(\xi(\tilde{f}_n,f_{n,0})\bigr)(\tilde{f}_n-f_{n,0})^2\, \mathrm{d}x \\[6pt]
  &\quad - \frac{C_n-C_{n,0}}{C_n C_{n,0}}
            \int \bar\phi \bigl(\exp(\tilde{f}_n)-\exp(f_{n,0})\bigr)\, \mathrm{d}x .
\end{align*}
For the first term, we know that $\exp\!\bigl(\xi(\tilde{f}_n,f_{n,0})\bigr)$ is bounded above and below, so we can treat it as a constant, and then we have
\begin{align*} 
C_{n,0}^{-2}\int \bar\phi \exp(f_{n,0})\, \mathrm{d}x \int \exp\!\bigl(\xi(\tilde{f}_n,f_{n,0})\bigr)(\tilde{f}_n-f_{n,0})^2\, \mathrm{d}x &\lesssim 
(P_{f_{n,0}}(\bar\phi) \,)(\Vert \tilde{f}_n - f_{n,0}\Vert_{\mu}^2)\\
&= O(1) O_p(n^{-\frac{2k+2}{2k+3}}) = o_p(n^{-1/2)}
\end{align*}

We use the shorthand notation for $\Vert \cdot \Vert_{L^2(P_{f_{n,0}})}$ as $\Vert \cdot \Vert_{f_{n,0}}$. By Taylor expanding of $C(\tilde f_n)$ at $C_{n,0}$ as in Eq~\ref{eq:Taylor_exp_norm_const}, we have $C(\tilde f_n) - C_{n,0} = O_p(\Vert\tilde{f}_n-f_{n,0}\Vert_{\mu})$. Then 
the second term can be bounded by $O(\Vert \phi\Vert_{f_{n,0}})$ times $\Vert \tilde{f}_n-f_{n,0}\Vert_{\mu}^2$.
\begin{align*}
\frac{(C_n-C_{n,0})^2}{C_n C_{n,0}^2}\int \bar\phi \exp(f_{n,0})\, \mathrm{d}x &= \frac{(C_n-C_{n,0})^2}{C_n C_{n,0}} \int \frac{\bar\phi \exp(f_{n,0})}{C_{n,0}} \mathrm{d}x \\
& \lesssim (\Vert \tilde{f}_n - f_{n,0}\Vert_{\mu}^2)(\Vert \phi\Vert_{f_{n,0}} \,)\\
&= O_p(n^{-\frac{2k+2}{2k+3}}) O(1) = o_p(n^{-1/2)}
\end{align*}

For the third term, again treating $\exp\!\bigl(\xi(\tilde{f}_n,f_{n,0})\bigr)$ as a constant, we need the Bounded Basis Assumption and Holder's inequality, and then we have,
\begin{align*}
\tfrac{1}{2}C_{n,0}^{-1}\int \bar\phi 
             \exp\!\bigl(\xi(\tilde{f}_n,f_{n,0})\bigr)(\tilde{f}_n-f_{n,0})^2\, \mathrm{d}x
&\lesssim \Vert \bar\phi \Vert_{\infty}\, \int (\tilde f_n - f_{n,0})^2 \mathrm{d}x\\
&=  O(J_{n,0}^{1/2})\, \Vert \tilde f_n - f_{n,0} \Vert_{\mu}^{2} \\
&= O_p(n^{1/4k+6}) O_p(n^{-\frac{2k+2}{2k+3}}) = o_p(n^{-1/2)},
\end{align*}

For the last term, we can similarly Taylor expand $\exp(\tilde{f}_n)$ on $\exp(f_{n,0})$, we have $\exp(\tilde{f}_n) - \exp(f_{n,0}) = \exp(\xi(\tilde{f}_n,f_{n,0}))(\tilde{f}_n - f_{n,0})$ for some $\xi$. Again using the Taylor Expansion of Eq~\ref{eq:Taylor_exp_norm_const}, and Cauchy-Schwartz inequality,
\begin{align*}
\frac{C_n-C_{n,0}}{C_n C_{n,0}}
            \int \bar\phi \bigl(\exp(\tilde{f}_n)-\exp(f_{n,0})\bigr)\, \mathrm{d}x &\lesssim \Vert \tilde{f}_n - f_{n,0} \Vert_{\mu}\Vert \phi \Vert_{\mu} \Vert \tilde{f}_n - f_{n,0} \Vert_{\mu}\\
            &= O_p(n^{-\frac{k+1}{2k+3}}) \, O(1) \,O_p(n^{-\frac{k+1}{2k+3}}) = o_p(n^{-1/2)}.
\end{align*}

Thus, combining all four terms, we showed $n^{1/2}R_{1n}(\bar\phi)$ is $o_P(1)$.

\subsection*{The Term $(P_{\tilde f_n} - P_{f_n})(\bar\phi)$}
Here, we argue why $n^{1/2}\bigl((P_{\tilde f_n} - P_{f_n})(\bar\phi)\bigr)$ is $o_P(1)$. This term denotes the projection error of the score function $S_{f_n}(\bar\phi)$ to the oracle score function $S_{f_{n,0}}(\bar\phi)$. And we bound it in a slighly different way.

\begin{align*}
(P_{\tilde f_n} - P_{f_n})(\bar\phi) &= J_{n,0}^{-\frac12}\sum_{j\in {\cal R}_{n,0}} (P_{\tilde f_n} - P_{f_n})(\phi_j^*) \phi_j^*(x)\\
&= J_{n,0}^{-\frac12}\sum_{j\in {\cal R}_{n,0}} P_{f_{n,0}}(p_{\tilde f_n} - p_{f_n}/p_{f_{n,0}}) \phi_j^* \phi_j^*(x)\\
&=J_{n,0}^{-\frac12}\sum_{j\in {\cal R}_{n,0}} P_{f_{n,0}}(p_{\tilde f_n} - p_{f_n}/p_{f_{n,0}}) \{\phi_j^* - P_{f_{n,0}}(\phi_j^*)\} \phi_j^*(x)\\
&=J_{n,0}^{-\frac12}\sum_{j\in {\cal R}_{n,0}} \langle (p_{\tilde f_n} - p_{f_n}/p_{f_{n,0}}), (\phi_j^*) \rangle_{\beta_{n,0}} \phi_j^*(x)\\
&= J_{n,0}^{-\frac12} \Pi(p_{\tilde f_n} - p_{f_n}/p_{f_{n,0}} \mid D^{(k)}(\mathcal{R}_{n,0}))(x)
\end{align*}

Notice that $p_{\tilde f_n} - p_{f_n}/p_{f_{n,0}}$ is a score in $L_0^2(P_{f_{n,0}})$, that is to say, the integration under measure $P_{f_{n,0}}$ is 0. And then use the inner product defined in Eq~\ref{eq:inner-product-score}, we can idenity the term as a projection of a score in $L_0^2(P_{f_{n,0}})$ to the space $D^{(k)}(\mathcal{R}_{n,0})$. 
\begin{align*}
   \Pi(p_{\tilde f_n} - p_{f_n}/p_{f_{n,0}} \mid D^{(k)}(\mathcal{R}_{n,0}))(x) &\leq \|\Pi(p_{\tilde f_n} - p_{f_n}/p_{f_{n,0}} \mid D^{(k)}(\mathcal{R}_{n,0}))\|_\infty \\
   &\leq \|p_{\tilde f_n} - p_{f_n}/p_{f_{n,0}} \|_\infty + \\
   &\quad \|\Pi(p_{\tilde f_n} - p_{f_n}/p_{f_{n,0}} \mid D^{(k)}(\mathcal{R}_{n,0})) - (p_{\tilde f_n} - p_{f_n}/p_{f_{n,0}}) \|_\infty \\
   &= O^+(1/J_{n,0}^{k+1}) + O^+(J_{n,0}^{-(k+1)}) \\
   &= O^+(1/J_{n,0}^{k+1}).
\end{align*}

All the three densities $p_{\tilde f_n}$, $p_{f_n}$, and $p_{f_{n,0}}$ are bounded from above and below, thus $\|p_{\tilde f_n} - p_{f_n}/p_{f_{n,0}} \|_\infty = O(\|\tilde f_n - f_{n,0}\|_{\infty}) = O^+(J_{n,0}^{-(k+1)})$, by the Uniform Approximation Assumption. Notice that $p_{\tilde f_n} - p_{f_n}/p_{f_{n,0}}$ has bounded sectional variation as well, and thus it is in $D^{(k)}_U([0,1])$. Thus, we can apply the Uniform Approximation Assumption again to get $\|\Pi(p_{\tilde f_n} - p_{f_n}/p_{f_{n,0}} \mid D^{(k)}(\mathcal{R}_{n,0})) - (p_{\tilde f_n} - p_{f_n}/p_{f_{n,0}}) \|_\infty = O^+(J_{n,0}^{-(k+1)})$. We can do this because $D^{(k)}(\mathcal{R}_{n,0})$ is a close linear subspace equipped with the $L^2(P_{f_{n,0}})$ norm, and thus the projection is unique. So the projection could be difined as the limiting result of a universal steepest descent algorithm, proving that projection has distance of the same order.

Thus, $(P_{\tilde f_n} - P_{f_n})(\bar\phi) \asymp J_{n,0}^{(-1/2)} O^+(J_{n,0}^{-(k+1)}) = O_P^+(n^{-\frac{1}{2}}) = o_p(n^{-\frac{1}{2}}),$ with some undersmoothening.

\subsection*{The Term $(P_n-P_0)\bigl\{S_{f_n}(\bar\phi)-S_{f_{n,0}}(\bar\phi)\bigr\}$}

Here, we argue why $(P_n-P_0)\bigl\{S_{f_n}(\bar\phi)-S_{f_{n,0}}(\bar\phi)\bigr\}=0.$
\[
E_n(\bar{\phi}_{n,x})\equiv (P_n-P_0)(S_{f_n}(\bar{\phi}_{n,x})-S_{f_{n,0}}(\bar{\phi}_{n,x})).\]
Again, based on our score expression, we have that $S_{f_n}(\phi)-S_{f_{n,0}}(\phi)=(P_{f_{n,0}}-P_{f_n})\phi$ is a constant due to cancellation of $\phi$. Therefore, the empirical process of a constant implies that $E_n(\bar{\phi}_{n,x})=0$.

\end{proof}

\subsection{Proof of Corollary~\ref{coro: Delta-method for Density Estimation}}
\begin{proof}
Let $g(f))(x)=\log p_f(x)$. Then, $g(f_0)(x)=\log p_0(x)=\log p_{f_0}(x)$. We wish to also analyze $g(f_n)(x)-g(f_0)(x)$ which implies the analysis for the density itself $p_{f_n}(x)-p_{f_0}(x)$ by a simple delta method argument. 
We note that
\begin{eqnarray*}
g(f_n)(x)-g(f_0)(x)&=& (f_n(x)-f_0(x))-(\log C(f_n)-\log C(f_0)).
\end{eqnarray*}
We have a linear approximation for $f_n(x)-f_0(x)$ given in Theorem~\ref{thm:asymptotic-linearity-hal-mle} given by $J_{n,0}^{1/2}(P_n-P_0)S_{f_{n,0} }(\bar{\phi}_{n,x})$. We can Taylor expand
\[
\log C(f_n)-\log C(f_0) = \frac{1}{C(f_0)}(C(f_n)-C(f_0)) - \frac{1}{C(\xi)^2}(C(f_n) - C(f_0))^2,
\]
where $\xi$ is between $f_n$ and $f_0$. Knowing $C(\xi)$ to be bounded,
\[
(C(f_n) - C(f_0))^2 = (\int_0^1 \exp(f_n - f_0) dx)^2 = (\int_0^1 \exp(\xi_2)(f_n - f_0) dx)^2 = O_p(\Vert f_n - f_0 \Vert_{\mu}^2),
\]
by another Taylor expansion with $\xi_2$ and $\exp(\xi_2)$ to be bounded. Thus, we have the second order remainder to be negligible by $(n/J_{n,0})^{1/2}O_p(\Vert f_n - f_0 \Vert_{\mu}^2) = o_p(1)$.

Then, notice that $C(f_0)$ is a pathwise differentiable function of $f_0$, choosing a path $f_{\epsilon}^h = f + \epsilon h$ where $\mathbb{E}_{P_{f_n}}[h] = 0$, 
\begin{align*}
    \left.\frac{\mathrm{d}}{\mathrm{d}\varepsilon} C(f_\varepsilon^h)\right|_{\varepsilon=0} &= \int_0^1 e^{f_0(x)} h(x) \, \mathrm{d}x \\
    &= \int_0^1  C(f) \, h\, \mathrm{d}P_f =0
\end{align*}
This implies $D^*_f =0$ for $C(f)$, i.e., the normalizing constant only changes in second order as one changes $P_{f}$. It will not contribute to linearization and we can treat $g(f_n)(x)-g(f_0)(x)$ as $f_n(x)-f_0(x)$.

Therefore, we can conclude that
\[
(n/J_{n,0})^{1/2}( g(f_n)(x)-g(f_0)(x))=(n/J_{n,0})^{1/2}(f_n(x)-f_0(x))+o_P(1).\]
Therefore, $\log p_{f_n}(x)-\log p_{f_0}(x)$ behaves as $(f_n-f_0)(x)$, which proves that 
\[
(n/J_{n,0})^{1/2}(p_{f_n}-p_{f_0})(x)=p_0(x)(n/J_{n,0})^{1/2}(f_n-f_0)(x)+o_P(1).
\]
\end{proof}

\newpage
\section{Proof of Uniform Convergence} \label{appendix_E: Uniform Convergence Rate}
This appendix includes the proof of Theorem~\ref{thm:uniform-density-hal}.
\begin{proof}
We extended the Basis Boundedness Assumption in Theorem~\ref{thm:density_HAL_asymptotic_normality} such that $R_{1n}(\bar{\phi}_{n,x})$ and $\tilde{r}_n(\bar{\phi}_{n,x})$ hold uniformly in $x\in [0,1]$. 
The uniform convergence result is obtained by using empirical process bounds in terms of the entropy integral to the empirical process  $(n^{1/2}(P_n-P_0)S: S\in {\cal S}_n)$ with ${\cal S}_n=\{S_{f_{n,0}}(\bar{\phi}_{n,x}):x\in [0,1]\}$. 
Notice that, by orthonormality and boundedness,
\[
\|J_{n,0}^{-\frac12}\bar\phi_{n,x}\|_{L^2(P_{f_{n,0}})}^2
=J_{n,0}^{-2}\sum_{j\in\mathcal R_{n,0}}\phi_j^*(x)^2
\;\lesssim\; J_{n,0}^{-1},
\]
hence we may take $\|J_{n,0}^{-\frac12}S_{f_{n,0}}(\bar{\phi}_{n,x})\|_{L^2(P)} \lesssim J_{n,0}^{-1/2}$. That is to say, the $L^2$-norm of these functions shrinks to zero at the rate $J_{n,0}^{-1/2}$. Also, because $x\mapsto \bar\phi_{n,x}$ is a one--dimensional mapping of x, its covering numbers satisfy
\[
\log N\!\bigl(\epsilon,\ \mathcal S_n,\ L^2(P_{f_{n,0}})\bigr)\ \lesssim\ \log(1/\epsilon),
\]
uniformly in $n$. Consequently, the entropy integral obeys the VC-1 form
\[
J(\delta,\mathcal S_n,L^2(P_{f_{n,0}}))
\;\equiv\;\int_0^\delta \sqrt{\log(1/\epsilon)}\,\mathrm{d}\varepsilon
\ \asymp \delta\sqrt{\log(1/\delta)}.
\]

So, let function class $\mathcal F_{S_n}$ be the function class that shrinks every function $s \in \mathcal S_n$ by $J_{n,0}^{\frac12}$. Then this function class $\mathcal F_{S_n}$ has a shrinkage $L^2$-norm and an entropy integral $J_2(\delta,\mathcal S_n)$

Recall that 
\begin{align*}
    \sqrt{\frac{n}{J_{n,0}}}\bigl(f_n - f_{0}\bigr)(x) 
     &= -\,n^{1/2}(P_n-P_0)S_{f_{n,0}}(\bar\phi)
     + o_P(1).
\end{align*}

Then, with a similar approach of bounding the empirical process in Appendix~\ref{appB: $L^2$ Convergence Analysis}, 
\begin{align*}
    \sup_{x\in[0,1]} |\,n^{1/2}(P_n-P_0)S_{f_{n,0}}(\bar\phi)| &=  
    J_{n,0}^{\frac12} \sup_{x\in[0,1]} |n^{1/2}(P_n-P_0) \Bigr(\frac{S_{f_{n,0}}(\bar\phi)}{J_{n,0}^{\frac12}} \Bigl) | \\
    &= J_{n,0}^{\frac12} \sup_{f \in \mathcal F_{S_n}, \|f\|_2 < \delta}|\mathbb G_n(f)| \\
    &\lesssim J_{n,0}^{\frac12} J_2(\delta_n, \mathcal F_{S_n}), \text{with a $\delta_n$ that makes $J_2(\delta_n)$ dominate} \\
    &= J_{n,0}^{\frac12} \delta_n\sqrt{\log(1/\delta_n)}.
\end{align*}

We know that $f \in \mathcal F_{S_n}, \|f\|_2 < J_{n,0}^{-1/2}$, then we can choose $\delta_n \asymp J_{n,0}^{-1/2}$. First, it is a valid choice for $J_2(\delta)$ to dominate, because $J_2(\delta_n) \asymp \delta_n\log(1/\delta_n) \lesssim \delta_n^2 n^{1/2}$. Then, $J_{n,0}^{\frac12} J_2(\delta_n, \mathcal F_{S_n}) = \sqrt{\log(1/\delta_n)} \asymp \sqrt{\log n}.$

With $J_{n,0} \asymp^+ n^{1/(2k+3)}$ this reads
\[
\sup_{x\in[0,1]}\, \big|\sqrt{n}(P_n-P_0)S_{f_{n,0}}(\bar\phi_{n,x})\big|
\;=\; O_p\!\Big(\sqrt{\log n}\Big).
\]

Then it follows:
 \[
(n/J_{n,0})^{1/2}\| f_n-f_0\|_{\infty}=O_P(\sqrt{\log n}),
\]
so that $ \| f_n-f_0\|_{\infty}=O_P^+(n^{-\frac{(k+1)}{(2k+3)}})$.

\end{proof}

\newpage
\section{Proof of Asymptotic Efficiency} \label{app_f: Asymp Efficiency}
This appendix includes the proof of Theorem~\ref{thm:Efficient Influence Curve Approximation} and Theorem~\ref{thm:plugin_HAL_asymptotic_efficiency}. We provide a standard TMLE proof establishing asymptotic efficiency under the assumption that $D^{(k)}_M({\cal R}_n)$ approximates $D^{(k)}_U([0,1])$. 

\subsection{Proof of Theorem~\ref{thm:Efficient Influence Curve Approximation}}
\begin{proof}
We have that $P_n S_{f_n}(\phi_j)=0$ for the $J_n - 1$ scores that define the $L_1$ constrained score space in $D^{(k)}({\cal R}_n)$. The tangent space generated by paths through $f$ is given by closure of linear span $\{S_f(\phi): \phi\in D^{(k)}_U([0,1])\}$. Thus, we can represent 
$D^*_f=S_f(\phi_f)$ for some $\phi_f\in D^{(k)}_U([0,1])$. Therefore, if the basis ${\cal R}_n$ is chosen so that $D^{(k)}({\cal R}_n)$  yields a $O^+(1/J_n^{k+1})$ $L^2$-approximation of $D^{(k)}_U([0,1])$ and this $L^2$ approximation can be extended to the corresponding $L_1$ constrained score space, then we can find a projection $\tilde{\phi}_f$ of $\phi_f$ onto the $L_1$ constrained score space of $D^{(k)}({\cal R}_n)$ that approximates $\phi_f$ in $L^2$-norm within distance $O^+(1/J_n^{k+1})$.
Therefore, following the same argument as in Appendix~\ref{App C: Score Analysis}, we have

\begin{align*}
P_n D^*_{f_n}=P_n S_{f_n}(\phi_{f_n})
&=P_n\{S_{f_n}(\phi_{f_n})-S_{f_n}(\tilde{\phi}_{f_n})\} + P_nS_{f_n}(\tilde{\phi}_{f_n})\\
&=(P_n-P_0)\{\phi_{f_n} - \tilde{\phi}_{f_n}\}+ \\
&\quad(P_0- P_{f_n})\{\phi_{f_n} - \tilde{\phi}_{f_n}\} + 0,
\end{align*}

where we note that the expectation of scores at $f_n$ under $P_{f_n}$ equals zero, because $S_{f_n}(\tilde{\phi}_{f_n}) = \tilde{\phi}_{f_n} - P_{f_n}\tilde{\phi}_{f_n}$ by definition.

\paragraph{Bounding the First Empirical Process Term}
We could directly use that the covering number of the basis class is the covering number of $D^{(k)}_U([0,1])$. And using the $L^2$ covering number from Appendix~\ref{appB: $L^2$ Convergence Analysis}, we can bound the first term can be bounded by $n^{-1/2}J_{2}(\delta_n,D^{(k)}_U([0,1]))$ with $\delta_n \asymp n^{-\frac{k+1}{2k+3}}$, where $J_{2}(\delta,D^{(k)}_U([0,1])\asymp\delta^{(2k+1)/(2k+2)}$. Thus, we have 
\begin{align*}
\left| (P_n - P_0)\left\{ \phi_{f_n} - \tilde{\phi}_{f_n} \right\} \right| 
&= O_P\left( n^{-1/2} J_2(\delta_n, D^{(k)}_U([0,1])) \right) \notag \\
&= O_P\left( n^{-1/2} \delta_n^{\frac{2k+1}{2k+2}} \right) \notag \\
&= O_P\left( n^{-\frac{2k+2}{2k+3}} \right) \notag \\
&= o_P\left( n^{-1/2} \right). 
\end{align*}

\paragraph{Bounding the Second Term}
The second term can be bounded with Cauchy-Schwarz inequality by $\| p_{f_n}-p_0\|_{\mu}\| \phi_{f_n}-\tilde{\phi}_{f_n}\|_{\mu}$, where we shorthanded $\|\cdot\|_{L^2(\mu)}$ by $\|\cdot\|_{\mu}$.
By the $L^2$ convergence Theorem~\ref{thm:l2-convergence-hal-mle}, we have $\| p_{f_n}-p_0\|_{\mu}=O_p(n^{-(k+1)/(2k+3)})$. By the $L^2$ Approximation Assumption, $\| \phi_{f_n}-\tilde{\phi}_{f_n}\|_{\mu}= O^+(J_n^{-(k+1)})$. So that we obtain the bound $O_p^+(n^{-(k+1)/(2k+3)} J_n^{-(k+1)})$. Choose $J_n \asymp^+ n^{1/(2k+3)}$, which could be achieved by cross validation or slight undersmoothening. Again, it follows that for all $k$, the second term is $o_P(n^{-1/2})$. 

\begin{align*}
\left| (P_0 - P_{f_n})\left\{ \phi_{f_n} - \tilde{\phi}_{f_n} \right\} \right|
&\leq \| p_{f_n} - p_0 \|_{\mu} \cdot \left\| \phi_{f_n} - \tilde{\phi}_{f_n} \right\|_{\mu} \notag \\
&\leq \| p_{f_n} - p_0 \|_{\mu} \cdot \left\|\phi_{f_n} - \tilde{\phi}_{f_n} \right\|_{\mu} \notag \\
&= O_P\left( n^{-\frac{k+1}{2k+3}} \right) \cdot O^+\left( J_n^{-(k+1)} \right) \notag \\
&= O_P^+\left( n^{-\frac{k+1}{2k+3}} n^{-\frac{k+1}{2k+3}} \right) \notag \\
&= o_P\left( n^{-1/2} \right), \label{eq:mean-difference-bound}
\end{align*}

Thus, we here finished proving Theorem~\ref{thm:Efficient Influence Curve Approximation}.
\end{proof}

\subsection{Proof of Theorem~\ref{thm:plugin_HAL_asymptotic_efficiency}}
\begin{proof}
Given we have established $P_n D^*_{f_n}=o_P(n^{-1/2})$, we have
\[
\Psi^F(f_n)-\Psi^F(f_0)=(P_n-P_0)D^*_{f_n}+R(P_{f_n},P_{f_0})+o_p(n^{-1/2}).\]
The Positivity Condition implies that the rate of $d_0(f_n,f_0)$ has the same rate of convergence in $L^2(\mu)$-norm. Then, by the Negligible Remainder Assumption in Theorem~\ref{thm:plugin_HAL_asymptotic_efficiency}, we can generally bound $R(P_{f_n},P_{f_0})$ by $\| p_{f_n}-p_{f_0}\|^2_{\mu}$.
So then $R(P_{f_n},P_{f_0})=O_P(n^{-\frac{2k+2}{2k+3}})$.
We also have that $\{D_f: f\in D^{(k)}_U([0,1])\}$ is a $P_0$-Donsker class and we assume 
$P_0\{D^*_{f_n}-D^*_{f_0}\}^2\rightarrow_p 0$. 

Therefore,
\[
(P_n - P_0)\{D^*_{f_0} - D^*_{f_n}\}
= n^{-1/2}\,o_p(1)
= o_p(n^{-1/2}).
\]

This then proves 
\[
\Psi^F(f_n)-\Psi^F(f_0)= (P_n - P_0)D^*_{f_0}+o_P(n^{-1/2}).\]
This proves asymptotic efficiency of $\Psi(f_n)$ as an estimator of $\Psi(f_0)$. 
\end{proof}

\subsection{Verification of the Assumptions for Median}
\begin{proof}
Recall that for the median functional $\psi(f) = F^{-1}(0.5)$, the efficient influence function is given by
\[
D_f(X) = \frac{1}{f(\psi)} \bigl( \tfrac{1}{2} - I(X \le \psi) \bigr),
\]
where $f$ is the corresponding density and $\psi = \psi(f)$ is the median. Let $\psi_n = \psi(f_n)$ and $\psi_0 = \psi(f_0)$.
\paragraph{Verification of the Second Assumption}
For the second assumption, we aim to show that 
\[
P_0\{D_{f_n} - D_{f_0}\}^2 = o_p(1).
\]

We have $\Vert f_n - f_0\Vert_{\infty} = O_P^+(n^{-\frac{k+1}{2k+3}})$. Then the CDFs $F_n$ and $F_0$ are close to each other in the sense that $\Vert F_n - F_0\Vert_{\infty} \leq \Vert f_n - f_0\Vert_{1} \leq \Vert f_n - f_0\Vert_{\infty} = O_P^+(n^{-\frac{k+1}{2k+3}})$.

Since $F_n(\psi_n) = F_0(\psi_0) = \tfrac{1}{2}$, we have
\[
0 = F_n(\psi_n) - F_0(\psi_0)
  = \{F_0(\psi_n) - F_0(\psi_0)\} + \{F_n(\psi_n) - F_0(\psi_n)\}.
\]
By the mean value theorem, there exists some $\xi_n$ between $\psi_n$ and $\psi_0$ such that
\[
F_0(\psi_n) - F_0(\psi_0) = f_0(\xi_n)\,(\psi_n - \psi_0).
\]
Taking absolute values yields
\[
|f_0(\xi_n)|\,|\psi_n - \psi_0| = |F_n(\psi_n) - F_0(\psi_n)|.
\]
Since $|F_n(\psi_n) - F_0(\psi_n)| \le \|F_n - F_0\|_\infty = O_P^+\!\left(n^{-\frac{k+1}{2k+3}}\right)$
and $\|f_0\|_\infty = O(1)$, we obtain
\[
|\psi_n - \psi_0|
= O_P^+\!\left(n^{-\frac{k+1}{2k+3}}\right).
\]

We start by expanding
\begin{align*}
D_{f_n} - D_{f_0}
&= \frac{1}{f_n(\psi_n)} \bigl( \tfrac{1}{2} - I(X \le \psi_n) \bigr)
 - \frac{1}{f_0(\psi_0)} \bigl( \tfrac{1}{2} - I(X \le \psi_0) \bigr) \\
&= \frac{f_0(\psi_0) - f_n(\psi_n)}{f_n(\psi_n) f_0(\psi_0)} \bigl( \tfrac{1}{2} - I(X \le \psi_0) \bigr)
  + \frac{1}{f_n(\psi_n)} \bigl(I(X \le \psi_0) - I(X \le \psi_n)\bigr).
\end{align*}
If we bound the $L^2(P_0)$-norm for the two terms on the right hand side, then we bound the left hand side.

For the first term, write
\[
f_0(\psi_0) - f_n(\psi_n) = \bigl(f_0(\psi_0) - f_n(\psi_0)\bigr)
   + \bigl(f_n(\psi_0) - f_n(\psi_n)\bigr).
\]
Assuming $f$ is differentiable with bounded derivative in a neighborhood of $\psi_0$, the mean-value theorem gives
\[
f_n(\psi_0) - f_n(\psi_n) = - f_n^{(1)}(\xi_n)(\psi_n - \psi_0)
\]
for some $\xi_n$ between $\psi_n$ and $\psi_0$. The $L^2(P_0)$-norm of the first term is bounded by $\Vert f_n - f_0\Vert_{P_0} + \vert\psi_n - \psi_0\vert = O_P^+(n^{-\frac{k+1}{2k+3}}) = o_P(1)$.

For the second term, note that WLOG, we have $\psi_0 < \psi_n$. Since $\psi_n-\psi_0$ converges to zero we have that the $L^2(P_0)$-norm of the latter indicator converges to zero.

\paragraph{Verification of the Third Assumption}
We aim to show that
\[
R(P_{f_n},P_0) = o_p(n^{-1/2}).
\]
\begin{align*}
R(P_{f_n},P_0)
&= \psi_n - \psi_0 + P_0 D_{f_n} \\
&= \psi_n - \psi_0 + \frac{F_0(\psi_0) - F_0(\psi_n)}{f_n(\psi_n)} \\
&= \psi_n - \psi_0 + \frac{f_0(\tilde{\xi}_n)}{f_n(\psi_n)} (\psi_0 - \psi_n) \\
&= \frac{f_n(\psi_n) - f_0(\tilde{\xi}_n)}{f_n(\psi_n)} (\psi_n - \psi_0) \\
& \leq \Vert f_n - f_0\Vert_{\infty} \cdot \vert\psi_n - \psi_0\vert = O_P^+(n^{-\frac{k+1}{2k+3}}) \cdot O_P^+(n^{-\frac{k+1}{2k+3}}) = o_P(n^{-1/2}).
\end{align*}

\end{proof}

\subsection{Proof of Theorem~\ref{thm:single-step-hal-tmle}}
For a single-step HAL-TMLE, knowing that $f_n^1$ converges to $f_0$ at a faster rate than $n^{-1/4}$, we have $P_n D^*_{f_n^1}= o_p(n^{-1/2})$, according to \citet{van2017generally}.
And thus the same proof of asymptotic efficiency holds for HAL-TMLE as well by extending all the $f_n$ to $f_n^1$. This is why we claim the proof above to be a standard TMLE proof.

\newpage
\section{Simulation Data Generating Process Setup}
\label{app:dgp_setup}
\subsection{Data Generating Process for Section~\ref{sec:simulation}}

We here provide the detailed data generating process (DGP) for the simulation study in Section~\ref{sec:simulation}. All distributions are defined on the support $[0,1]$ and we implement six different DGPs to evaluate the performance of our HAL-MLE methods across various distributional shapes and complexities.

\subsubsection{DGP Definitions and Parameters} \label{app_g_1:dgp_definitions}

\paragraph{1. Truncated Normal Distribution}
The truncated normal distribution is defined as:
\begin{align}
f(x) = \frac{\phi\left(\frac{x-\mu}{\sigma}\right)}{\sigma \left[\Phi\left(\frac{b-\mu}{\sigma}\right) - \Phi\left(\frac{a-\mu}{\sigma}\right)\right]}, \quad x \in [a,b]
\end{align}
where $\phi$ and $\Phi$ are the standard normal PDF and CDF respectively.

\textbf{Parameters:} $\mu = 0.5$, $\sigma = 0.1$, $a = 0$, $b = 1$

\paragraph{2. Sinusoidal Distribution}
The sinusoidal-based density is defined as:
\begin{align}
f(x) = \frac{\sin(\pi x) + 1.1}{C}, \quad x \in [0,1]
\end{align}
where $C = \int_0^1 (\sin(\pi t) + 1.1) dt = \frac{2}{\pi} + 1.1 \approx 1.7366$ is the normalization constant.

\textbf{Parameters:} No configurable parameters (fixed functional form)

\paragraph{3. Truncated GMM Symmetric Three Components}
A mixture of three truncated normal distributions:
\begin{align}
f(x) = \sum_{i=1}^3 w_i \cdot f_i(x), \quad x \in [0,1]
\end{align}
where each $f_i(x)$ is a truncated normal component.

\textbf{Parameters:}
\begin{itemize}
\item Component 1: $\mu_1 = 0.2$, $\sigma_1 = 0.05$, $w_1 = 0.33$
\item Component 2: $\mu_2 = 0.5$, $\sigma_2 = 0.05$, $w_2 = 0.34$
\item Component 3: $\mu_3 = 0.8$, $\sigma_3 = 0.05$, $w_3 = 0.33$
\end{itemize}

\paragraph{4. Truncated GMM Five Spikes}
A mixture of six truncated normal components with five narrow spikes and one broad component:
\begin{align}
f(x) = \sum_{i=1}^6 w_i \cdot f_i(x), \quad x \in [0,1]
\end{align}

\textbf{Parameters:}
\begin{itemize}
\item Spike components ($i = 1,\ldots,5$): $\mu_i \in \{0.45, 0.475, 0.5, 0.525, 0.55\}$, $\sigma_i = 0.005$, $w_i = 1/15$
\item Broad component: $\mu_6 = 0.5$, $\sigma_6 = 0.05$, $w_6 = 2/3$
\end{itemize}

\paragraph{5. Truncated GMM Asymmetric Three Components}
An asymmetric mixture of three truncated normal distributions:
\begin{align}
f(x) = \sum_{i=1}^3 w_i \cdot f_i(x), \quad x \in [0,1]
\end{align}

\textbf{Parameters:}
\begin{itemize}
\item Component 1: $\mu_1 = 0.35$, $\sigma_1 = 0.1$, $w_1 = 0.4$
\item Component 2: $\mu_2 = 0.65$, $\sigma_2 = 0.05$, $w_2 = 0.4$
\item Component 3: $\mu_3 = 0.9$, $\sigma_3 = 0.2$, $w_3 = 0.2$
\end{itemize}

\paragraph{6. Step Function Distribution}
A piecewise constant distribution:
\begin{align}
f(x) = \frac{1}{C} \begin{cases}
\ell_1, & x \in [0, b) \\
\ell_2, & x \in [b, 1]
\end{cases}
\end{align}
where $C = \ell_1 \cdot b + \ell_2 \cdot (1-b)$ is the normalization constant.

\textbf{Parameters:} $\ell_1 = 1.0$, $\ell_2 = 0.5$, $b = 0.7$

\subsubsection{Population Parameters}

Table~\ref{tab:dgp_population_params} summarizes the true population parameters for each DGP, including the mean, median, variance, second moment, and survival probability at $x = 0.5$.

\begin{table}[htbp]
\centering
\caption{Population Parameters for Data Generating Processes}
\label{tab:dgp_population_params}
\begin{threeparttable}
\begin{tabular}{lccccc}
\toprule
DGP & Mean & Median & Variance & Second Moment & $S(0.5)$ \\
\midrule
Truncated Normal & 0.5000$^*$ & 0.5000$^*$ & 0.010000 & 0.260000 & 0.5000$^*$ \\
Sinusoidal & 0.5000$^*$ & 0.5000$^*$ & 0.070145 & 0.320145 & 0.5000$^*$ \\
GMM Symmetric & 0.5000$^*$ & 0.5000$^*$ & 0.061896 & 0.311896 & 0.5000$^*$ \\
GMM Five Spikes & 0.5000$^*$ & 0.5000$^*$ & 0.002092 & 0.252092 & 0.5000$^*$ \\
GMM Asymmetric & 0.559669 & 0.608370 & 0.041087 & 0.354317 & 0.619610 \\
Step Function & 0.438235 & 0.425000 & 0.071283 & 0.263333 & 0.411760 \\
\bottomrule
\end{tabular}
\begin{tablenotes}
\small
\item $^*$Exact values due to symmetry around $x = 0.5$
\item $S(0.5) = P(X > 0.5)$ is the survival probability at $x = 0.5$
\end{tablenotes}
\end{threeparttable}
\end{table}

The population parameters in Table~\ref{tab:dgp_population_params} serve as ground truth for evaluating the accuracy of our HAL-MLE estimators across different distributional characteristics. 

\subsection{Random Seed Generation} \label{app_g_2:random_seed}
In this subsection, we provide the random seeds used for generating the simulation data. A master seed ($m=42$) generates a deterministic sequence of replicate seeds $\{s_r\}_{r=1}^R$ with $R=1000$ per $(\text{DGP}, n)$. In replicate $r$ the data sampler is seeded by $s_r$; cross‑validation fold assignments use the same $s_r$ (ensuring identical partitions across estimators).

\newpage
\section{Trend Filtering with a Specialized ADMM Algorithm for Density Estimation}
\label{app:tf_admm}

In this appendix, we describe how we extend the Trend Filtering (TF) method to our density estimation framework. We start from the equivalency between zero-order falling factorial basis and zero-order truncated power basis. This allows us to reparameterize our zero-order HAL-MLE to TF with zero-order divided differences. We start with the likelihood function $p_{\text{fn}}$ and construct the likelihood using linear combinations of basis functions through the link function. By reparameterizing our $\phi \beta$ as $\theta$, we face the challenge of the normalizing constant $\int_0^1 \exp(\phi \beta) dx$, which we approximate by breaking the integral into a sum over grid bins, yielding the normalizing constant $\sum_j \exp(\theta_j) \Delta x_j$.

\subsection{Problem Formulation}

\subsubsection{Uniform Grid Construction}

Our current implementation in \texttt{methods/non\_HAL\_method/TF\_CVXPY} uses an \emph{equally spaced} grid on $[0,1]$ (rather than a data-adaptive grid). This avoids numerical pathologies caused by extremely small gaps in data-adaptive grids when $n$ is large, and matches the grid choice used elsewhere in our experiments.

\begin{enumerate}
\item Choose the number of bins $J = n+1$ (so the parameter dimension is comparable to the legacy data-adaptive choice).
\item Define uniform knots: $x_j = j/J$ for $j = 0,1,\ldots,J$.
\item Create bins: $[x_0, x_1), [x_1, x_2), \ldots, [x_{J-1}, x_J]$ with widths $\Delta x_j = x_{j+1}-x_j = 1/J$.
\end{enumerate}

Given observations $y_1,\ldots,y_n$, let $c_j$ be the histogram count in bin $j$:
\[
c_j \;=\; \#\{i:\; y_i \in [x_j,x_{j+1})\}, \qquad j=0,1,\ldots,J-1,
\]
so that $\sum_{j=0}^{J-1} c_j = n$.

\subsubsection{Likelihood Construction and Reparameterization}

We start with the likelihood function $p_{\text{fn}}$ and construct the likelihood using linear combinations of basis functions through the link function. The key insight is that the TF basis functions are equivalent to truncated power basis functions, allowing us to reparameterize our original $\phi \beta$ parameters as $\theta$.

We model the log-density as piecewise constant over the grid:
\begin{align*}
\log p_\theta(x) = \theta_j, \quad x \in [x_j, x_{j+1}), \quad j = 0, 1, \ldots, J-1.
\end{align*}

As in the implementation, we fix an identifiability constraint $\theta_0=0$ and optimize over the full vector $\theta=(\theta_0,\ldots,\theta_{J-1})\in\mathbb{R}^J$ subject to this constraint.

The normalizing constant is approximated by a Riemann sum on the grid:
\begin{align*}
Z(\theta) = \sum_{j=0}^{J-1} \exp(\theta_j)\,\Delta x_j.
\end{align*}

The corresponding density function is:
\begin{align*}
p_\theta(x) = \frac{\exp(\theta_j)}{Z(\theta)}, \quad x \in [x_j, x_{j+1}).
\end{align*}

Using the binned likelihood (i.e., grouping observations by their bin counts), the negative log-likelihood becomes:
\begin{align*}
f(\theta) = -\sum_{j=0}^{J-1} c_j \theta_j + n \log Z(\theta).
\end{align*}

\subsubsection{Total Variation Penalty}

The total variation representation of TF is the $\ell_1$ norm of $\phi^{-1} \theta$. Since TF does not penalize the parametric part, this representation simplifies to the $\ell_1$ norm of the zero-order divided difference of $\theta$. We can further extend this to $k$-th order divided differences.

The $k$-th order difference matrix $D^{(k)}$ is constructed recursively by Eq(3) of \citep{ramdas2016fast}.

The complete optimization problem is:
\begin{align} \label{eq:tf_admm_problem}
\boxed{
\min_{\theta \in \mathbb{R}^{J}} f(\theta)
\quad \text{s.t.}\quad \theta_0 = 0,\;\; \|D^{(k+1)}\theta\|_1 \leq C
}
\end{align}

This constrained form is equivalent (under standard constraint qualification conditions) to a penalized form with a Lagrange multiplier $\lambda\ge 0$:
\[
\min_{\theta:\,\theta_0=0}\; f(\theta) + \lambda \|D^{(k+1)}\theta\|_1.
\]

\subsection{ADMM Algorithm}

Following the approach in the trend filtering literature, we adapt the specialized ADMM algorithm for trend filtering to our density estimation case. The algorithm still consists three steps.

\subsubsection{Augmented Lagrangian}

We introduce auxiliary variables $\alpha = D^{(k)} \theta$. And we use $0$ as a placeholder for $\theta_0$ and such that $\theta = (0, \tilde{\theta})$ is the full parameter vector. The normalization is achieved by $Z(\theta)$. The augmented Lagrangian is:
\begin{align*}
\mathcal{L}_\rho(\theta, \alpha, u) &= f(\theta) + \lambda \|D^{(1)} \alpha\|_1 \\
&\quad + \frac{\rho}{2} \|\alpha - D^{(k)} \theta + u\|_2^2 - \frac{\lambda}{2} \|u\|_2^2
\end{align*}
where $u$ is the scaled dual variable and in practice we choose $\rho = \lambda$ in practice \citet{ramdas2016fast}.

\subsubsection{ADMM Iterations}

The ADMM algorithm alternates between three updates:

\paragraph{$\theta$-step (Non-smooth Optimization):}
\begin{align*}
\tilde{\theta}^{t+1} = \arg\min_{\tilde{\theta}} f(\tilde{\theta}) + \frac{\rho}{2} \|\alpha^t - D^{(k)} (0, \tilde{\theta}^t) + u^t\|_2^2
\end{align*}

The first step no longer has a closed-form solution like in the regression case \citet{ramdas2016fast} due to the likelihood. So we have to solve it using the L-BFGS algorithm.

\paragraph{$\alpha$-step (Fused Lasso):}
\begin{align*}
\alpha^{t+1} = \arg\min_\alpha \frac{1}{2} \|\alpha - v^t\|_2^2 + \frac{\lambda}{\rho} \|D^{(1)} \alpha\|_1
\end{align*}
where $v^t = D^{(k)} \theta^{t+1} - u^t$.

This is a standard fused lasso problem that can be solved using the solution path algorithm \citep{tibshirani2011solution}.

\paragraph{Dual Update:}
\begin{align*}
u^{t+1} = u^t + \alpha^{t+1} - D^{(k)} (0, \tilde\theta^{t+1})
\end{align*}

\subsubsection{Convergence Criteria}

We monitor primal and dual residuals:
\begin{align*}
r^t &= \alpha^t - D^{(k)} \theta^t \quad \text{(primal residual)} \\
s^t &= \rho D^{(k)T} (\alpha^t - \alpha^{t-1}) \quad \text{(dual residual)}
\end{align*}

The algorithm terminates when:
\begin{align*}
\|r^t\|_2 \leq \varepsilon_{\text{pri}} \quad \text{and} \quad \|s^t\|_2 \leq \varepsilon_{\text{dual}}
\end{align*}
with tolerances $\varepsilon_{\text{pri}} = 10^{-4} \sqrt{m}$ and $\varepsilon_{\text{dual}} = 10^{-4} \sqrt{n}$, where $m$ is the dimension of $\alpha$.

\subsection{TFPP} \label{subsec:tfpp}
We include the variant of TF with penalty on the parametric part as TFPP, by replacing the divided difference matrix $D^{(k)}$ with the variant $H$ in Eq(10) and Eq(11) of \citet{wang2014falling}. And so others can be adapted accordingly. Notice that Eq(10) and Eq(11) of \citet{wang2014falling} are for the data-adaptive knots, and we use the uniform knots here simply by replacing the diagonal matrix $\Delta^{(k)}$ with an scaled identity matrix, such that the ADMM algorithm can be adapted accordingly.

One problem of adapting this construction to data-adaptive knots is that the entries of the inverse of matrix H can be extremely large when we sample from the distribution defined on $[0,1]$. The data points can get extremely close to each other which a gap of $10^{-16}$ or smaller. This can cause numerical instability when we iterate the $D^{(k)}$, and the optimization can diverge. We address this by computing the inverse of matrix H using \citet[Algorithm 2]{wang2014falling}. This algorithm is more stable because the matrix H, with entries of $10^{16}$ or larger, can be calculated implicitly. However, a similar ADMM algorithm for TFPP density estimation with data-adaptive knots is still under development and will be addressed in the future.

\newpage
\section{Additional Simulation Results} \label{app_i:addl-sim}

\noindent This appendix provides additional simulation results complementing Section~\ref{sec:simulation}. It includes detailed visual summaries and diagnostics across all DGPs and sample sizes.

\paragraph{Contents of this appendix}
\begin{itemize}
    \item Asymptotic normality and variance estimation (Section~\ref{app_i_subsec:asymptotic_normality_and_var_est}): combined full\textendash page coverage and CI\textendash width panels across DGPs and sample sizes.
    \item Asymptotic efficiency (Section~\ref{app_i_subsec:asymptotic_efficiency}): efficiency comparisons for TN, GS3, GA3, GS5, Step, and Sine.
    \item EIC\textendash based variance estimation for statistical estimands (Section~\ref{app_i_subsec:eic-based-variance-estimation}): per\textendash DGP CI width and coverage figures.
\end{itemize}

\subsection{Asymptotic normality and variance estimation}\label{app_i_subsec:asymptotic_normality_and_var_est}
\begin{figure}[p]
    \centering
    \includegraphics[angle=90, height=0.90\textheight]{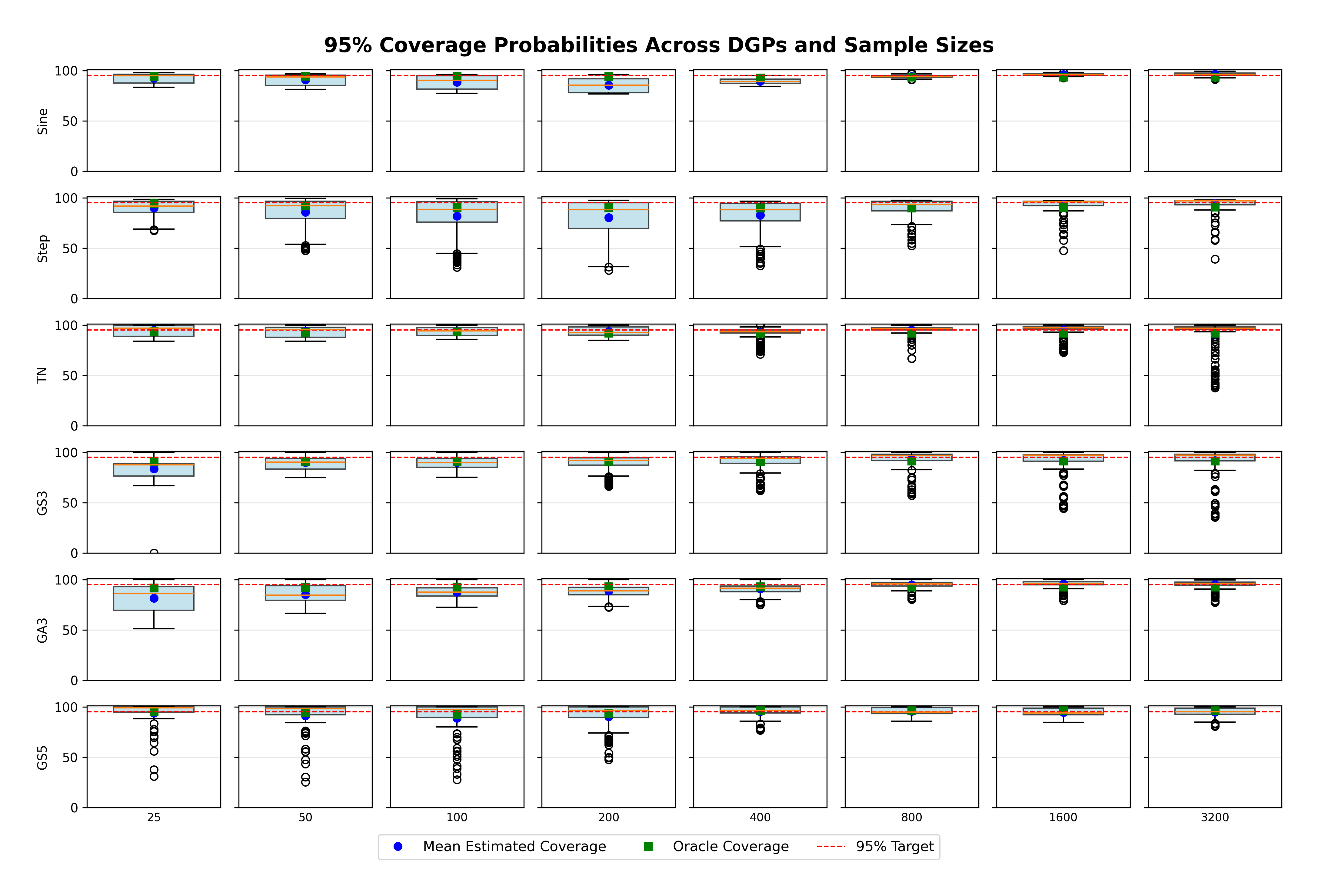}
    \caption{95\% coverage probabilities across DGPs and sample sizes (rotated for full\textendash page display).}
\end{figure}

\begin{figure}[p]
    \centering
    \includegraphics[angle=90, height=0.90\textheight]{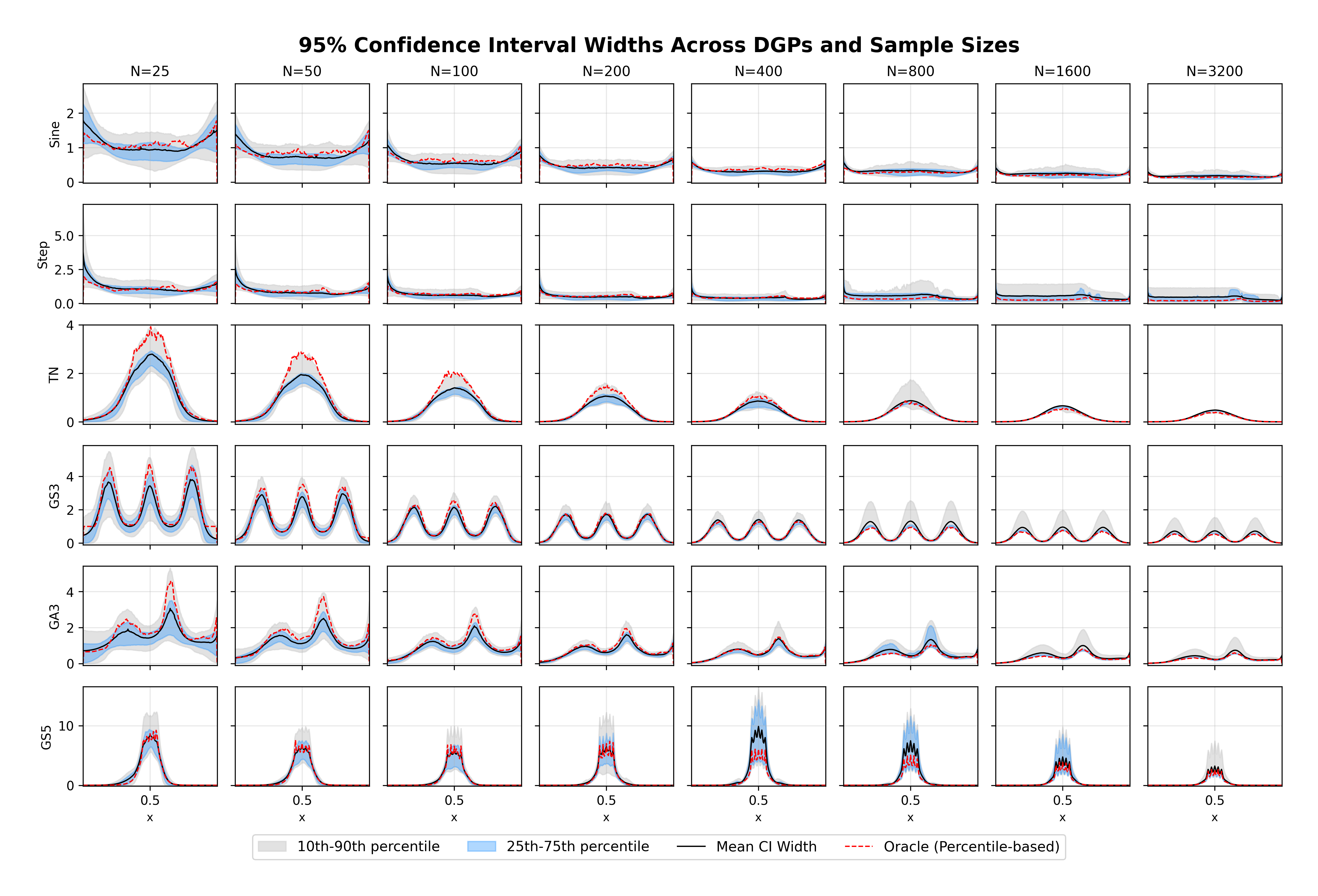}
    \caption{95\% confidence\textendash interval widths across DGPs and sample sizes (rotated for full\textendash page display).}
\end{figure}
\clearpage

\subsection{Asymptotic efficiency}\label{app_i_subsec:asymptotic_efficiency}
\begin{figure}[H]
    \centering
    \includegraphics[width=0.9\textwidth]{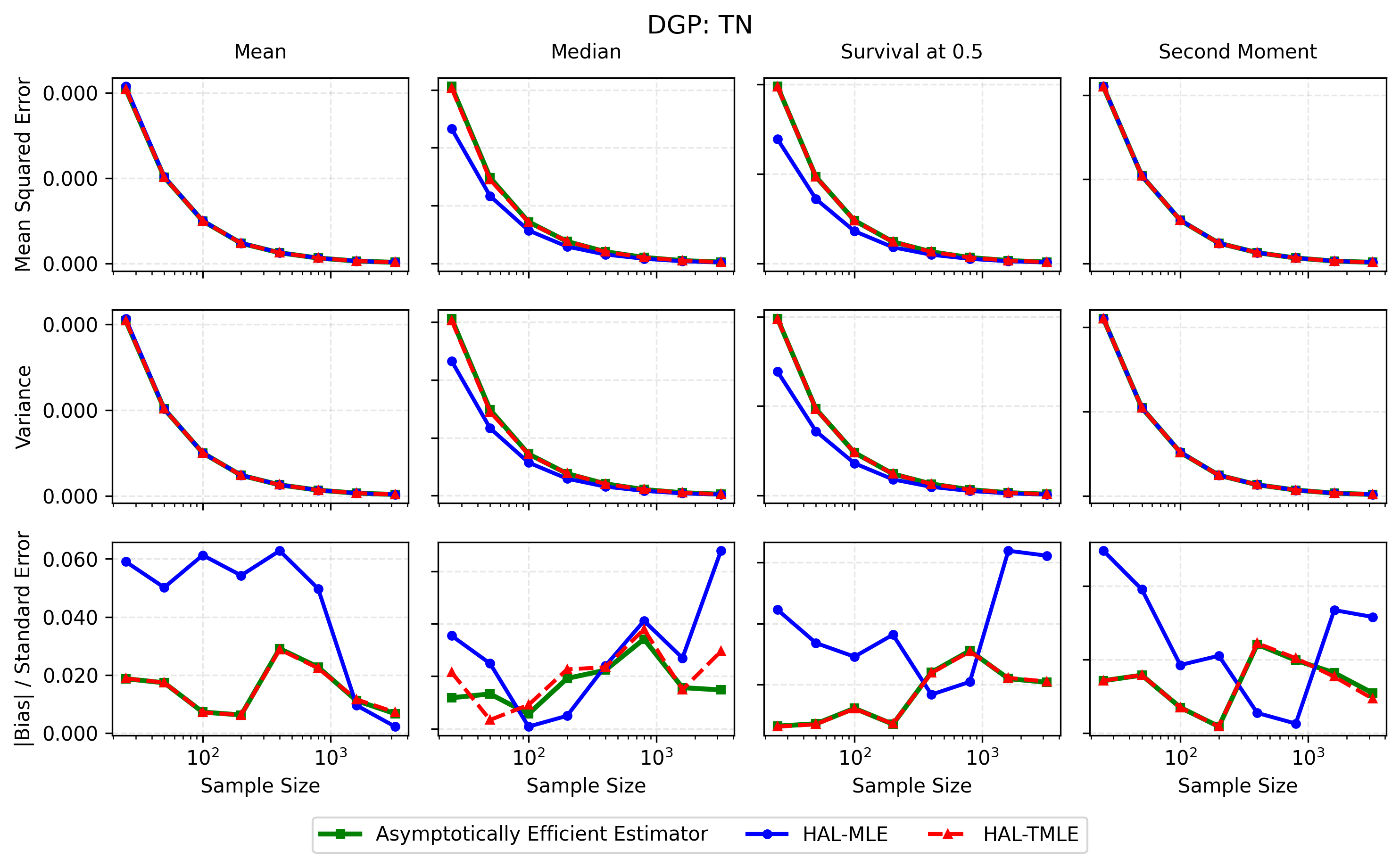}
    \caption{Asymptotic\textendash efficiency comparison for TN.}
\end{figure}

\begin{figure}[H]
    \centering
    \includegraphics[width=0.9\textwidth]{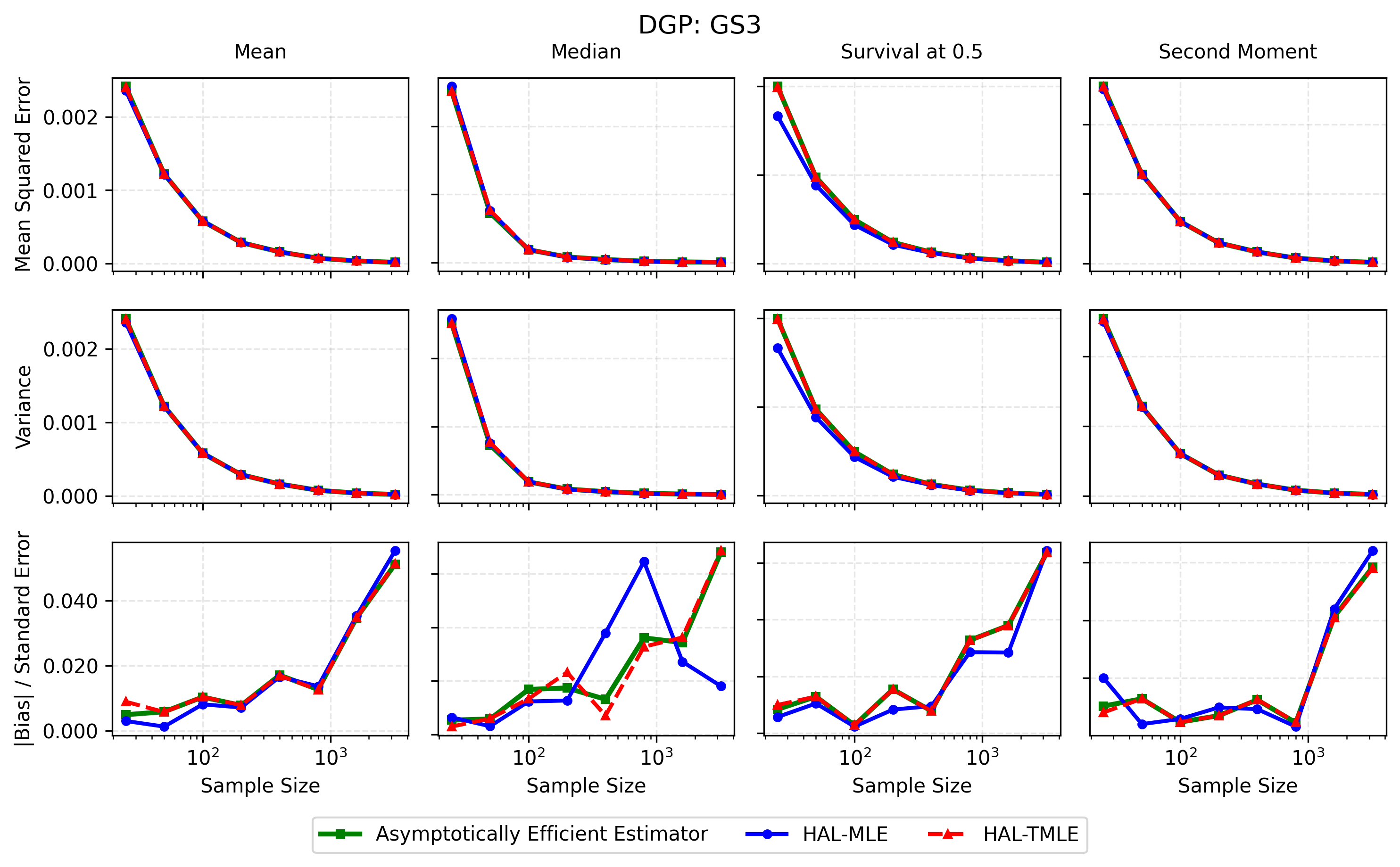}
    \caption{Asymptotic\textendash efficiency comparison for GS3.}
\end{figure}

\begin{figure}[H]
    \centering
    \includegraphics[width=0.9\textwidth]{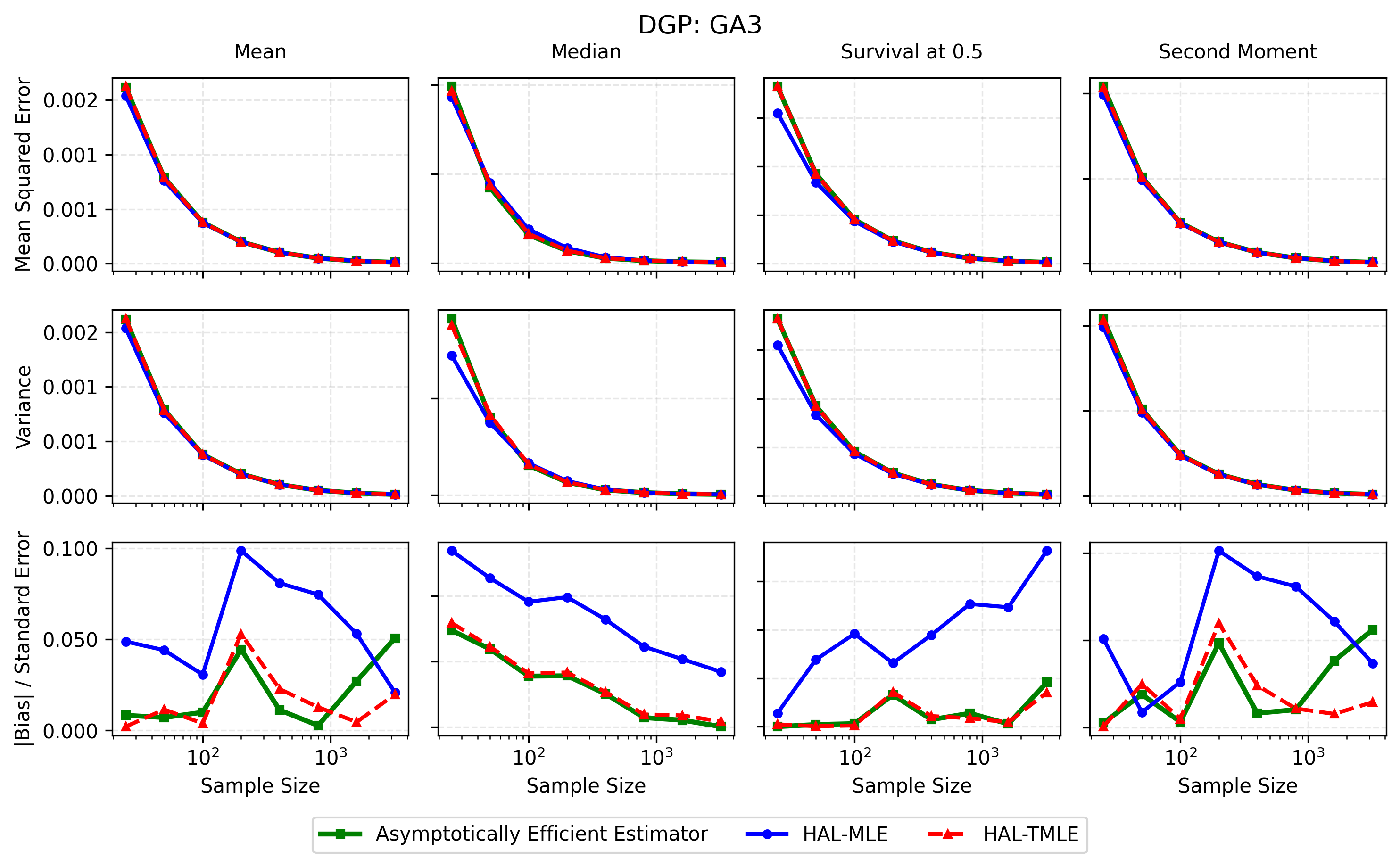}
    \caption{Asymptotic\textendash efficiency comparison for GA3.}
\end{figure}

\begin{figure}[H]
    \centering
    \includegraphics[width=0.9\textwidth]{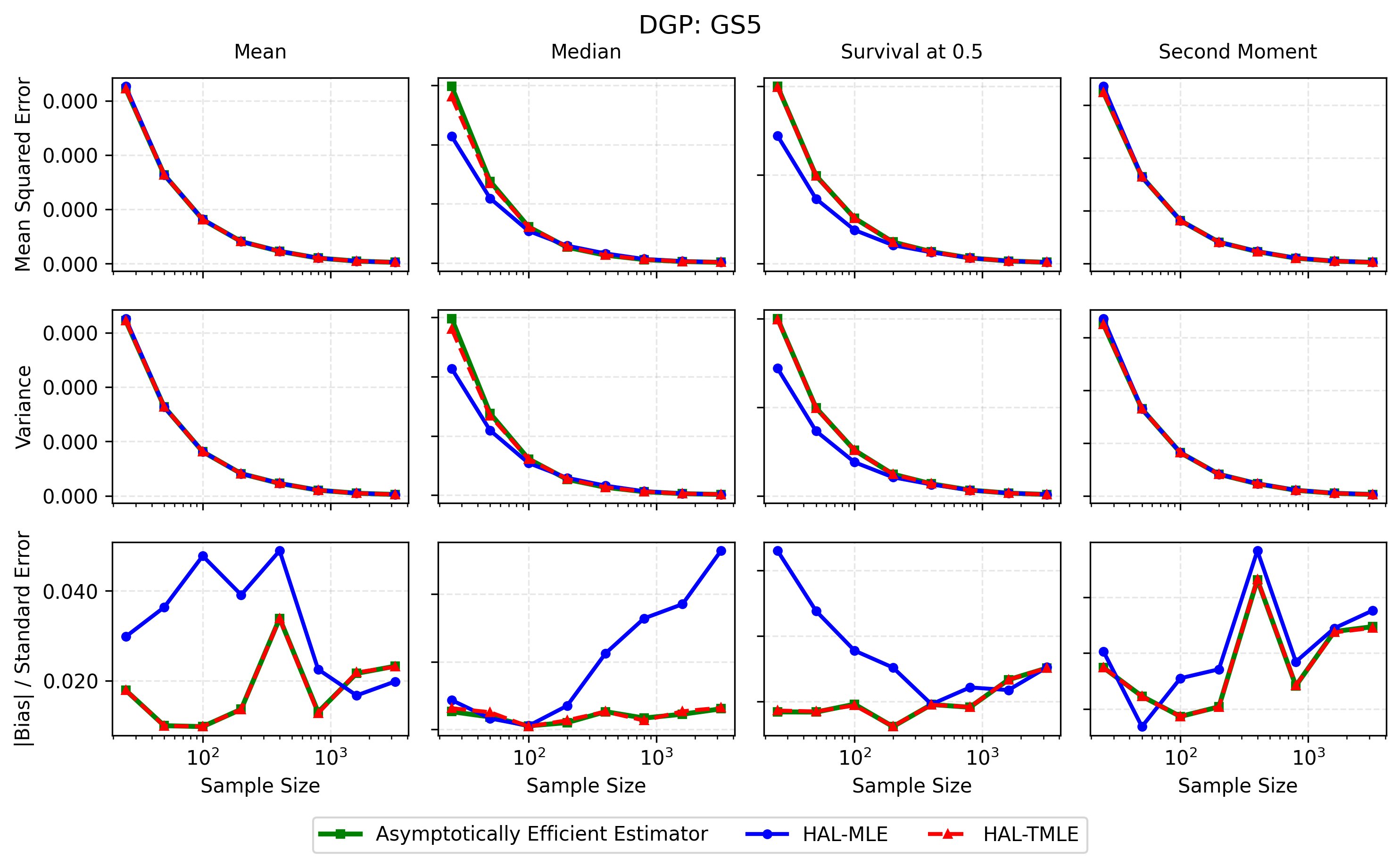}
    \caption{Asymptotic\textendash efficiency comparison for GS5.}
\end{figure}

\begin{figure}[H]
    \centering
    \includegraphics[width=0.9\textwidth]{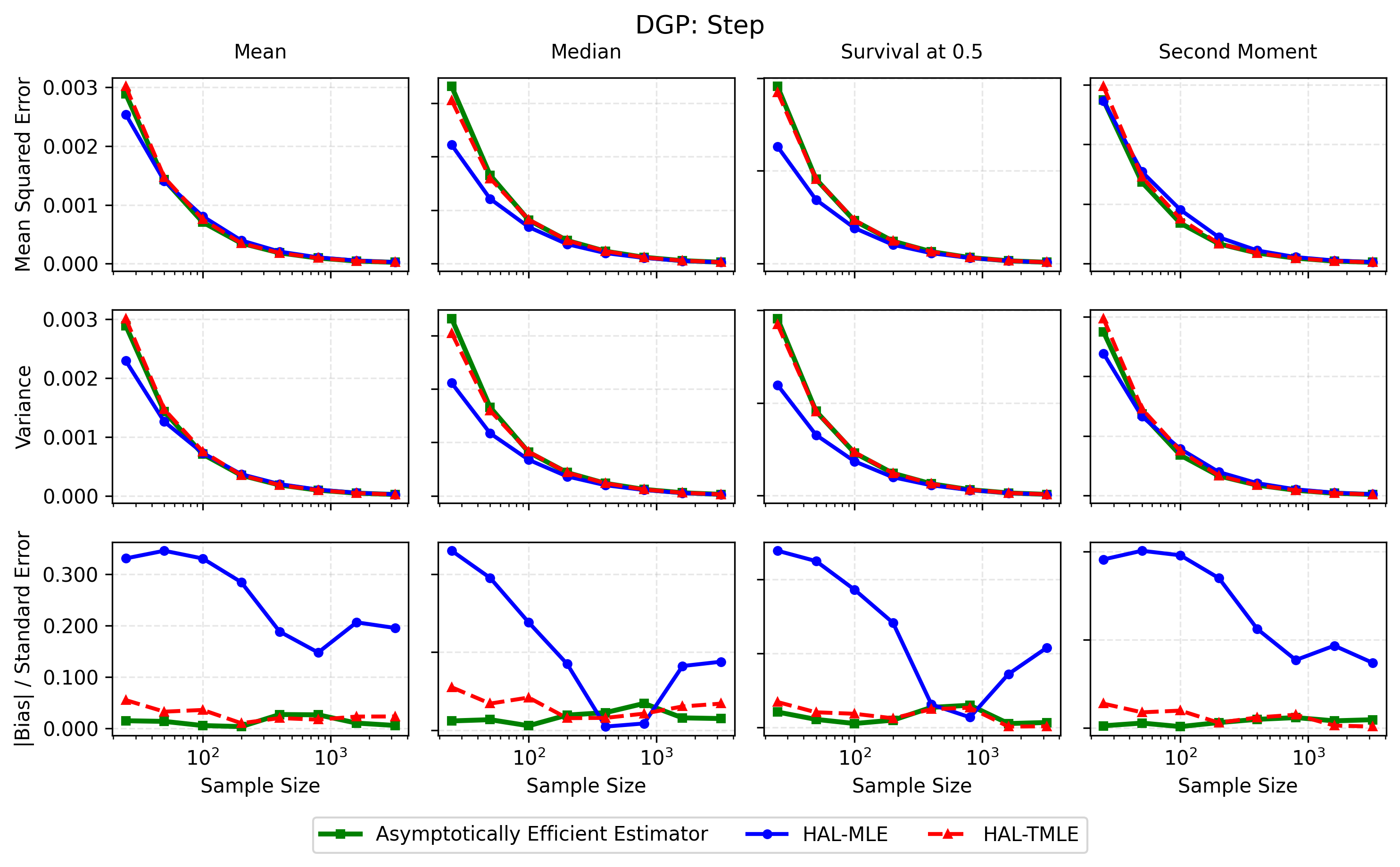}
    \caption{Asymptotic\textendash efficiency comparison for Step.}
\end{figure}

\begin{figure}[H]
    \centering
    \includegraphics[width=0.9\textwidth]{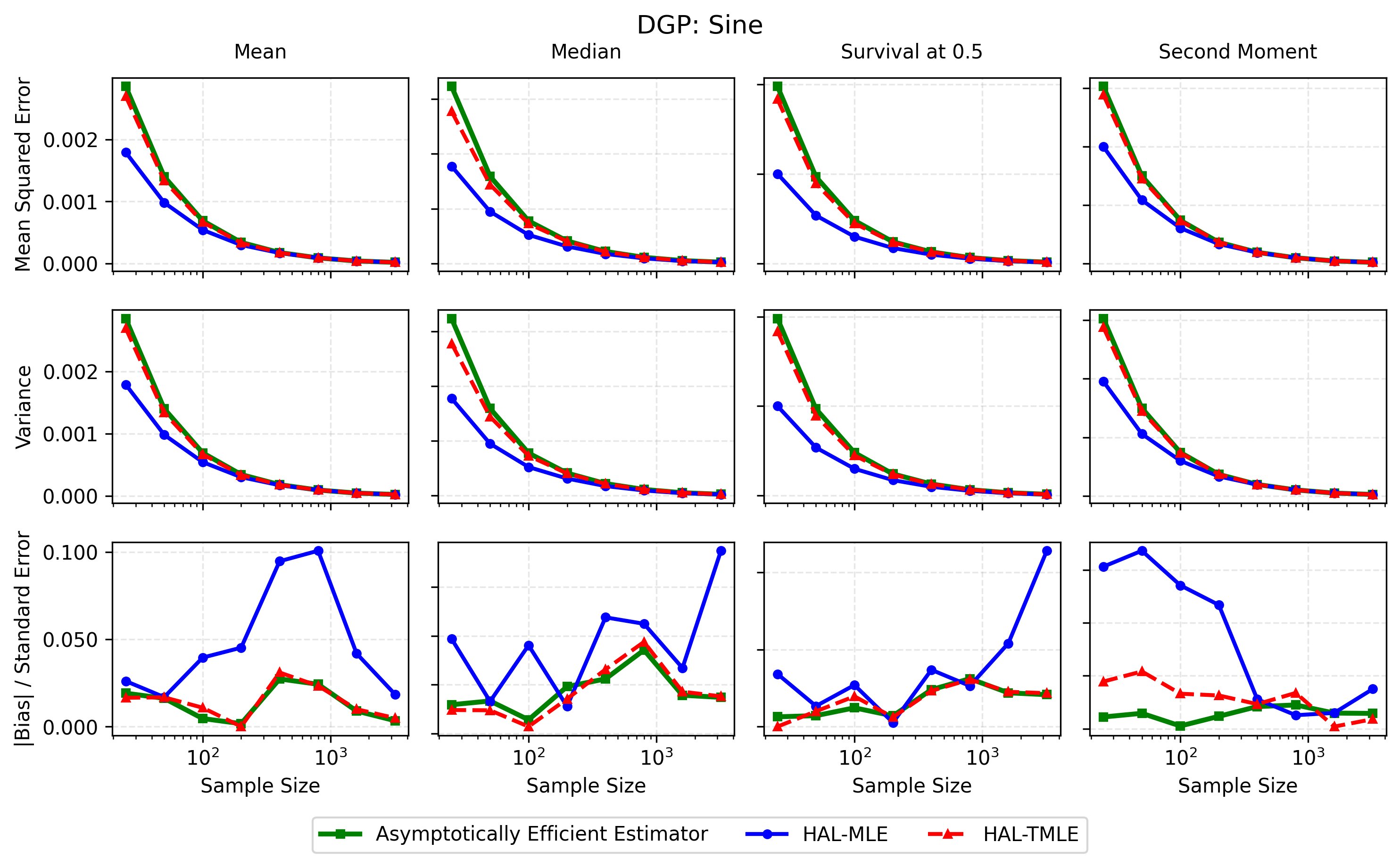}
    \caption{Asymptotic\textendash efficiency comparison for Sine.}
\end{figure}

\subsection{EIC\textendash based variance estimation for statistical estimands}\label{app_i_subsec:eic-based-variance-estimation}

\begin{figure}[H]
    \centering
    \includegraphics[width=0.75\textwidth]{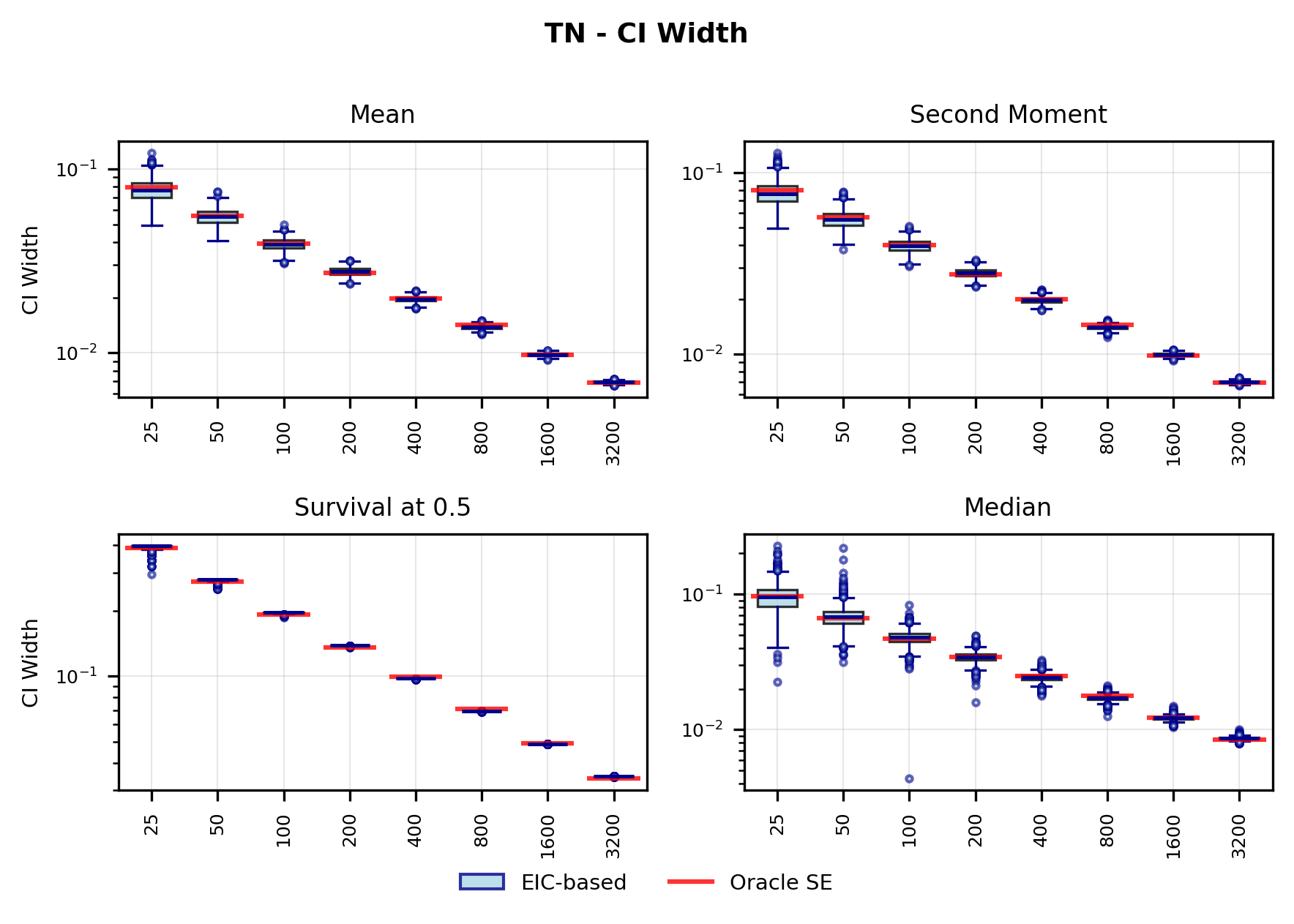}
    \caption{EIC\textendash based variance estimation for TN: CI width.}
\end{figure}

\begin{figure}[H]
    \centering
    \includegraphics[width=0.75\textwidth]{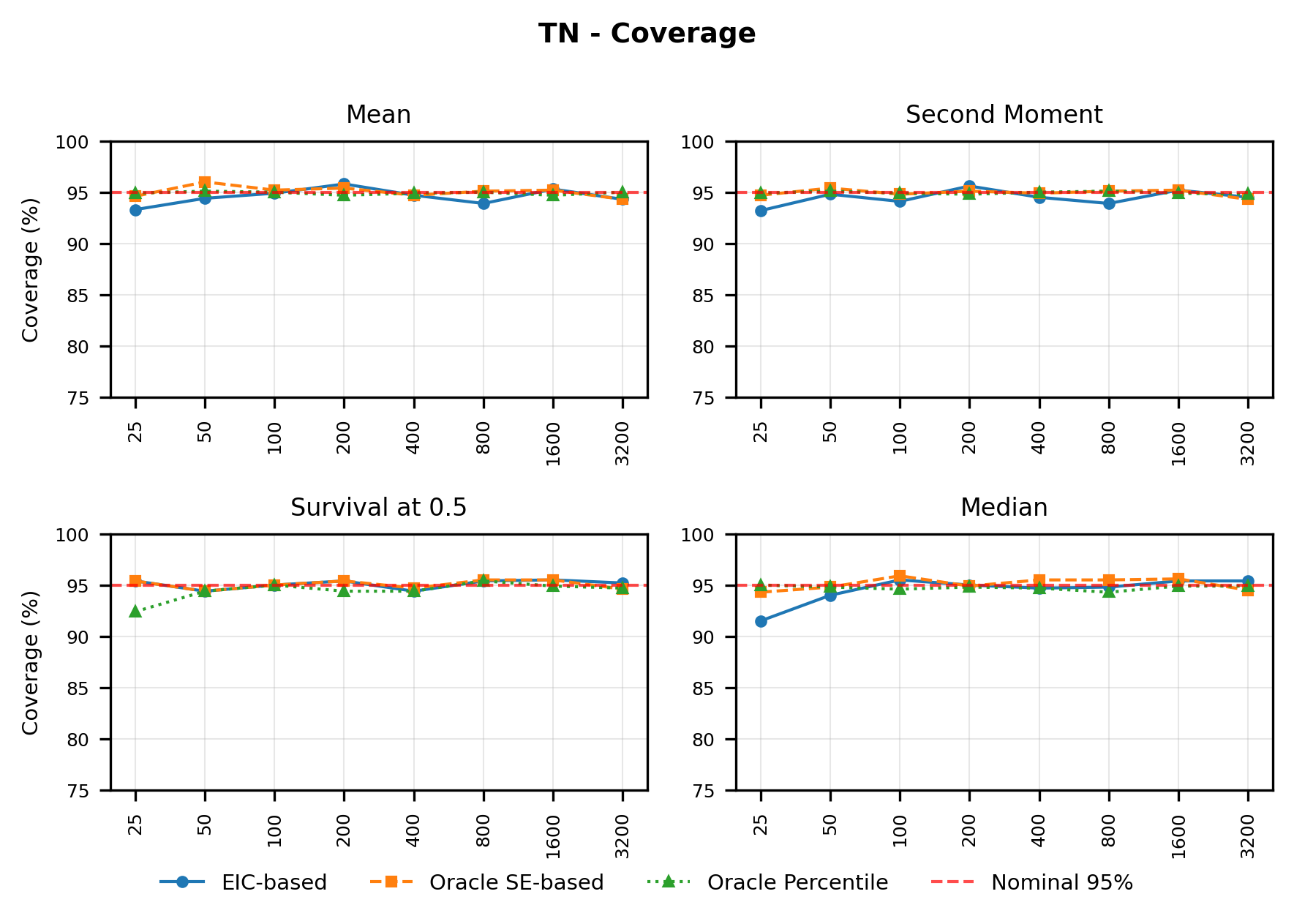}
    \caption{EIC\textendash based variance estimation for TN: coverage.}
\end{figure}

\begin{figure}[H]
    \centering
    \includegraphics[width=0.75\textwidth]{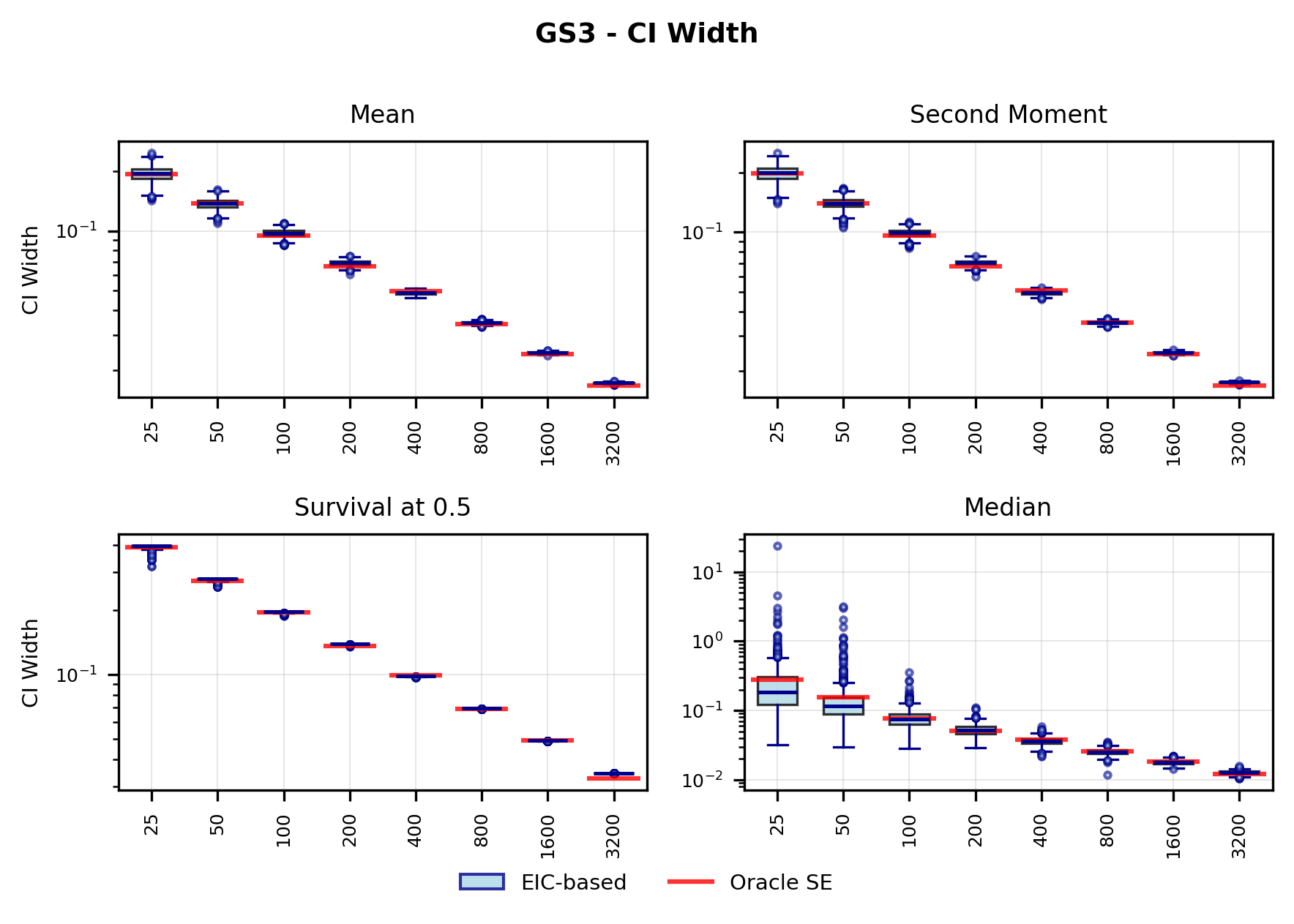}
    \caption{EIC\textendash based variance estimation for GS3: CI width.}
\end{figure}

\begin{figure}[H]
    \centering
    \includegraphics[width=0.75\textwidth]{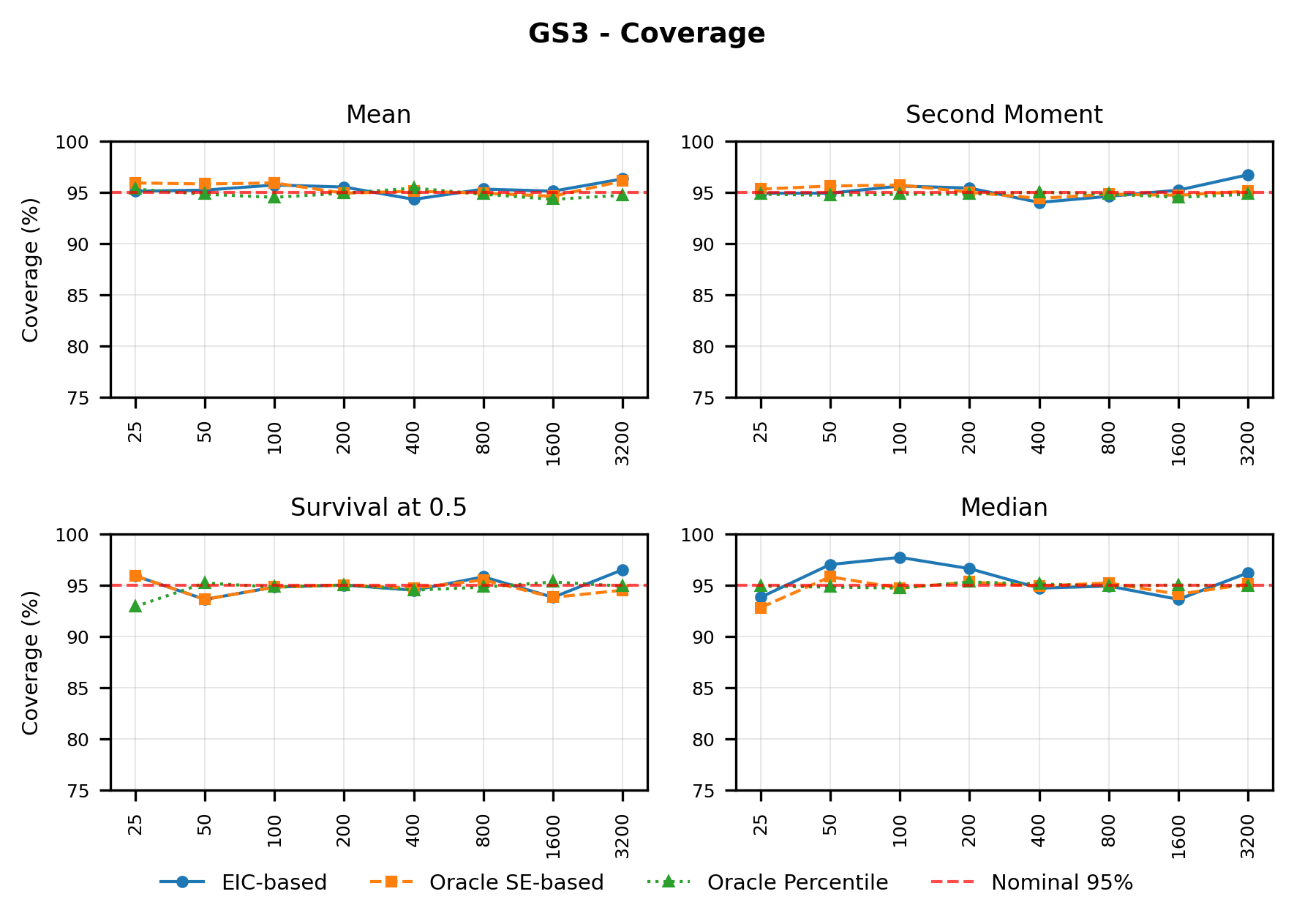}
    \caption{EIC\textendash based variance estimation for GS3: coverage.}
\end{figure}

\begin{figure}[H]
    \centering
    \includegraphics[width=0.75\textwidth]{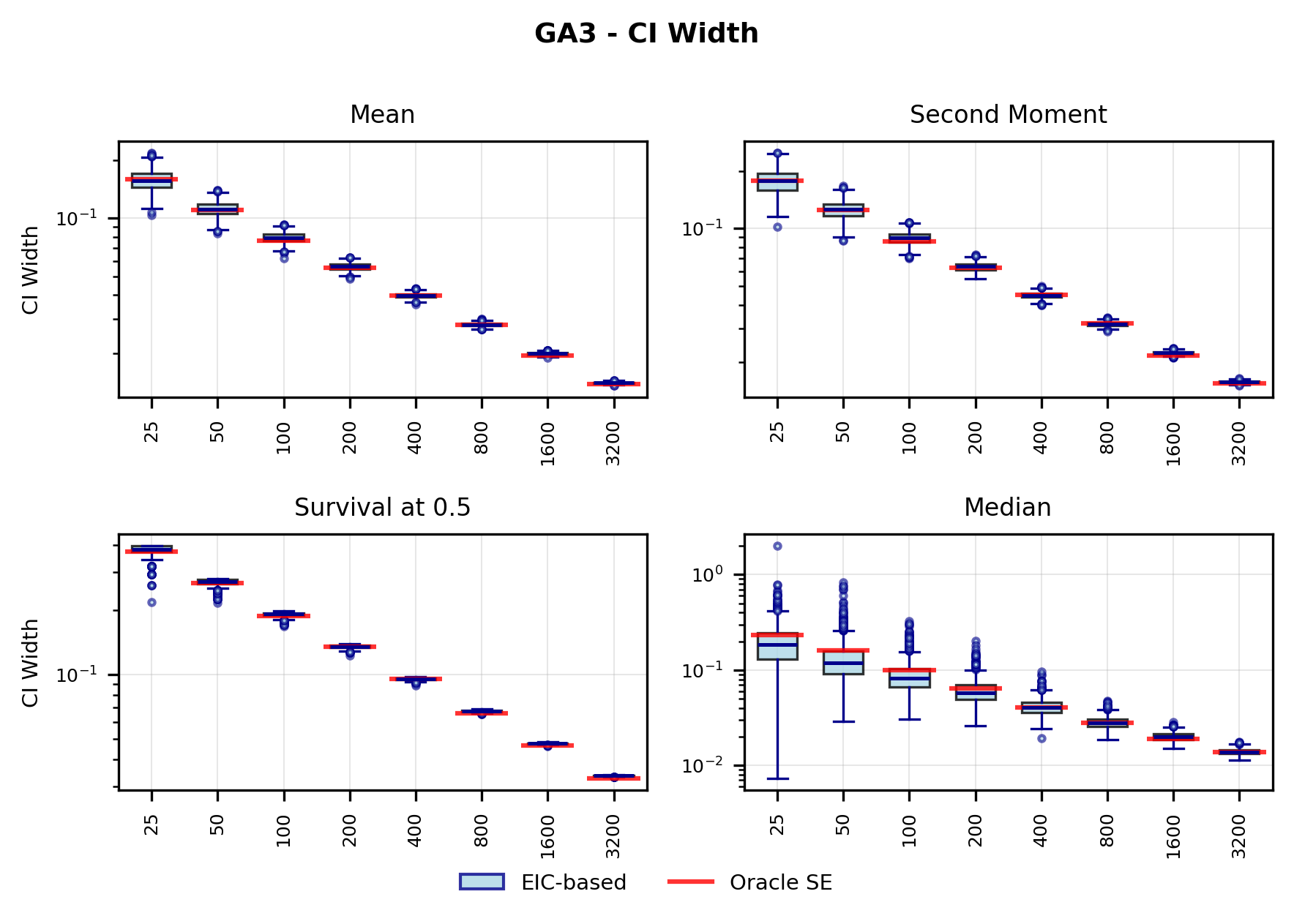}
    \caption{EIC\textendash based variance estimation for GA3: CI width.}
\end{figure}

\begin{figure}[H]
    \centering
    \includegraphics[width=0.75\textwidth]{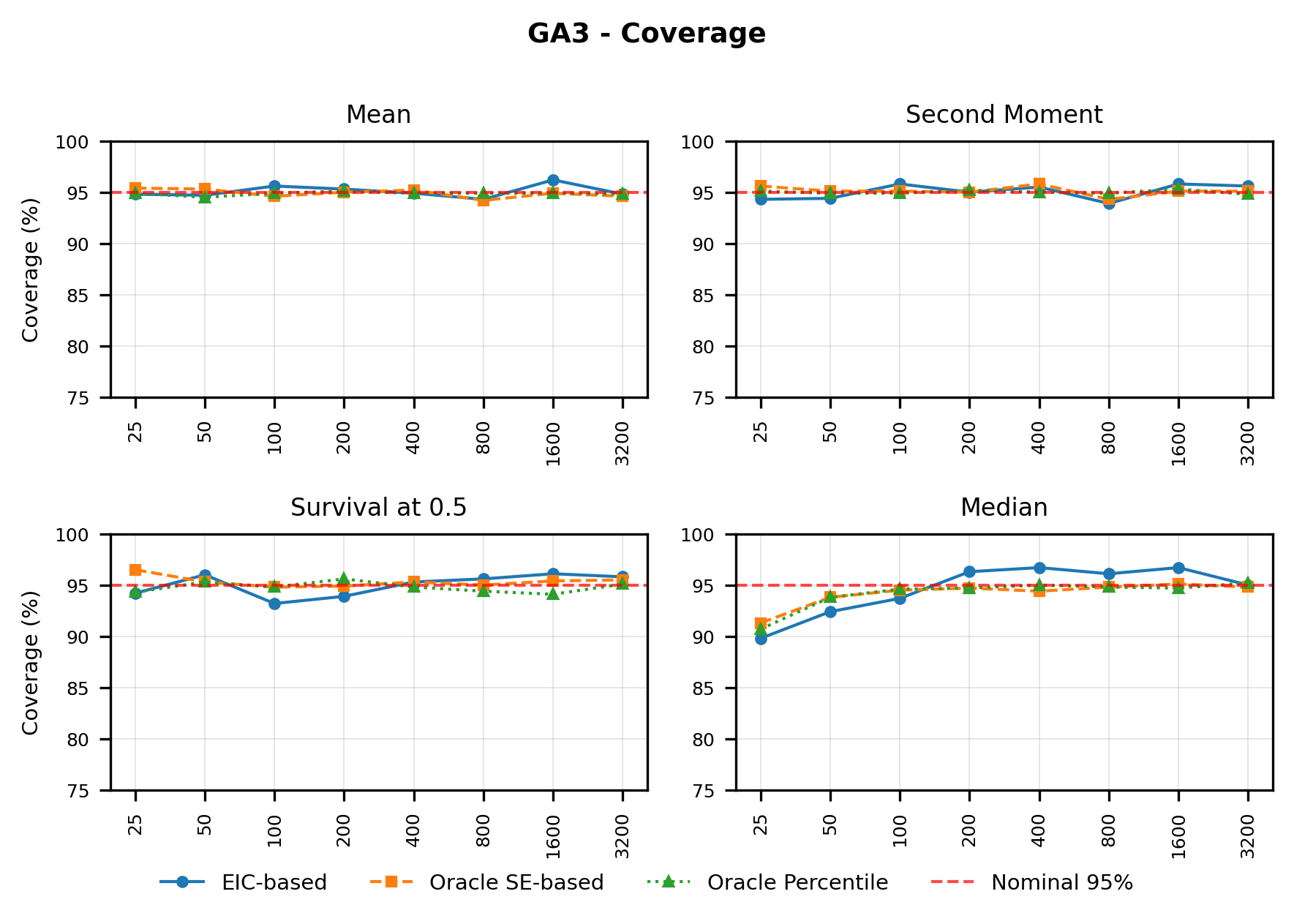}
    \caption{EIC\textendash based variance estimation for GA3: coverage.}
\end{figure}

\begin{figure}[H]
    \centering
    \includegraphics[width=0.75\textwidth]{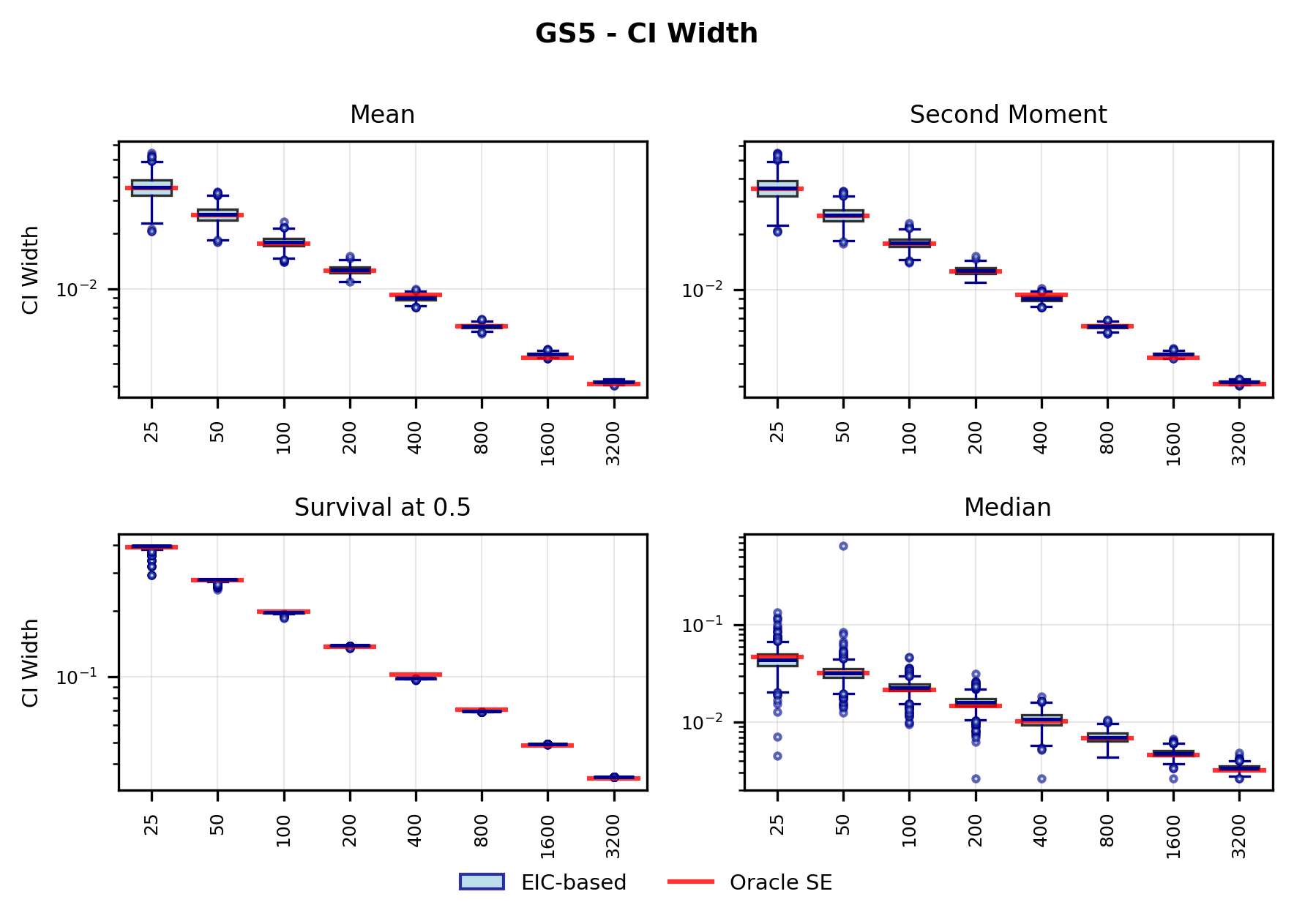}
    \caption{EIC\textendash based variance estimation for GS5: CI width.}
\end{figure}

\begin{figure}[H]
    \centering
    \includegraphics[width=0.75\textwidth]{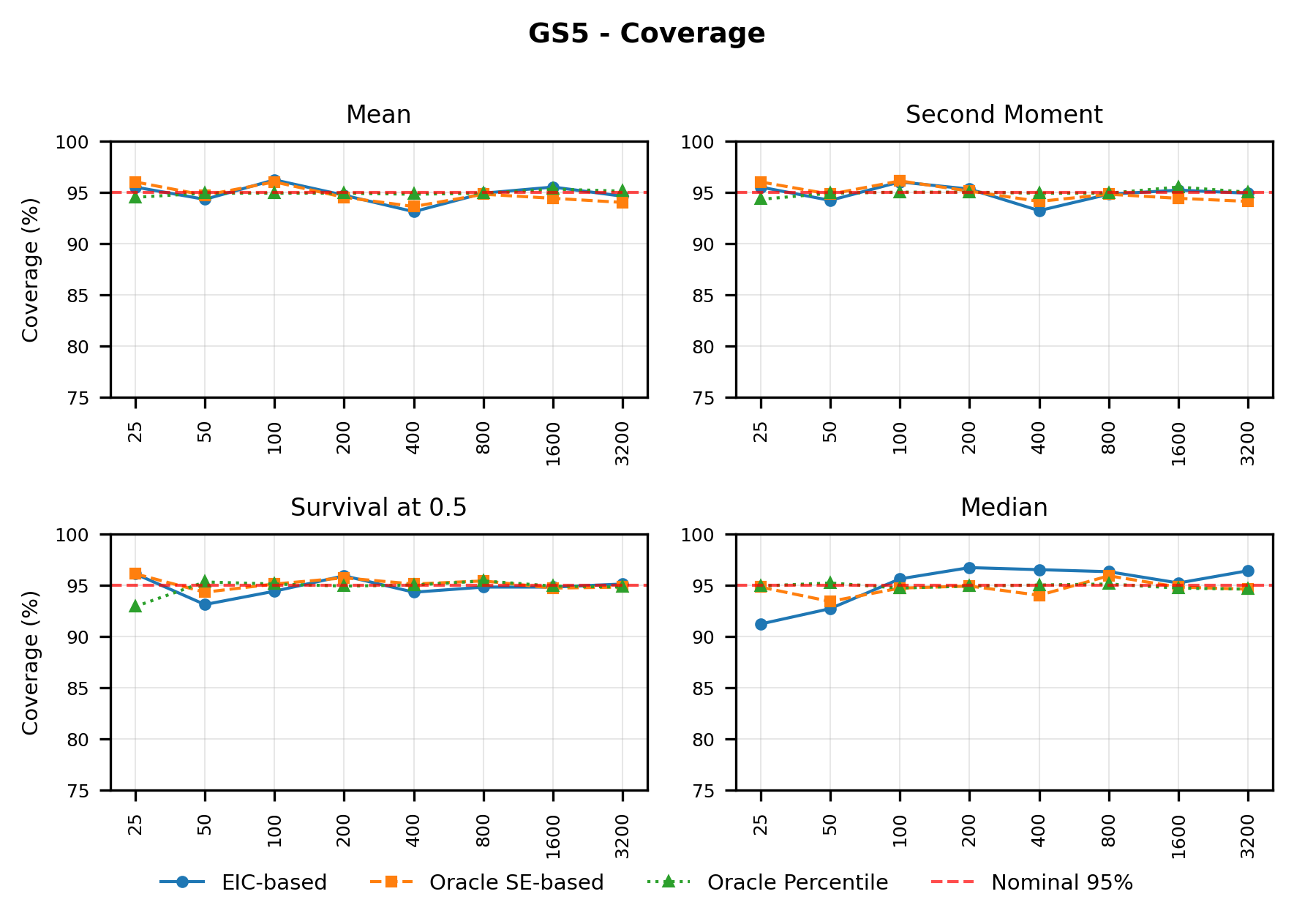}
    \caption{EIC\textendash based variance estimation for GS5: coverage.}
\end{figure}

\begin{figure}[H]
    \centering
    \includegraphics[width=0.75\textwidth]{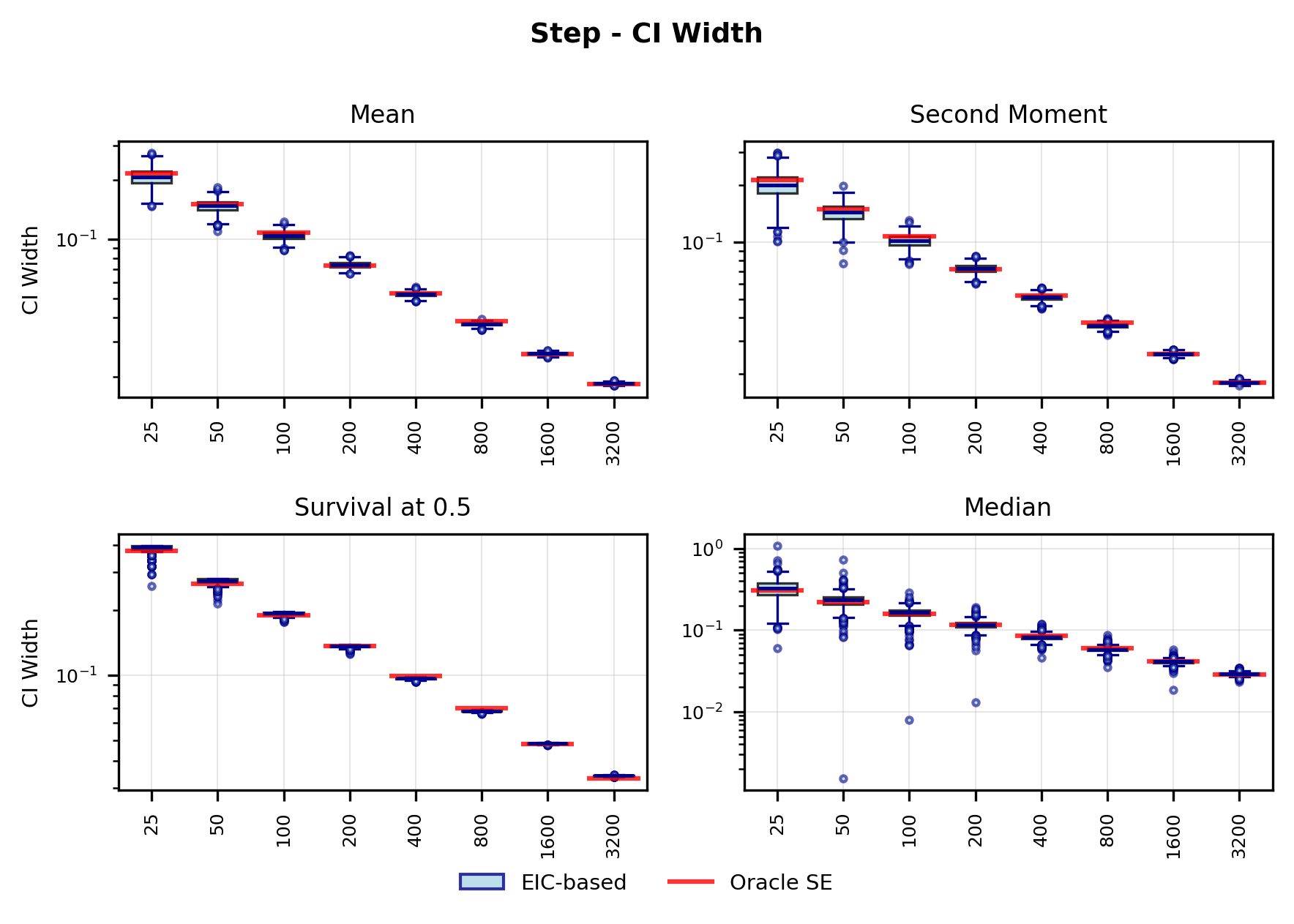}
    \caption{EIC\textendash based variance estimation for Step: CI width.}
\end{figure}

\begin{figure}[H]
    \centering
    \includegraphics[width=0.75\textwidth]{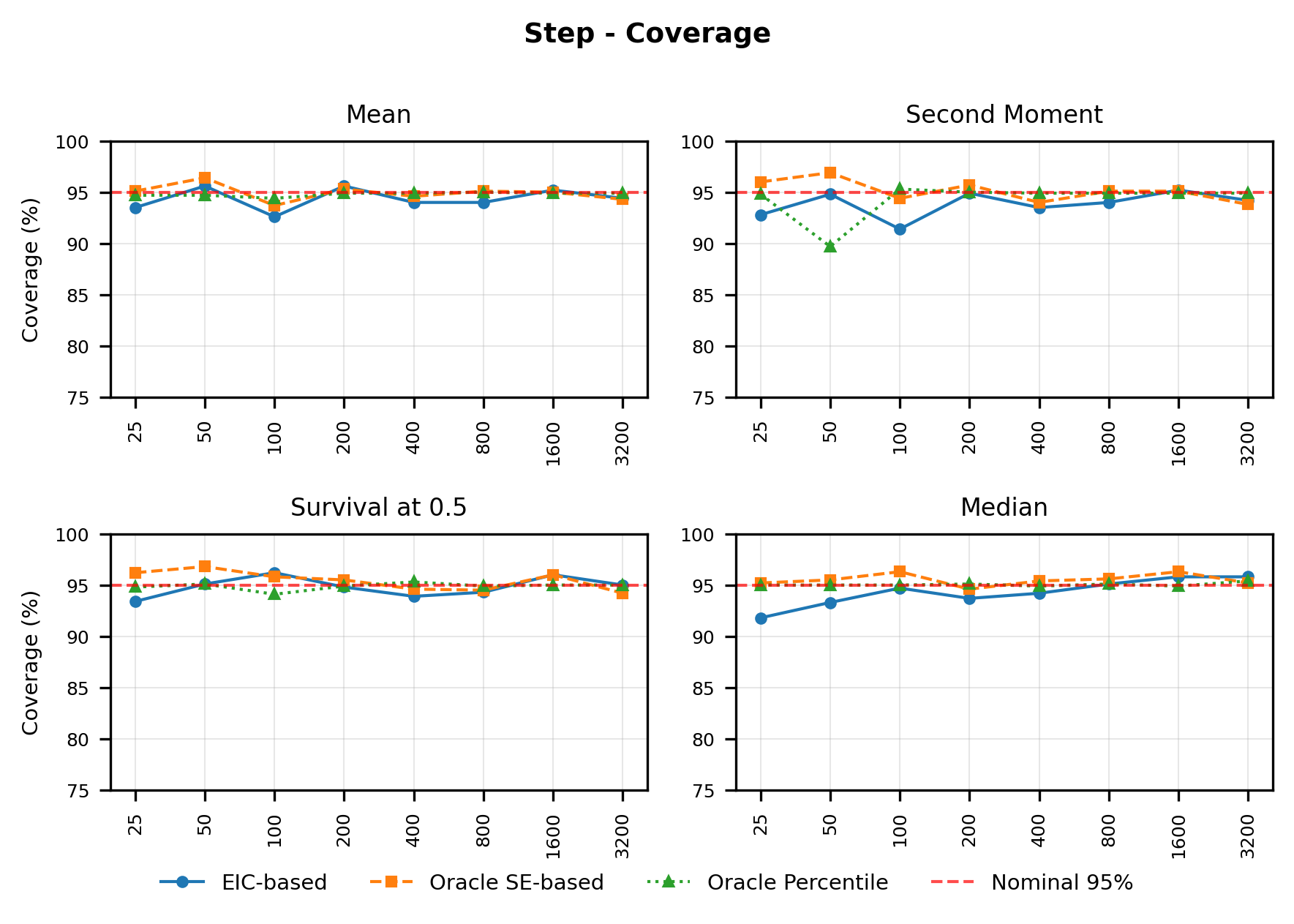}
    \caption{EIC\textendash based variance estimation for Step: coverage.}
\end{figure}

\begin{figure}[H]
    \centering
    \includegraphics[width=0.75\textwidth]{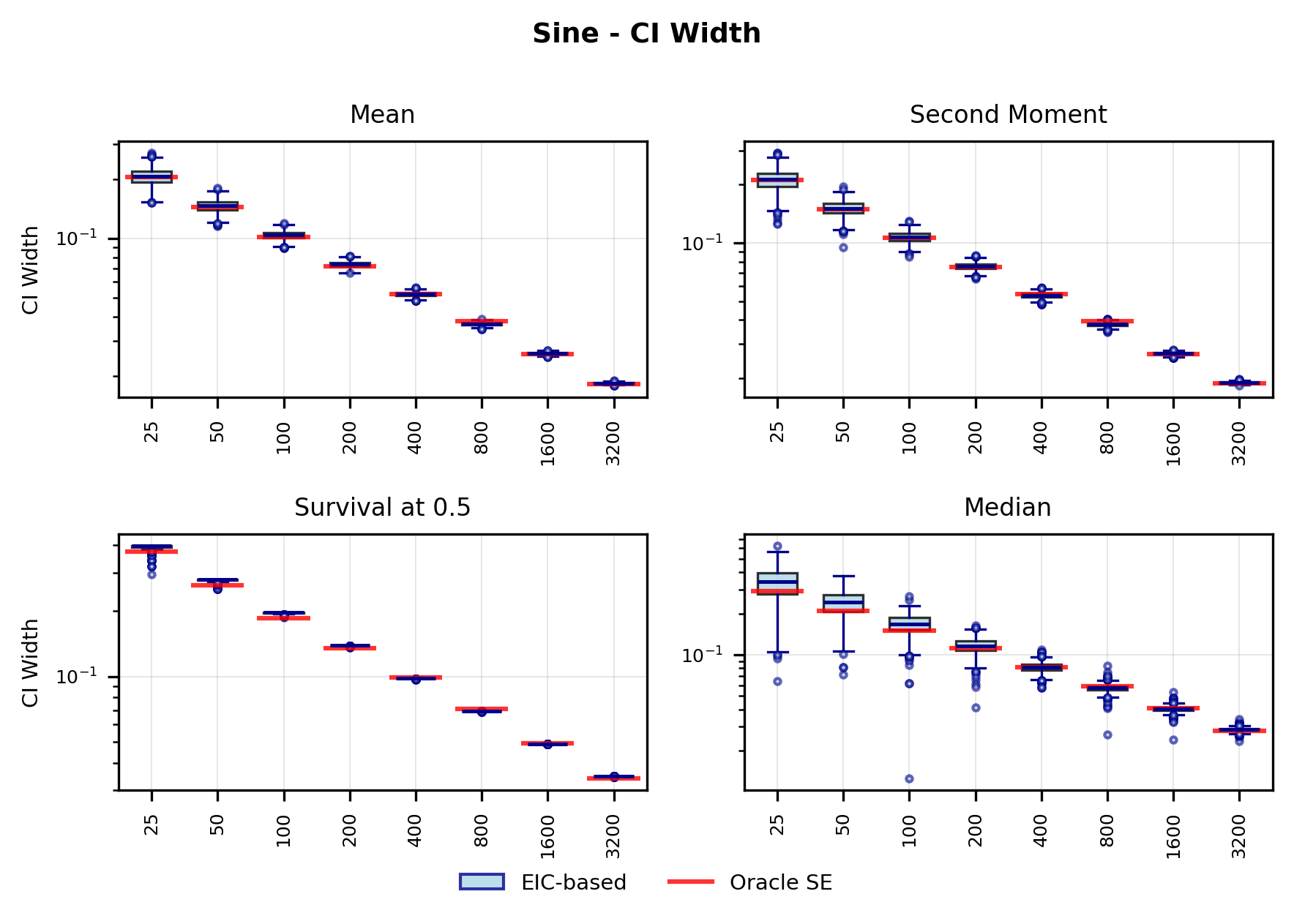}
    \caption{EIC\textendash based variance estimation for Sine: CI width.}
\end{figure}

\begin{figure}[H]
    \centering
    \includegraphics[width=0.75\textwidth]{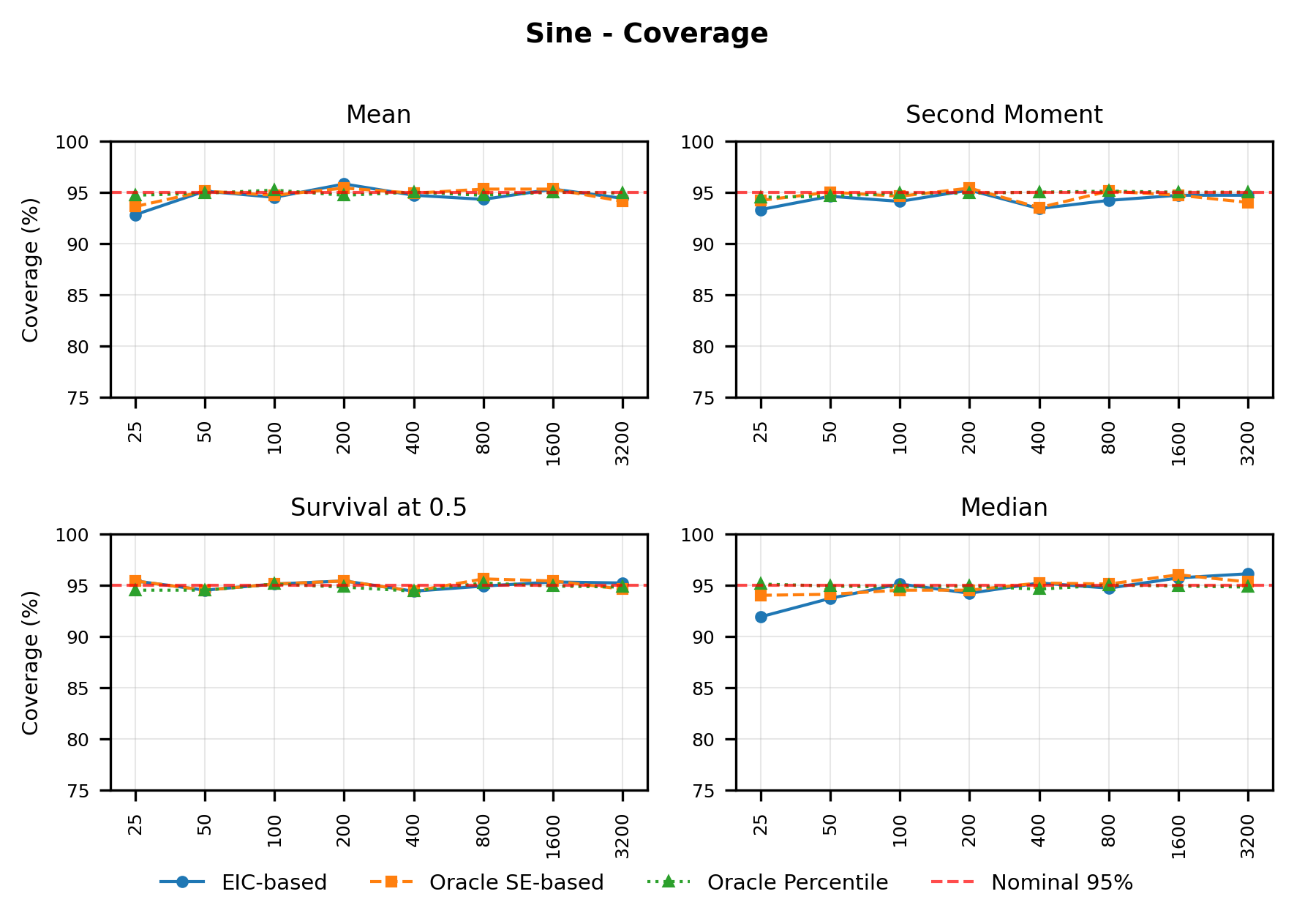}
    \caption{EIC\textendash based variance estimation for Sine: coverage.}
\end{figure}

\clearpage

\newpage

\section{Optimization Algorithm Performance Analysis} \label{app:knot-selection}

This appendix provides comprehensive visualization of optimization algorithm performance across all six data-generating processes (DGPs) and basis orders. The figures show both knot selection behavior and loss convergence patterns for four optimization algorithms: FISTA, ProximalAdaGrad, ProximalNewton, and ProximalNewtonLBFGS, compared against the \texttt{CVXPY} reference solution.

\subsection{Notation and setup}
We denote by $n$ the number of samples, $G$ the number of midpoint intervals used for normalization on $[0,1]$, $k \in \{0,1,2\}$ the basis order, $K$ the total number of basis coefficients (intercept, $k$ polynomial terms, and $M$ truncated\,-\,power terms), $s$ the active set size used in coordinate descent, and $L_k$ the number of objective evaluations performed by backtracking at iteration $k$. Throughout, matrix\,-\,vector multiplies with dense $a\times b$ arrays are costed as $2ab$ FLOPs.

\subsection{Per-iteration FLOP derivations}
We summarize implementation\,-\,aware FLOP counts for each algorithm; see the project note  \texttt{experiments/compare\_knot\_selection/algorithm\_flop\_per\_iter.md} in the experiment git repo for additional commentary and logging details.

\paragraph{FISTA.} A forward pass on data and midpoints and a backward pass via autograd each cost $\Theta\big((n+G)K\big)$, with an additional $\Theta(GK)$ for the exact normalization update. Dominant term: $\Theta\big((n+G)K\big)$.

\paragraph{Proximal AdaGrad.} Computing the gradient uses $\Theta\big(nK + GK\big)$; forming the diagonal Hessian via variance on midpoints contributes another $\Theta(GK)$ (tighter bound) to $\Theta(3GK)$ (conservative). Including the normalization step yields a dominant $\Theta\big((n+G)K\big)$.

\paragraph{Proximal Newton (full).} Each iteration forms the gradient in $\Theta\big((n+G)K\big)$, the Hessian as a weighted covariance on midpoints in $\Theta\big(GK^2\big)$, solves the proximal subproblem via coordinate descent in $\Theta\big(sK^2\big)$, and performs $L_k$ objective evaluations contributing $\Theta\big(L_k (n+G)K\big)$; normalization adds $\Theta(GK)$. Dominant term: $\Theta\big(GK^2 + sK^2\big) + \Theta\big((1+L_k)(n+G)K\big)$.

\paragraph{Proximal Newton L\,-\,BFGS.} Replacing the full Hessian with an L\,-\,BFGS update avoids the $GK^2$ term; per iteration is dominated by the gradient and line search plus normalization: $\Theta\big((1+L_k)(n+G)K\big)$.

\subsection{From iterations to cumulative FLOPs}
Given per\,-\,iteration FLOP estimates $F(k)$, cumulative FLOPs up to iteration $t$ are $C(t) = \sum_{k=1}^t F(k)$. Our visualization utilities infer $L_k$ from logged step sizes $\alpha_k$ (with backtracking factor $\beta=0.5$) and use consistent baselines at iteration/FLOP zero. See the note for practical defaults and alignment details.

\begin{figure}[H]
    \centering
    \includegraphics[width=.49\textwidth]{resources/optimization_algorithms/per_iter/TruncatedNormal_order_2_loss_per_iter.png}
    \hfill
    \includegraphics[width=.49\textwidth]{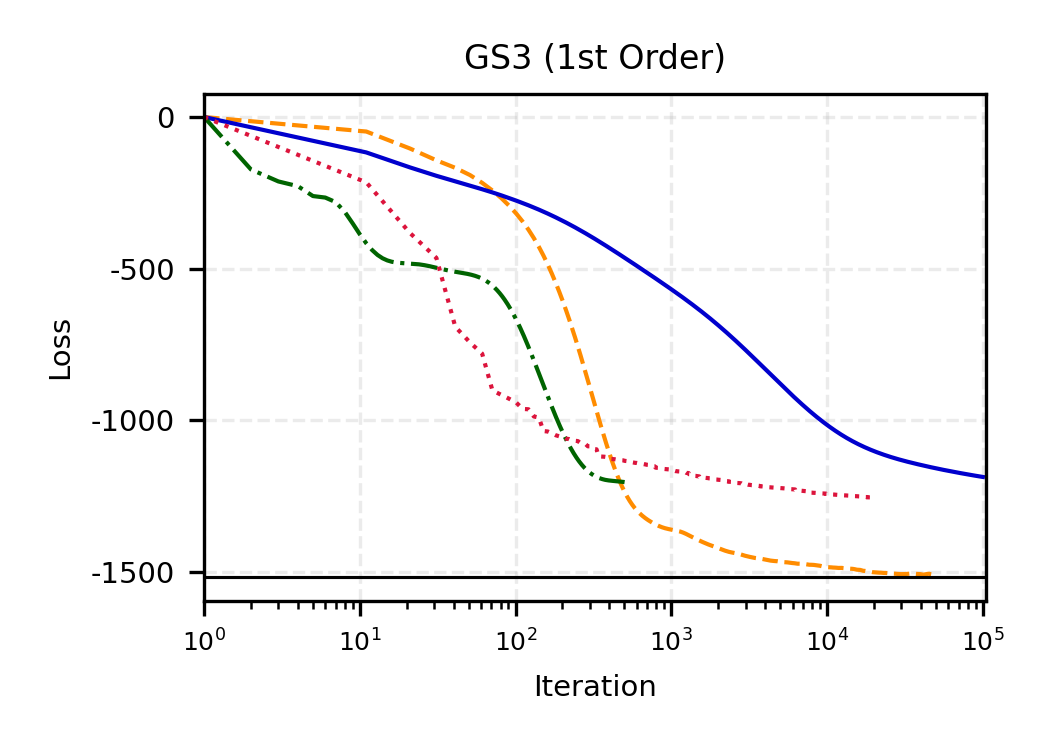}
    
    \includegraphics[width=.49\textwidth]{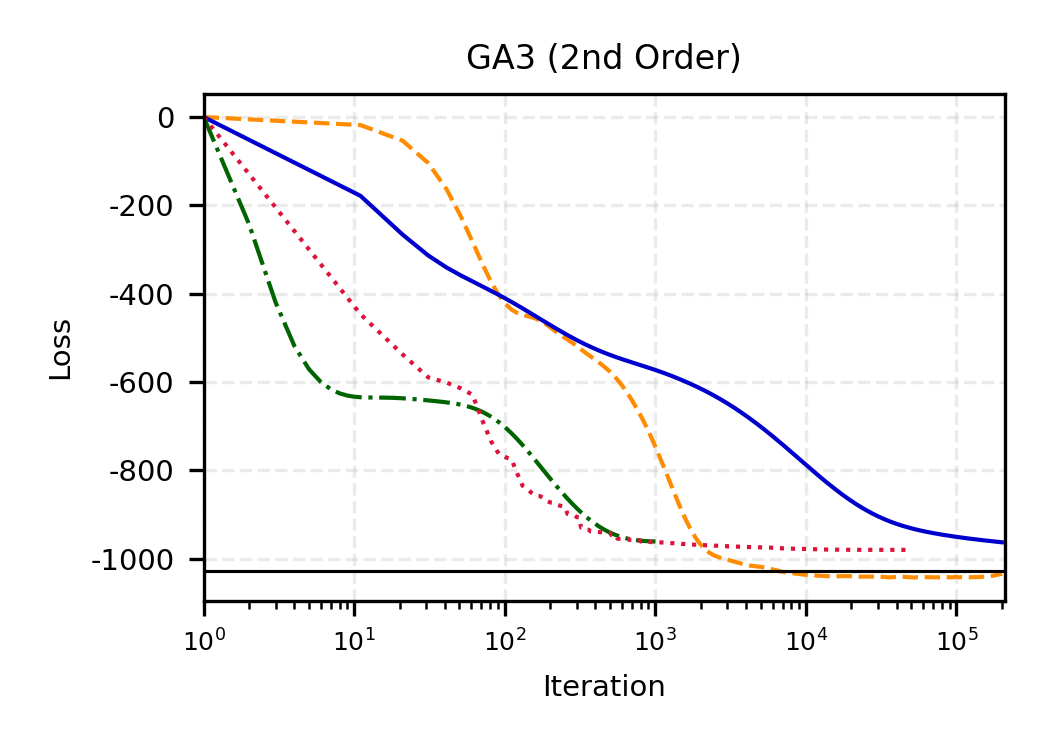}
    \hfill
    \includegraphics[width=.49\textwidth]{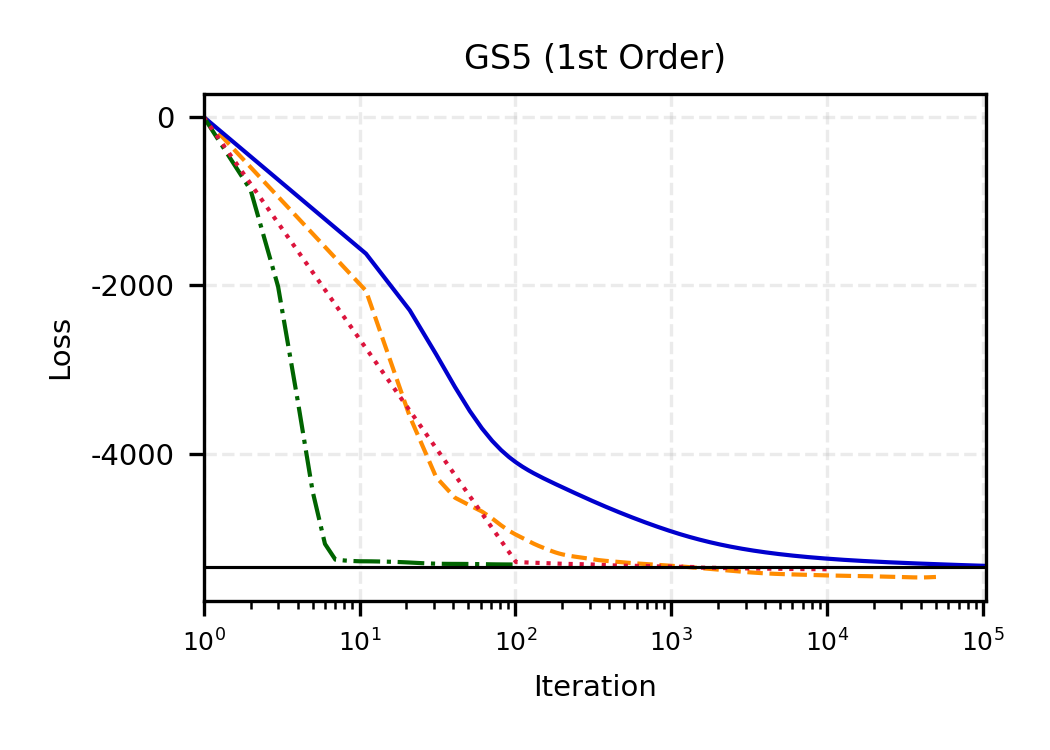}
    
    \includegraphics[width=.49\textwidth]{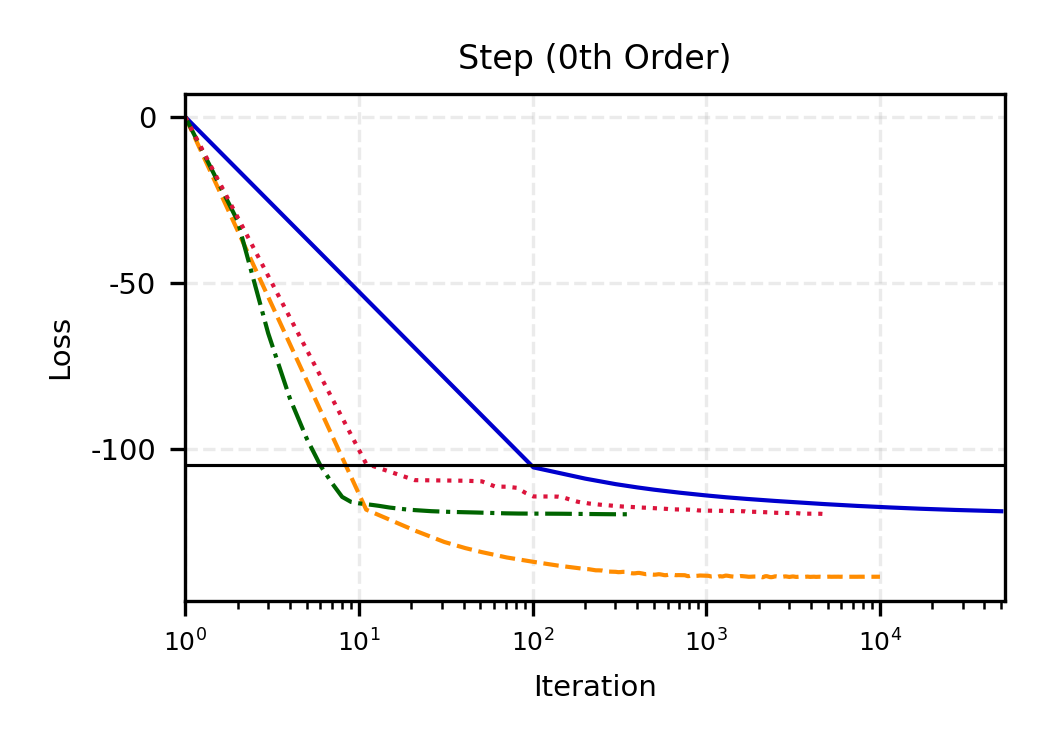}
    \hfill
    \includegraphics[width=.49\textwidth]{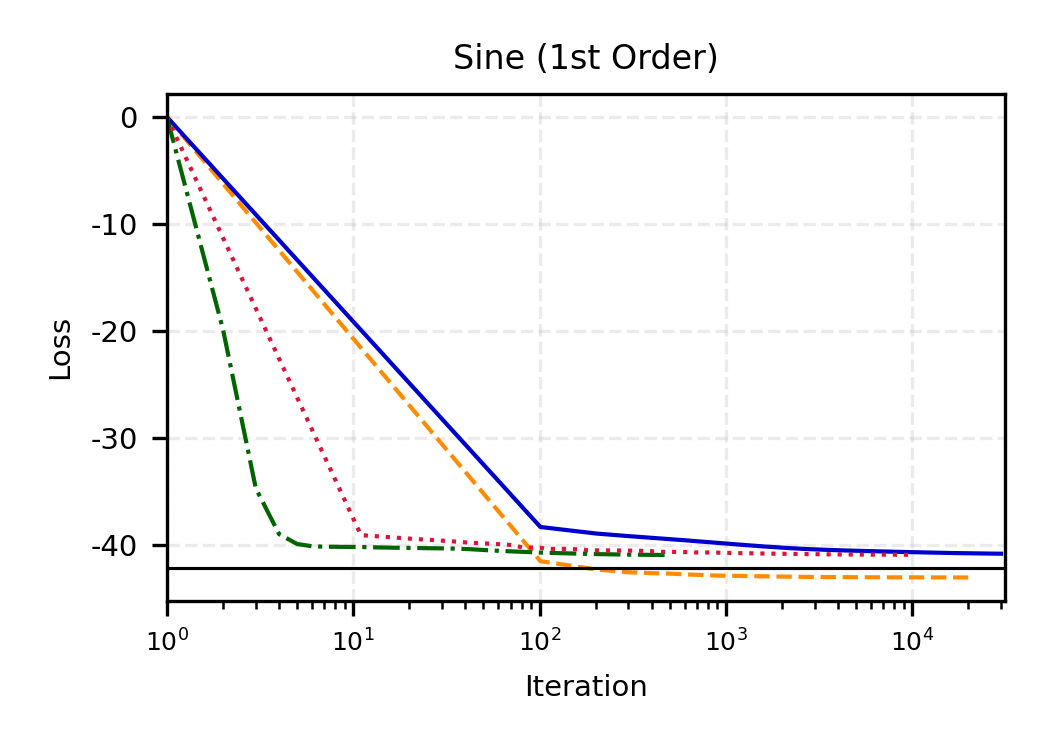}
    
    \caption{Per-iteration loss convergence for optimization algorithms across six DGPs. Panels (row-wise, left to right): (a) TN, (b) GS3, (c) GA3, (d) GS5, (e) Step, (f) Sine.}
    \label{fig:opt-per-iter-loss}
    \end{figure}
    
    \begin{figure}[H]
    \centering
    \includegraphics[width=.49\textwidth]{resources/optimization_algorithms/per_flop/TruncatedNormal_order_2_loss_per_flop.png}
    \hfill
    \includegraphics[width=.49\textwidth]{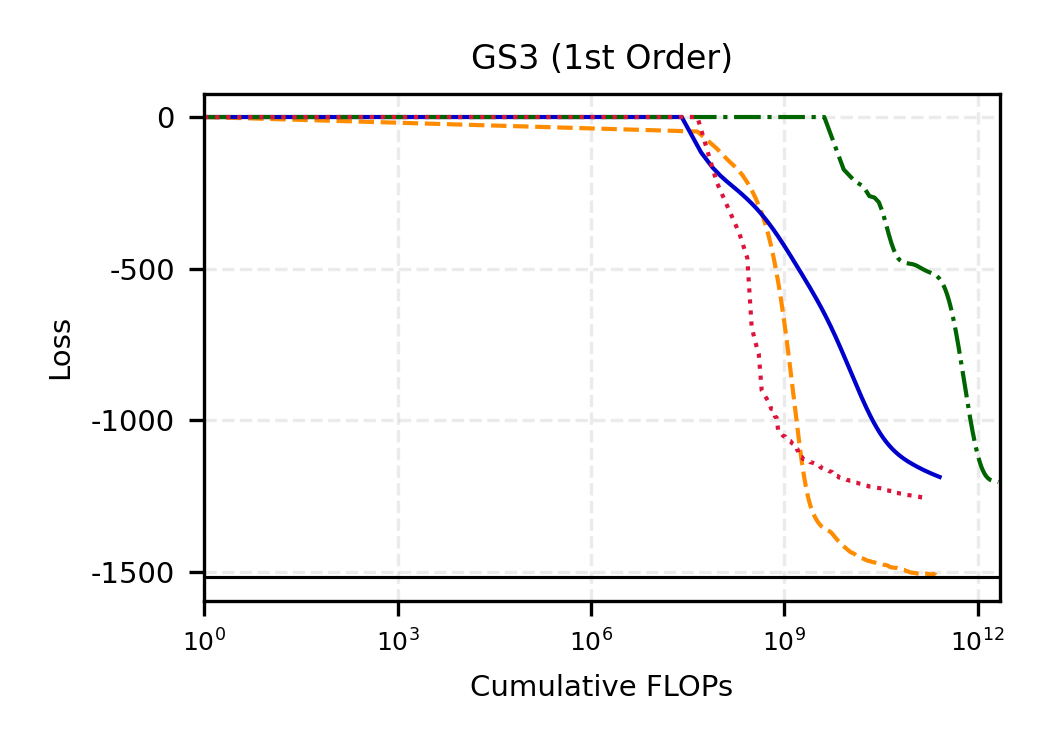}
    
    \includegraphics[width=.49\textwidth]{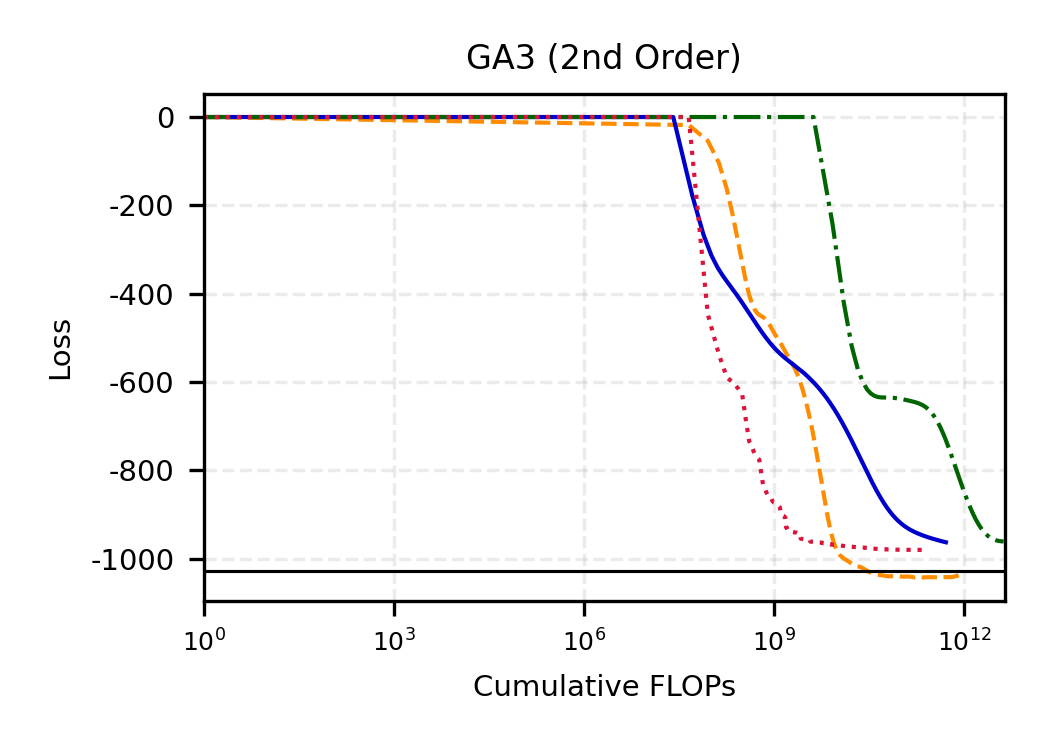}
    \hfill
    \includegraphics[width=.49\textwidth]{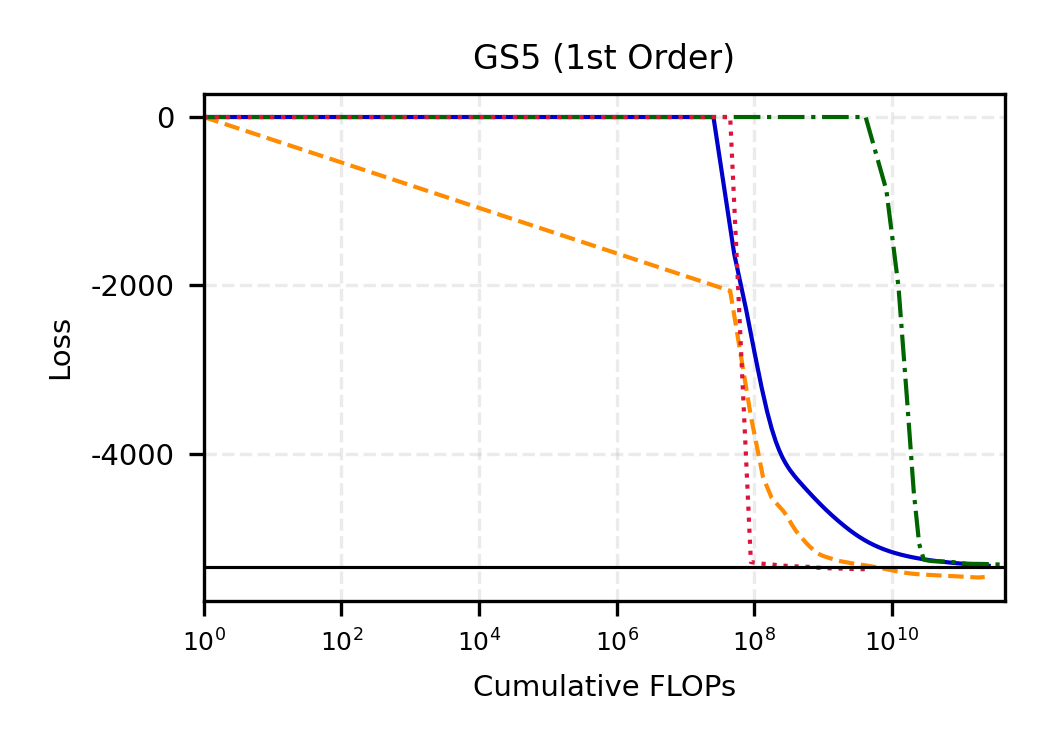}
    
    \includegraphics[width=.49\textwidth]{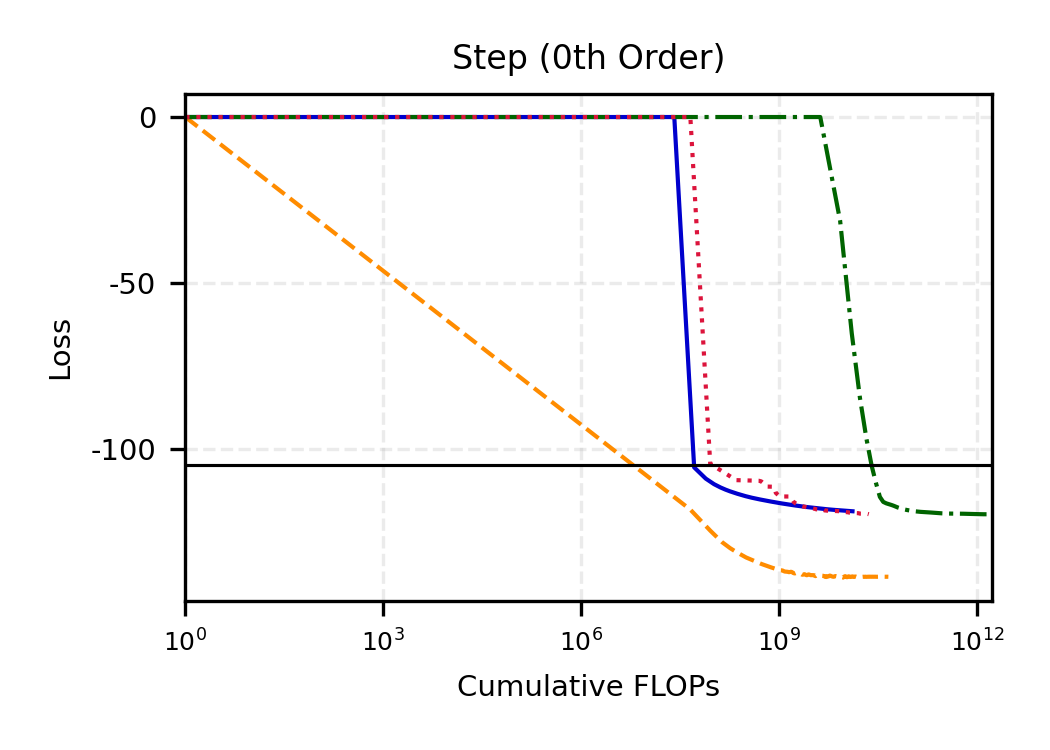}
    \hfill
    \includegraphics[width=.49\textwidth]{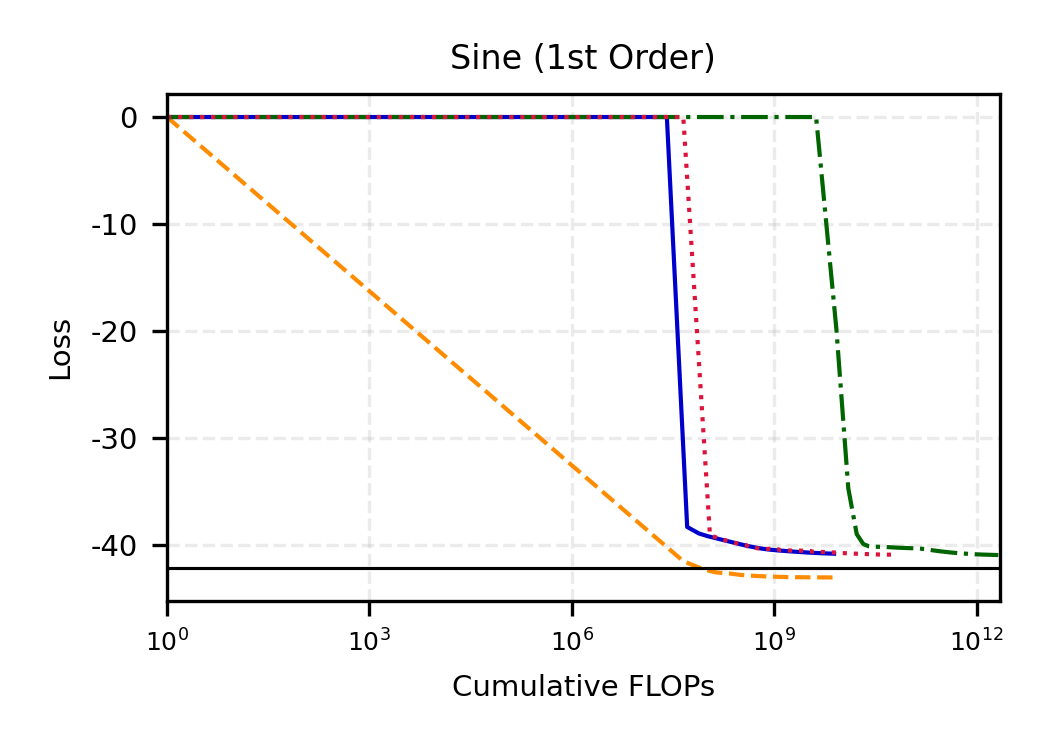}
    
    \caption{Per-FLOP-normalized loss convergence across six DGPs. Panels (row-wise, left to right): (a) TN, (b) GS3, (c) GA3, (d) GS5, (e) Step, (f) Sine.}
    \label{fig:opt-per-flop-loss}
 \end{figure}
 
 \begin{figure}[H]
 \centering
 \includegraphics[width=.49\textwidth]{resources/optimization_algorithms/per_iter/TruncatedNormal_order_2_knot_selection.png}
 \hfill
 \includegraphics[width=.49\textwidth]{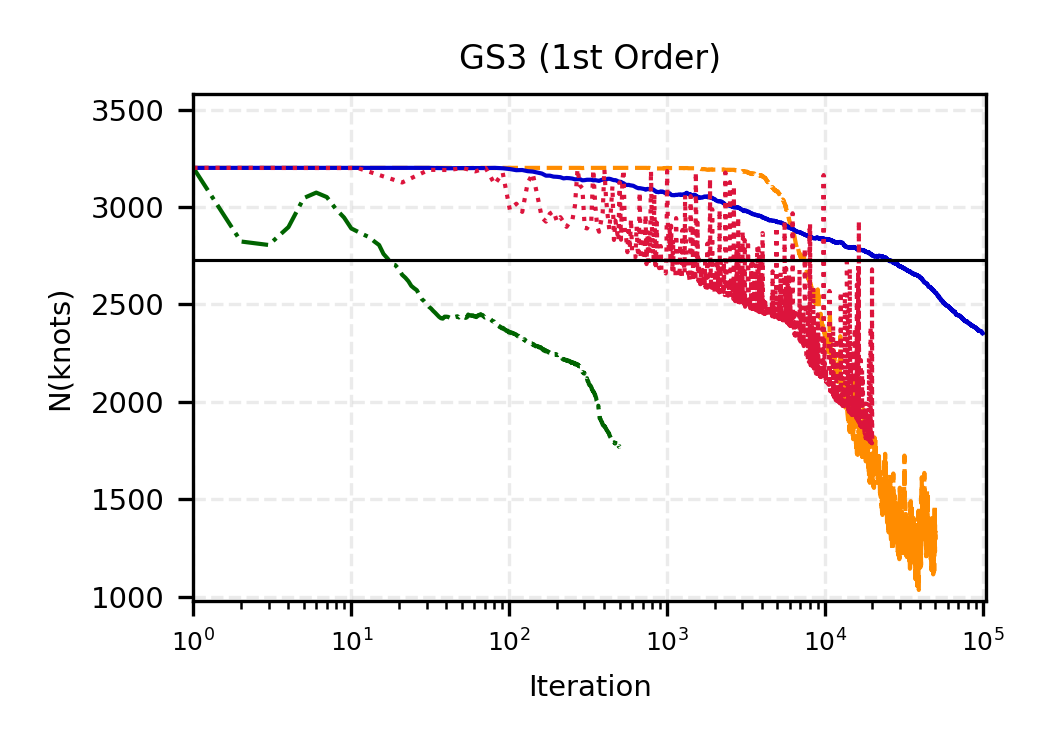}
 
 \includegraphics[width=.49\textwidth]{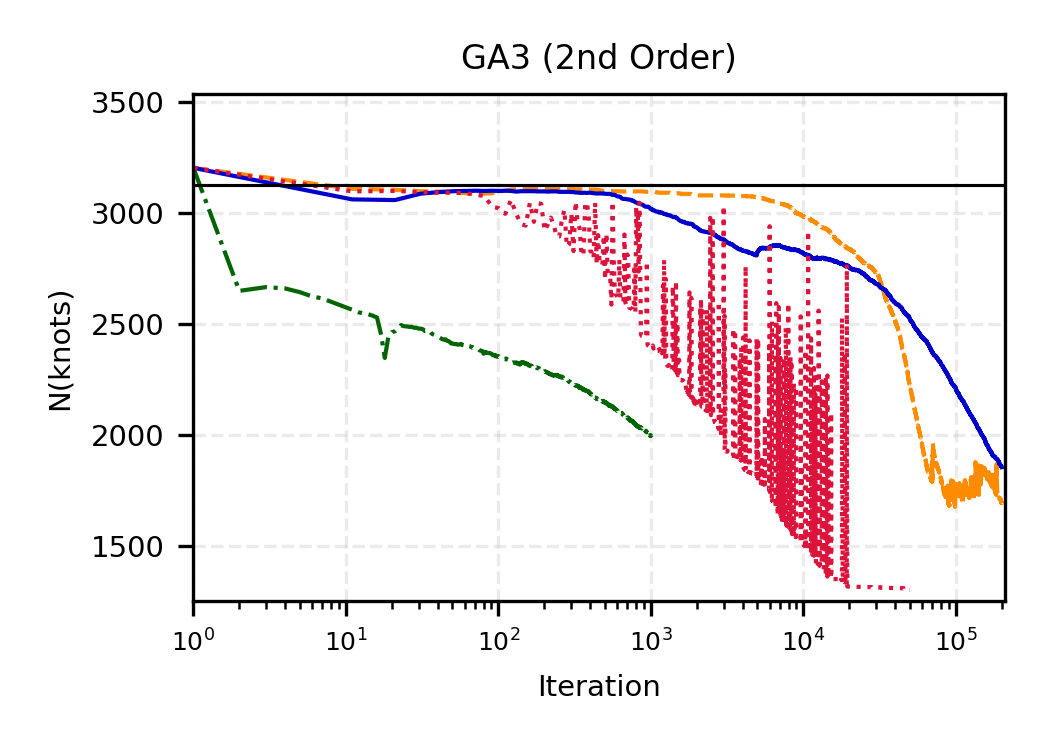}
 \hfill
 \includegraphics[width=.49\textwidth]{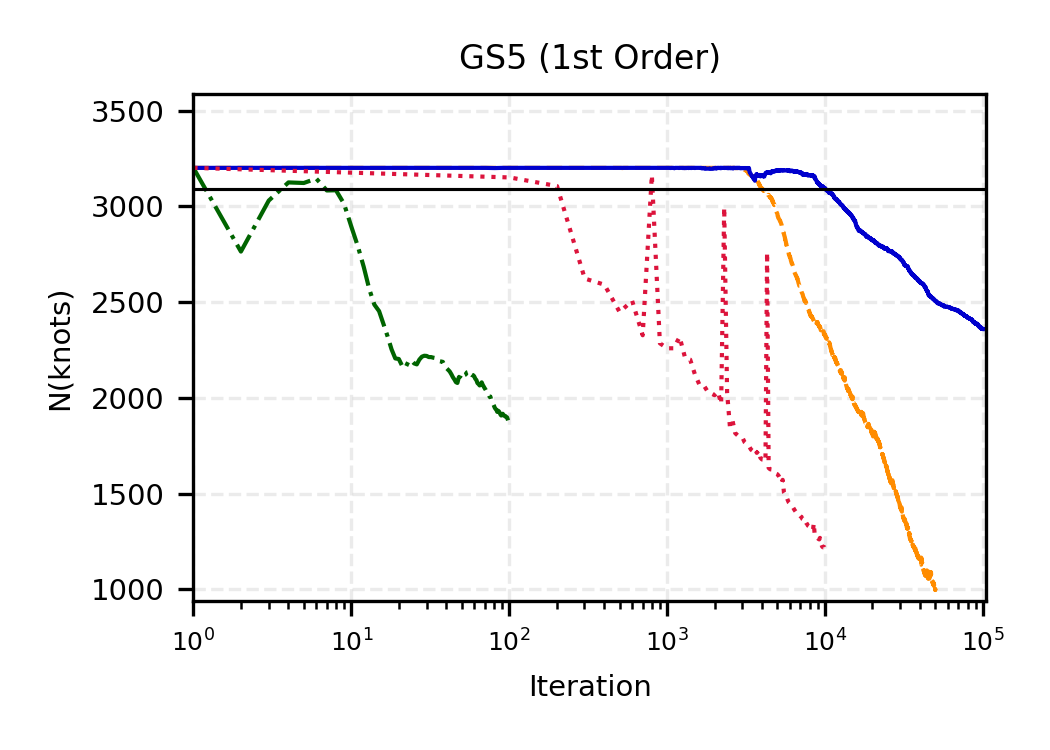}
 
 \includegraphics[width=.49\textwidth]{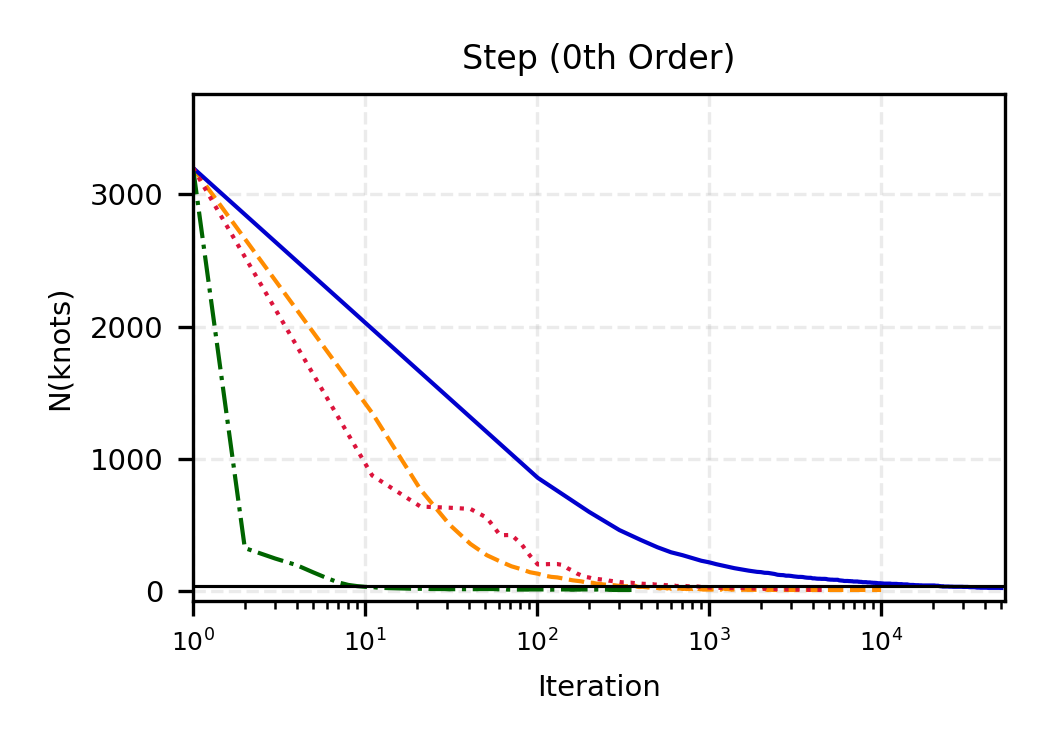}
 \hfill
 \includegraphics[width=.49\textwidth]{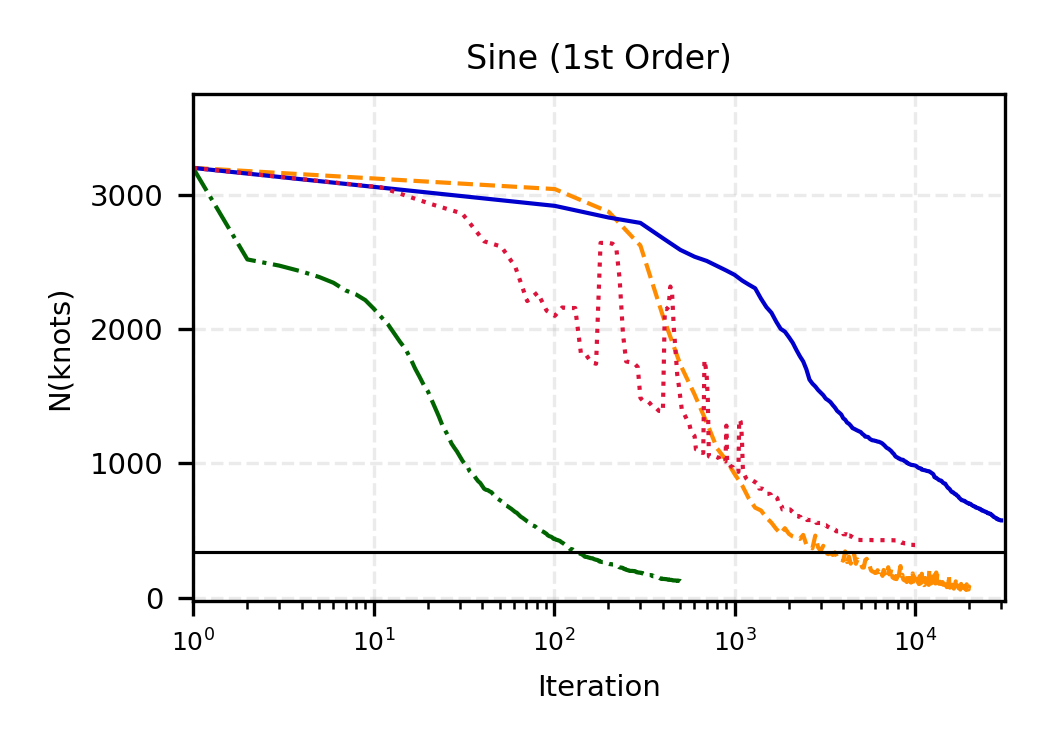}
 
 \end{figure}
 
 \begin{figure}[H]
    \centering
    \includegraphics[width=.9\textwidth]{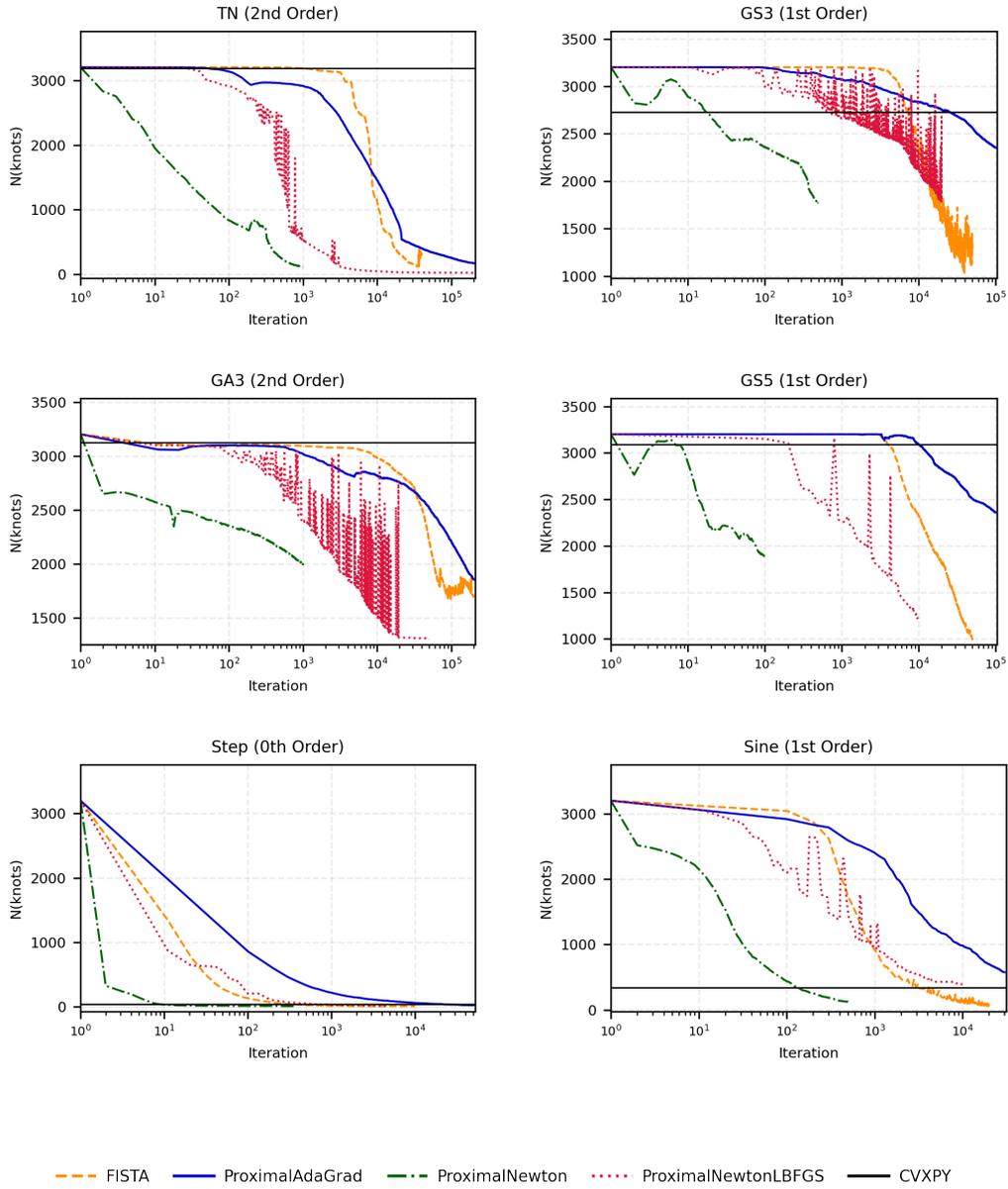}
    \caption{Per-iteration knot selection for optimization algorithms across six DGPs. Panels (row-wise, left to right): (a) TN, (b) GS3, (c) GA3, (d) GS5, (e) Step, (f) Sine.}
 \label{fig:opt-per-iter-knot}
 \end{figure}
 
 \begin{figure}[H]
 \centering
 \includegraphics[width=.49\textwidth]{resources/optimization_algorithms/per_flop/TruncatedNormal_order_2_knot_selection_flops.png}
 \hfill
 \includegraphics[width=.49\textwidth]{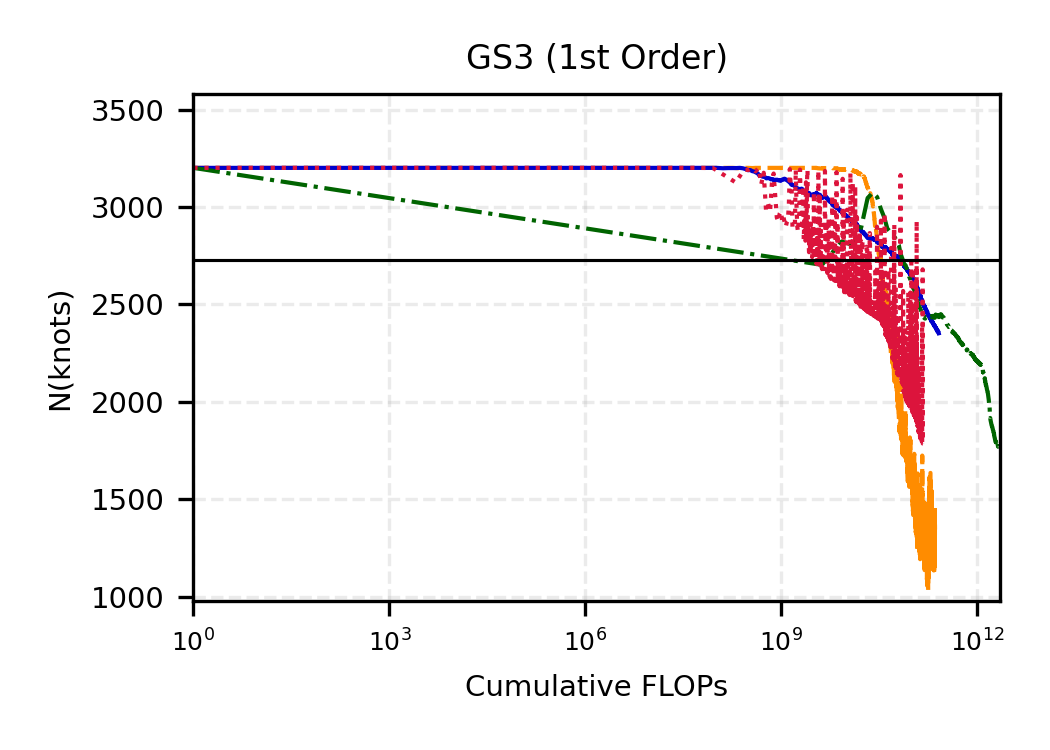}
 
 \includegraphics[width=.49\textwidth]{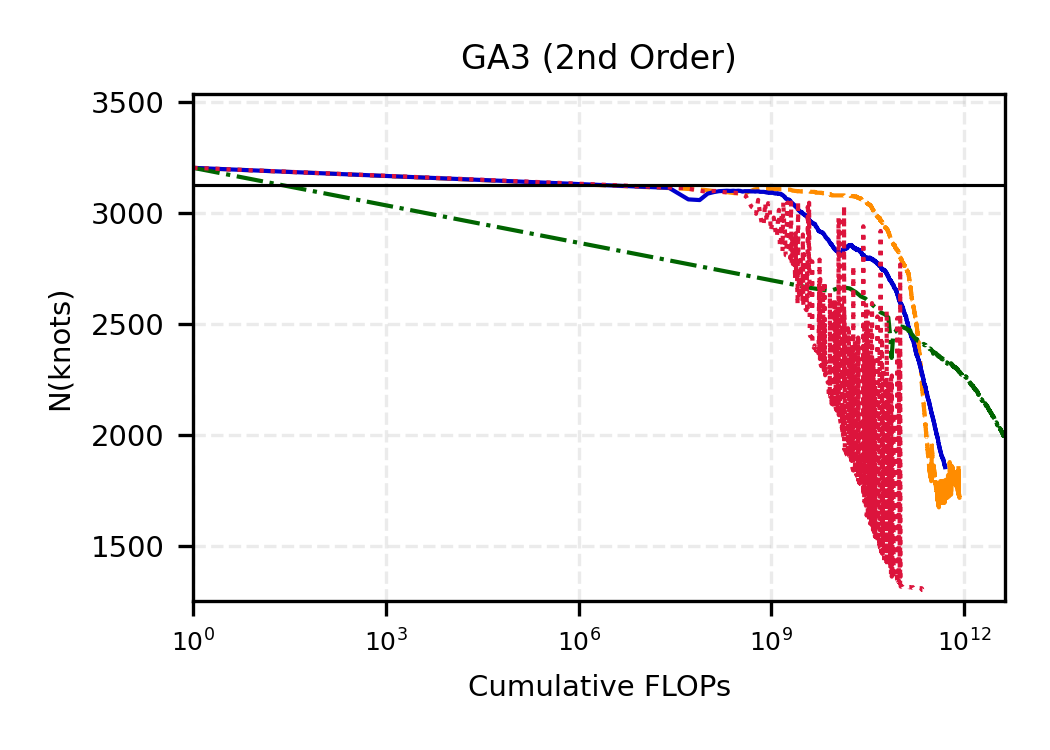}
 \hfill
 \includegraphics[width=.49\textwidth]{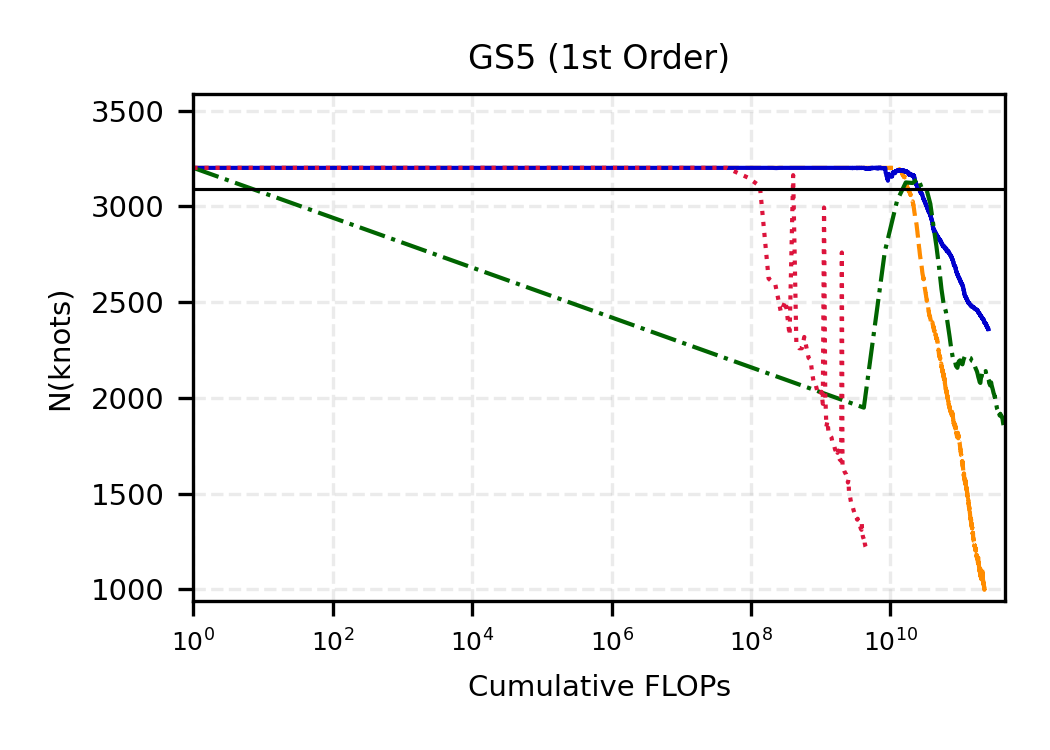}
 
 \includegraphics[width=.49\textwidth]{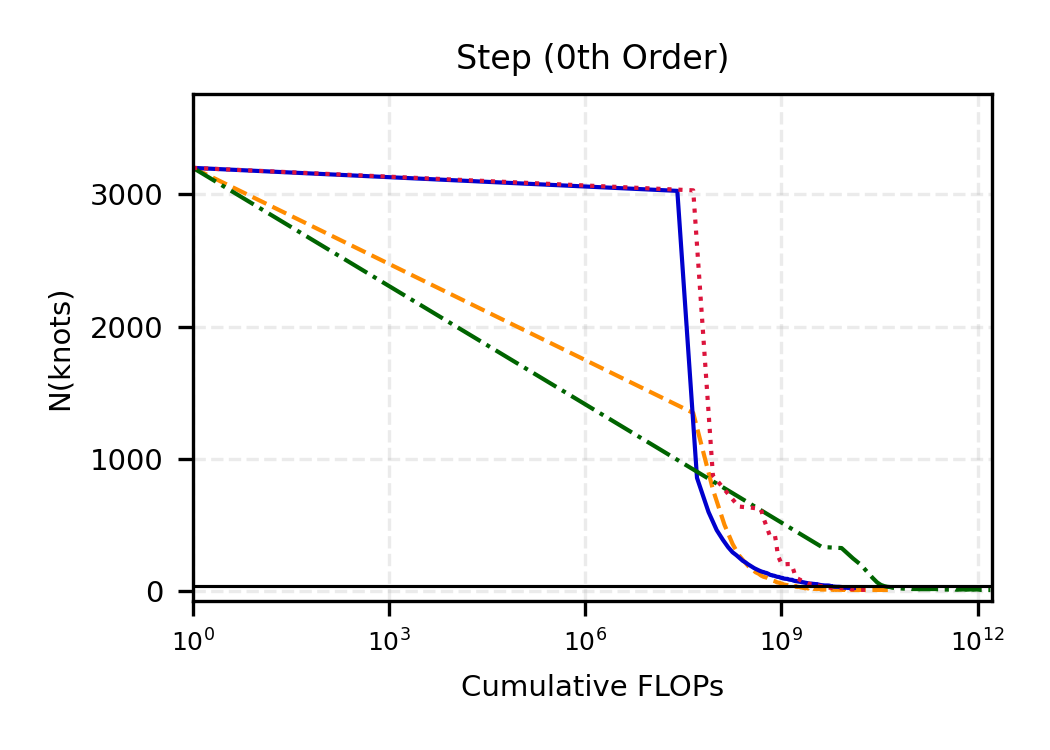}
 \hfill
 \includegraphics[width=.49\textwidth]{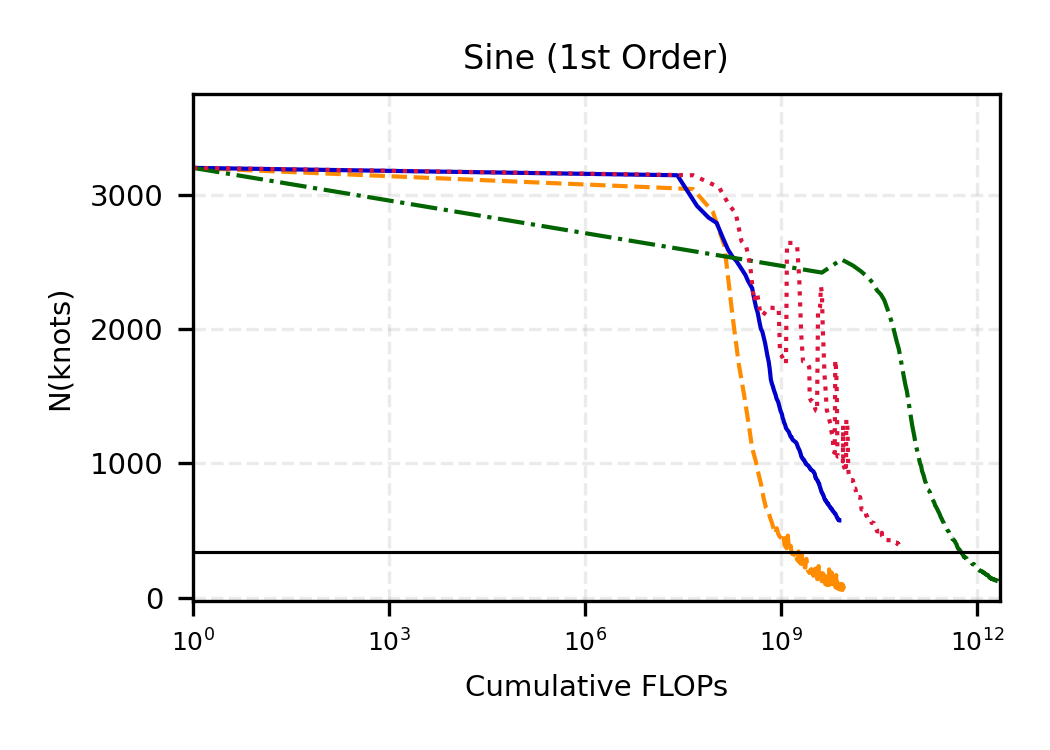}
 \end{figure}
 
 \begin{figure}[H]
    \centering
    \includegraphics[width=.9\textwidth]{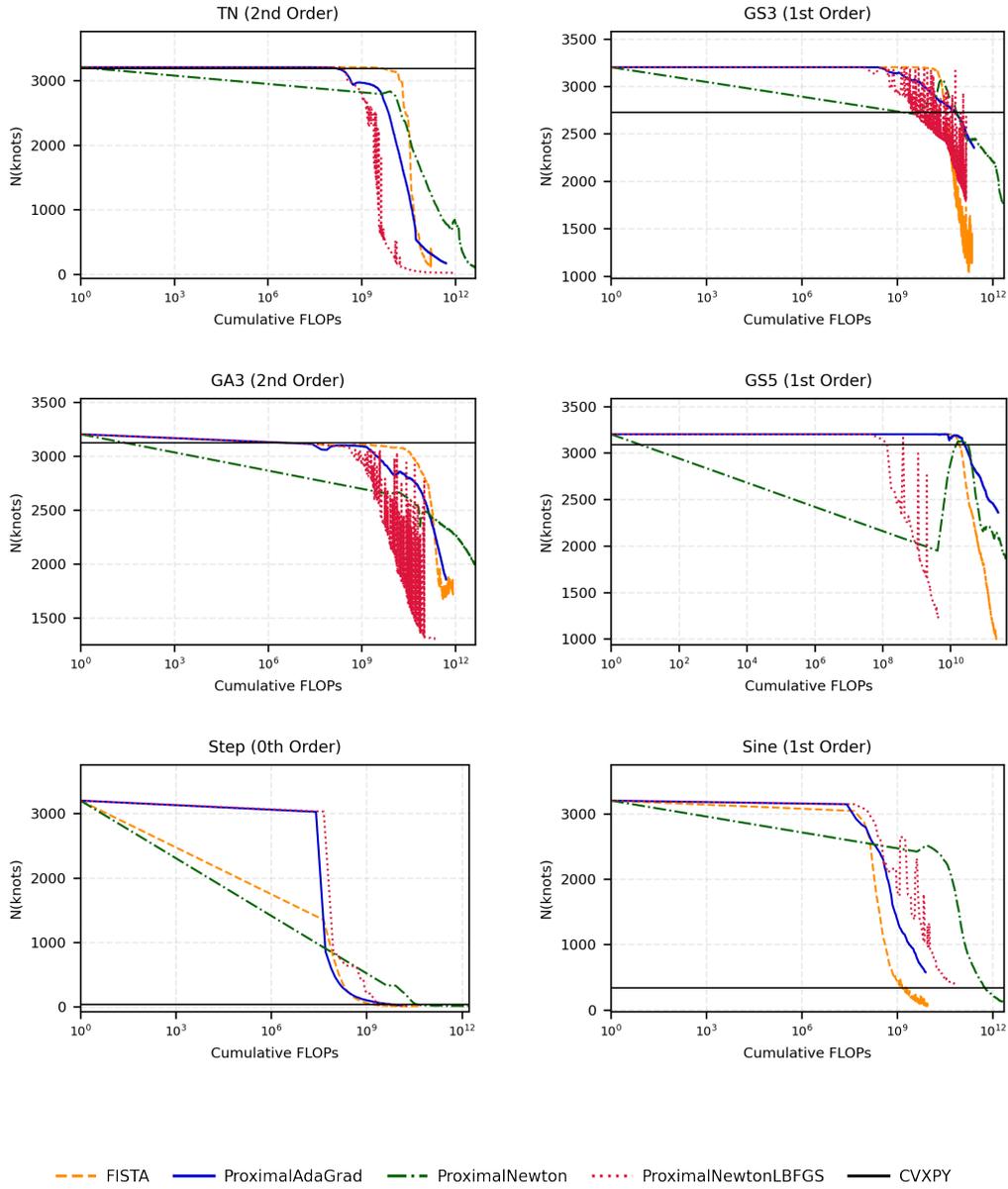}
    \caption{Per-FLOP-normalized knot selection across six DGPs. Panels (row-wise, left to right): (a) TN, (b) GS3, (c) GA3, (d) GS5, (e) Step, (f) Sine.}
    \label{fig:opt-per-flop-knot}
 \end{figure}

\onecolumn

\end{document}